\documentclass[a4paper,12pt]{article}

\usepackage{tikz}
\usepackage{amssymb}
\usetikzlibrary{matrix,arrows,backgrounds}
\usepackage[plainpages=false]{hyperref}
\usepackage{amsfonts,latexsym,rawfonts,amsmath,amssymb,amsthm,mathrsfs}
\usepackage{amsmath,amssymb,amsfonts,latexsym,lscape,rawfonts}

\usepackage[all]{xy}
\usepackage{eufrak}
\usepackage{makeidx}         
\usepackage{graphicx,psfrag}

\usepackage{array,tabularx}

\usepackage{setspace}
\usepackage{geometry}

\newtheorem{thm}{Theorem}[section]
\newtheorem{cor}[thm]{Corollary}
\newtheorem{lem}[thm]{Lemma}
\newtheorem{clm}[thm]{Claim}
\newtheorem{prop}[thm]{Proposition}

\theoremstyle{remark}
\newtheorem{rmk}[thm]{Remark}

\theoremstyle{definition}
\newtheorem{Def}[thm]{Definition}                                        %

\def \R {\mathbb R}

\title{Deformation of  singular connections I:\\
$G_{2}-$instantons with point singularities}

{\small{\author{Yuanqi Wang}}}

\date{\vspace{-5ex}}
\begin{document}
\maketitle 
{\small{\begin{abstract} In dimension 7, we establish a Fredholm theory for a  Dirac-type operator   associated to  a  connection with point singularities. There are two  applications. $1$. over a closed 7-manifold, under some natural conditions,  a  $G_{2}-$instanton and  its point singularities can still be "seen" when the $G_{2}-$structure is properly perturbed. $2$.  over a ball in $\R^{7}$,  for any almost-Euclidean $G_{2}-$structure, there exists  a $G_{2}-$monopole asymptotic to an arbitrary  Hermitian Yang-Mills connection on $S^{6}$. 
\end{abstract}
{\small{\tableofcontents}}}}

\section{Introduction} 
The celebrated work of Donaldson-Thomas \cite{DonaldsonThomas} has inspired extensive studies  on special holonomy.  On a seven-dimensional manifold $M$ with a $G_{2}-$structure $\phi$,  for a vector bundle $E\rightarrow M$, it's suggested in  \cite{DonaldsonThomas} to consider the $G_{2}-$instanton equation of  connections on $E$: 
 
\begin{equation}\label{equ instanton equation for A}
F_{A}\wedge \psi=0, 
\end{equation} 
where $\psi$ is the co associative form uniquely determined by $\phi$, and $F_{A}$ is the curvature form of the connection $A$.
As pointed out in \cite{DonaldsonThomas}, the  $G_{2}-$instantons  should form a basis for a Casson-Floer-type theory for $7-$manifolds.  By adding the torsion-free $G_{2}-$structures $\phi\ (\psi)$ as a "parameter" in (\ref{equ instanton equation for A}),  a very natural moduli space is the set of solutions $(\phi,A)$ to  (\ref{equ instanton equation for A}) (modulo gauge  and some other  natural equivalence).  According to the seminal paper of Tian \cite{Tian},  it's expected that    instantons with point singularities should appear in the natural compactification. Therefore, a fundamental step to understand this moduli space (and any related one)  is to study the following question.
   
 \textit{Given a  $G_{2}-$instanton with point singularities, can we still see the instanton and the singularity for nearby $G_{2}-$structures?}

  The best story we can expect is that  the singularity  disappears by perturbing the $G_{2}-$structure. On mean curvature flows, by the work of Colding-Minicozzi \cite{Colding},  this indeed happens:  any flow which develops singularity as "shrinking donuts" (Angenent \cite{Angenent}) can be perturbed away. However, our following main result shows that this does not  happen very often  for $G_{2}-$instantons. 
\begin{thm}\label{Thm Deforming instanton simple version} Let $E$ be an admissible bundle defined away from finitely-many  points ($O_{j}$) on a $7-$manifold with a $G_{2}-$structure.  Suppose $E$ admits an admissible $G_{2}-$instanton with   trivial co-kernel, then for any small enough admissible deformation of the $G_{2}-$structure, there exists a $G_{2}-$monopole with the same tangent connection at each $O_{j}$.
\end{thm}
\begin{rmk} The  above statement is not  the most precise, but we hope it is easy to understand.    \textbf{The most precise version of Theorem \ref{Thm Deforming instanton simple version} is Theorem \ref{Thm Deforming instanton}, to which we strongly recommend the readers to  pay attention}. We will only prove  Theorem \ref{Thm Deforming instanton}. 
\end{rmk}
\begin{rmk} In the rest of the article we  call the  points where $E$ is   undefined  \textbf{singular points}. Please notice that \textbf{our definition also includes  the case when the bundle is smooth  across some of (or all) the singularities, and in this case we allow both singular and smooth connections}. Nevertheless, we are more interested in the case when the bundle (connection) is truly singular. When the bundle and connection are both smooth across a  singular point, our \textbf{local right inverse} of the linearised operator is
still different from the standard one (see \cite{Gilkey}, \cite{Lawson}).  
\end{rmk}
\begin{rmk} When the $G_{2}-$form is co-closed, any  $G_{2}-$monopole (\ref{equ instantonequation without cokernel})  on a closed manifold is an instanton. However,  a  locally defined one might not.
\end{rmk}

\begin{rmk}\label{rmk on Tian work}  Theorem 1 of Yang \cite{Yang} and Lemma \ref{lem homotopy}  suggest that, via a bundle isomorphism and a  smooth gauge    away from the singularities,  any $G_{2}-$ instanton (on a singular bundle) with  quadratic curvature blowing-up at each singular point   can be reduced to the case in  Theorem \ref{Thm Deforming instanton}. The work of Tian \cite{Tian} indicates that the tangent connections at the singularities are  the cone connections (bundles) on $\R^{7}\setminus O$, pulled back from a smooth Hermitian-Yang-Mills connection over $S^{6}$ (with respect to the standard nearly-K\"ahler structure) via the spherical projection (Remark \ref{rmk homotopy property and def of r}). The work of Charbonneau-Harland \cite{Harland} indicates that  the deformation of these Hermitian connections can be identified with a subspace of the kernel of a Dirac operator. 
\end{rmk}
\begin{rmk} We expect the co-kernel  to be trivial for most singular instantons i.e. we have tranversality in most of the cases. This is  reasonable at least when the inner product is unweighted: the instanton constructed by Walpuski \cite{Walpuski} is rigid, and the self-adjointness implies the co-kernel is trivial.
\end{rmk}

 The key to the deformation problem is a \textbf{Fredholm-theory} for the linearised operator (\ref{equ introduction formula for deformation operator}). 
 It is a Dirac operator i.e. the square of it is  a Laplacian. In the model case, \textbf{though we can not do separation of variable to  the  deformation operator itself, we can do it for the Laplacian}. As the standard Laplacian in polar coordinate, we have a \textbf{polar coordinate formula for the Laplacian of any cone connection} (Lemma \ref{lem  cone formula for laplacian}). Using the Galerkin method (see  7.1.2 in \cite{Evans}),  we can construct a \textbf{local inverse for the Laplacian}, and this gives a \textbf{local inverse for the deformation operator} between the desired weighted-Sobolev spaces. 

  To handle the non-linearity of the instanton equation (or to preserve the tangent cone),  a theory of Sobolev-spaces  is not sufficient. The  spaces should satisfy some multiplicative properties.  Therefore, it should be helpful to  turn on the  a priori Schauder-estimates of  Douglis-Nirenberg (Theorem 1 in \cite{Nirenberg}. Nevertheless,  the essentially difficulty is the $C^{0}-$estimate.
   
   Our crucial observation is that the \textbf{$W^{1,2}_{p,b}-$estimate of sufficiently negative $p$ yields the $C^{0}-$estimate}.  Moreover, to handle the non-linearity of (\ref{equ instanton equation for A}),  it suffices to consider a hybrid space consisting of a weighted-$C^{2,\alpha}$ space  and the weighted Sobolev-space (with norm as the sum of  the two). The global version of our main analytic theorem is:
\begin{thm}\label{Thm Fredholm}
Let  $E\rightarrow  M$  be the same as in Theorem \ref{Thm Deforming instanton}. Suppose $A$ is an admissible  connection of order $4$.  Then for any $A-$generic  negative  $p$ (Definition \ref{Def A generic}) and $b\geq 0$,   $L_{A}$ (see (\ref{equ introduction formula for deformation operator})) is $(p,b)$-Fredholm (Definition \ref{Def Fredholm operators and isomorphisms}) from  $W^{1,2}_{p,b}$ to $L^{2}_{p,b}$  (weighted Sobolev-spaces in Definition \ref{Def global weight and Sobolev spaces}). $L_{A}$ is also $(p,b)$-Fredholm from $H_{p,b}$ to $N_{p,b}$ (Hybrid spaces in Definition \ref{Def Hybrid spaces}).
\end{thm}
\begin{rmk} \textbf{Our Fredholm-theory works for a much larger class of operators, as long as the model operator is cone-type and admits  seperation of variable with respect to some operator on the link}. In particular, it  works  for the Laplace-type  operators (Theorem \ref{thm W22 estimate on 1-forms}). \textbf{We  only assume  discreteness and some natural asymptotics of the spectrum of the  operator on the link}.  When the  eigenfunctions are explicitly known, one might have a  summation formula for the heat kernel and Green's function, thus more information can be extracted  (Theorem 4.3  and 1.13 in  \cite{Myself2013}).
\end{rmk}
\begin{rmk} It does not follow directly from definition that the kernel of the formal adjoint is finite-dimensional, nor equal to the co-kernel. Nevertheless, a trick (in Lemma \ref{lem  boostrap Kstar}) ensures that \textbf{we can decrease the blowing-up rate of the co-kernel  a little bit with respect to the spectrum gaps}. This does not only give us an interesting PDE-result, but also implies \textbf{the co-kernel  is precisely the kernel of the formal adjoint (Theorem \ref{thm characterizing cokernel})}.
\end{rmk}

\begin{rmk} We can't have optimal  Sobolev-estimates unless the weight is properly chosen. Roughly speaking, $A-$generic means the weight $p$ of our Fredholm-theory  avoids some discrete values determined by the spectrum of the tangential operators.  This phenomenon, in other settings, is well understood (see \cite{Donaldson}, \cite{Degeratu}, \cite{Lockhart}).
\end{rmk}
\begin{rmk} Usually the weight in the Schauder-estimate is required to have non-negative power (Lemma 3 in \cite{Nirenberg}). Thus, to obtain Fredholmness for every negative $p$, we should use a different norm   (see (\ref{equ Def local Schauder norm 1}))  when the power is negative, and adopt a trick in \cite{GilbargTrudinger} to avoid  global interpolations (the use of (\ref{equ C20 estimate in the apriori Schauder})).
\end{rmk}
\begin{rmk} \textbf{For Theorem \ref{Thm Deforming instanton} (Theorem \ref{Thm Deforming instanton simple version}), we only need the hybrid theory when $p\in (-\frac{5}{2},-\frac{3}{2})$ and $b=0$}.  The most important usage is to handle the first iteration (\ref{equ first iteration}). Nevertheless, a theory for all $p$ and $b$  is useful for other applications.
\end{rmk}
The local version of our deformation theory states as follows.
\begin{thm}\label{Thm Deforming local instanton} In the setting of Remark \ref{rmk on Tian work}, for any  smooth $SO(m)-$bundle  $E\rightarrow S^{6}$     equipped with a smooth Hermitian Yang-Mills connection $A_{O}$,  there is a $\delta_{0}>0$, such that for  any admissible $\delta_{0}$-deformation $(\underline{\phi}, \underline{\psi})$ over $B_{O}(\frac{1}{2})\subset \R^{7}$  of the Euclidean $G_{2}-$forms (\ref{eqnarray Euc G2 forms}),  there exists a $G_{2}-$monopole of $\underline{\psi}$ (\ref{equ instantonequation without cokernel}) over $B_{O}(\frac{1}{4})$ tangent to $A_{O}$ at the origin $O$. 
 \end{thm}
\begin{rmk} Currently there is only one Hermitian Yang-Mills connection known over $S^{6}$:  the canonical connection (Example 2.2 in \cite{Xu}). \textbf{Theorem \ref{Thm Deforming local instanton} produces concrete local examples of singular $G_{2}-$monopoles tangent to  the canonical connection  for  almost Euclidean $G_{2}-$structures.} 
\end{rmk}

Historically, the Fredholm-problem of elliptic operators has been extensively studied. The most related work to the present article is done by   Lockhart-McOwen \cite{Lockhart}.  They proved that, over non-compact manifolds, a large class of operators  are Fredholm between proper weighted
Sobolev-spaces. Melrose-Mendoza \cite{Melrose} also obtained similar results in the $W^{k,2}-$setting generalized to pseudo-differential operators.  Our hybrid-spaces, though not the most general, are sufficient for this study  and are  \textbf{specially designed for singular connections}. 

   Very recently the author learned from Thomas Walpuski that, using cylindrical method for the deformation operator, and the theory in  \cite{Lockhart},  he could also obtain a local inverse between weighted Schauder-spaces for cones. This local Schauder-estimate is  well illustrated in Section 2.1 of \cite{HaskinsHein}. The author also learned from  Goncalo Oliveira that in \cite{Goncalo}, he obtained $G_{2}-$monopoles with diffferent kind of singularities.

 In the aspect of $G_{2}-$instantons or monopoles,  related work are conducted by   Walpuski \cite{Walpuski}, Sa Earp-Walpuski \cite{SaEarp}, and Oliveira (\cite{Goncalo1}, \cite{Goncalo}). On  monopoles in other settings, see the work of Foscolo \cite{LorenzoFoscolo} and Oliveira (\cite{Goncalo2},\cite{Goncalo3}). In the metric setting, the most related research is done by   Joyce \cite{Joyce} (ALE space), Degeratu-Mazzeo \cite{Degeratu} (Quasi ALE space), Mazzeo \cite{Mazzeo} (Edge-operators), Donaldson \cite{Donaldson11} (conic K\"ahler),  the author-Chen \cite{Myself2013} (parabolic  conic K\"ahler),  Akutagawa-Carron-Mazzeo \cite{ACM} (Yamabe problem on singular spaces). 
 The author believes the above list is not complete, and refers the readers to the references therein.
   
   Omitting  a number of necessary intermediate results, the following diagram  shows the important steps to prove the  main theorems. 
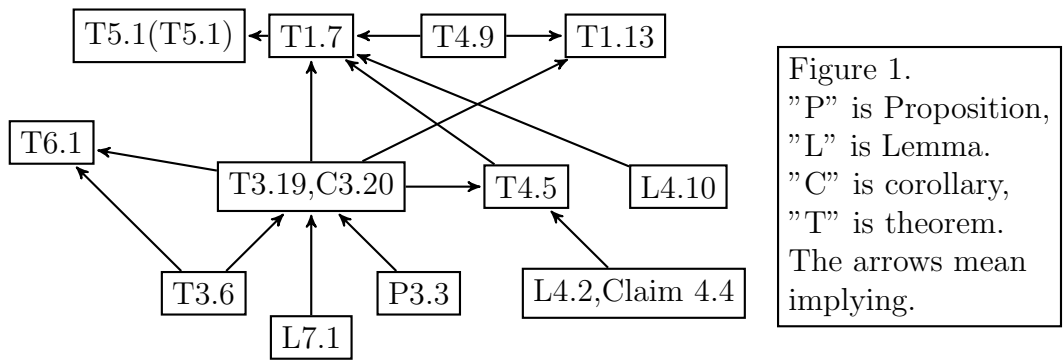
\begin{center}
\begin{tikzpicture}[->,>=stealth',shorten >=1pt,auto,node distance=2cm,
  thick,main node/.style={rectangle,draw}]
 \node[main node] (3)  {T\ref{Thm Fredholm}} ;
   \node[main node] (2) [left of=3] {T\ref{Thm Deforming instanton}(T\ref{Thm Deforming instanton})};

\node[main node] (4) [below   of=3] {T\ref{thm W22 estimate on 1-forms},C\ref{Cor solving model laplacian equation over the ball without the compact support RHS condition}};
  \node[main node] (5) [below left of=2] {T\ref{thm characterizing cokernel}};
\node[main node] (8) [right  of=3] {T\ref{thm C0 est}};
\node[main node] (12) [right  of=8] {T\ref{Thm Deforming local instanton}};
  \node[main node] (6) [below left of=4] {T\ref{thm existence of good solutions to the uniform ODE with estimates}};
   \node[main node] (7) [below right of=4] {P\ref{prop seperation of variable for general cone}};
    \node[main node] (10) [below of=4] {L\ref{lem formula for LA squared}};
    \node[main node] (9) [above right  of=7] {T\ref{thm global apriori L22 estimate}};
    \node[main node] (11) [right of=9] {L\ref{lem compact imbedding}};
     \node[main node] (13) [below right of=9] {L\ref{lem bound on C3 norm of solution to laplace equation when f is smooth and vainishes near O},Claim \ref{clm local regularity L12 for global apriori estimate}};
    \node[draw,align=left] at (8,-2) {Figure 1. \\ "P" is Proposition,\\ "L" is Lemma.\\  "C" is corollary,\\  "T" is theorem. \\  The arrows mean\\ implying.};
    \path[every node/.style={font=\sffamily\small}]
   (9) edge node [left] {} (3)
     (4) edge node [left] {} (9)
  (8) edge node [left] {} (3)
    (10) edge node [left] {} (4)
       (4) edge node [left] {} (3)
    (6) edge node [left] {} (4)
    (6) edge node [left] {} (5)
    (7) edge node [left] {} (4)
     (11) edge node [left] {} (3)
      (4) edge node [left] {} (5)
        (8) edge node [left] {} (12)
        (4) edge node [left] {} (12)
        (13) edge node [left] {} (9)
    (3) edge node [left] {} (2);
  
\end{tikzpicture}

         \end{center}
         
                This article  is organized as following: \textbf{most of the notions and symbols are defined in Section \ref{section Setting up the analysis and notations}}. In Section \ref{section Seperation of variable for  the system in the model case}, we do seperation of variable, and reduce the "squared" model linearized equation to ODEs. In Section \ref{section Solutions  to the  ODEs on the  Fourier-coefficients}, we solve these ODEs. In Section \ref{section Local solutions for the model cone connection}, we  establish the  optimal local Sobolev-theory. In Section \ref{section Global apriori estimate}, \ref{section Hybrid space and C0-estimat},  \ref{section Global Fredholm and Schauder Theory}, we establish the global Fredholm theory of Sobolev and Hybrid spaces. In Section \ref{section Perturbation}, we  prove the main geometric theorems. In Section 
 \ref{section Characterizing the cokernel}, we prove the PDE result and characterize the co-kernel.
 
\textbf{Acknowledgements}: The author would like to thank  Professor Simon Donaldson for suggesting this problem to work on, and for numerous inspiring  conversations. The  author is  grateful to  Song Sun and Thomas Walpuski for many valuable discussions, and for careful reading of the previous versions of this article. The author is grateful to  Alex Waldron, Lorenzo Foscolo, Gao Chen, and Professor Xianzhe Dai for  valuable discussions.

\section{Definitions and  Setting   \label{section Setting up the analysis and notations}} 

We work under the setting of Theorem \ref{Thm Deforming instanton}. By a bundle, we mean a open  cover and associated overlap functions. Two  bundles with different overlap functions are considered to be different, even when they are isomorphic. \textbf{The definitions in this section are all routine and natural, a reader familiar with related material such as \cite{GilbargTrudinger},  \cite{Nirenberg}, \cite{Donaldson} can skip this section and come back if necessary}. 
\begin{Def}\label{Def the bundle xi} A smooth $SO(m)-$bundle $E\rightarrow M\setminus (\cup_{j}O_{j})$ is said to be an \textbf{admissable bundle} if  
\begin{itemize}
\item  $E$ is defined by an admissible cover $U_{\rho_{0}}$ (Definition \ref{Def admissable open cover}) for some $\rho_{0}>0$,
\item  for each singular point $O_{j}$, the overlap function between $V_{+,O_{j}}$ and  $V_{-,O_{j}}$ does not depend on  $r$ (see Remark \ref{rmk homotopy property and def of r}) i.e. the overlap function is pulled back from the sphere. 
\end{itemize}

  Let  $\Xi$ denote  $\Omega^{0}(ad E)\oplus \Omega^{1}(ad E)$ ($adE$-valued $0-$form and $1-$form), and  the corresponding bundle over $S^{n-1}$ as  in Section \ref{section Seperation of variable for  the system in the model case}. \textbf{All the analysis in this article are on sections to} $\Xi$, over $M$ or various domains. \textbf{We omit $\Xi$ in the notations of the section spaces in Definition \ref{Def abbreviation of notations for spaces}}. All the definitions and discussions below apply to $\Xi$ as well. When $E$ is a  complex bundle, we  require it to be a $U(\frac{m}{2})$-bundle, and we still view it as a real bundle. 
\end{Def}
\begin{Def}\label{Def admissable open cover} (Admissible open cover).
Given an (reference) open cover of $M$  and a (reference) coordinate system, a refinement (with the same coordinate maps) denoted as $\mathbb{U}_{\tau_{0}}=\{B_{l},B_{O_{j}} (V_{+,O_{j}},V_{-,O_{j}}),\ l,j \ \textrm{are integers with finite range}\}$
 is called an $\tau_{0}-$admissible cover if the following conditions are satisfied.
 \begin{enumerate}
 \item Each $B_{l}$ is in the smooth part of $E$, the ball $100B_{l}$ (concentric and of radius 100 times larger) is still  away from the singularities and is contained in a coordinate chart. This is different from saying that $B_{l}$ is a metric ball in the manifold. In this article, by abuse of notation, $B_{l}$ means both the ball in the chart and the open set in the manifold (it should be clear from the specific context which notion  we  mean).
 \item Each $B_{O_{j}}$ is centred at a singular point of $E$ with radius $\tau_{0}$, and contain no other singular point. $O_{j}$ corresponds to the origin in the chart. $100B_{O_{j}}$ is still a ball in a coordinate chart and are disjoint from each other.  Moreover, in this coordinate  $\phi(O_{j})$ is the standard $G_{2}-$form. 
 \item $\frac{B_{l}}{100}$ and $\frac{B_{O_{j}}}{100}$  still form a cover of $M$.
 
 \end{enumerate}
 
 When $\tau_{0}$ small enough with respect to $M$ and $E$, this cover always exists if one adds enough balls of small radius. 
 
 The letter $O$ always means a singular point among the $O_{j}$'s, and also  the origin in the coordinate (by abuse of notations).  We denote it as "$B_{O}(\rho)$" when we want a  ball with radius $\rho$. The symbols "$B_{O}$" ("$B_{O_{j}}$") without radius usually  means one of balls in $\mathbb{U}_{\tau_{0}}$ defined above. 
 
 Let $M_{\tau}$  denote $M\setminus \cup_{j}B_{O_{j}}(\tau)$ (the part far away from the singularities).
\end{Def}
\begin{rmk}\label{rmk in practice, we usually consider normal coordinate} In practice, we usually choose the coordinates  as  the normal coordinates of the underline Riemannian metric, though our definition allows any smooth coordinate.  In  \cite{Tian}, the existence of tangent cone connection (near the singularities) is proved in normal  coordinates. 
\end{rmk}
Let $B_{O}(1)$ denote the unit ball in $\R^{n}$ centred at the origin. Since $B_{O}(1)$ admits a natural smooth deformation retraction onto $S^{n-1}\times \{\frac{1}{2}\}$,  the well known homotopy property (Theorem 6.8 and the last paragraph in page 58 of \cite{BottTu}) of  bundles gives the following lemma. 
\begin{lem}\label{lem homotopy}  Any smooth $SO(m)-$bundle $\widehat{E}\rightarrow M\setminus (\cup_{j}O_{j})$ defined by a locally finite cover is isomorphic to an admissible bundle $E$  in Definition \ref{Def the bundle xi}. The isomorphism covers the identity map from  $M\setminus (\cup_{j}O_{j})$ to itself.
\end{lem}
\begin{rmk}\label{rmk homotopy property and def of r} Near each $O_{j}$, for some $\tau_{0}>0$, the smooth isomorphism (away from $O_{j}$) is the one in  Theorem 6.8 of \cite{BottTu}, with respect to the natural homotopy deforming the identity map $id$ of  $B_{O_{j}}(\tau_{0})$ to the map $g\circ f$: $$  \begin{tikzpicture}
  \node at (0,0.4) {g};
\draw[->,semithick] (-1,0.1) -- (1,0.1);
  \node at (-2.2,0) {$S^{n-1}\times (0,\tau_{0})$} ;
  \node  at (-4.5,0) {$B_{O_{j}}(\tau_{0})\simeq$};
 \node  at (2.2,0) {$S^{n-1}\times (\frac{\tau_{0}}{2}),$};
			\draw[<-,semithick] (-1,-0.1) -- (1,-0.1);
			 \node at (0,-0.4) {f} ;
	\end{tikzpicture}  $$
	where $g$ is the \textbf{spherical projection} $(x,t)\rightarrow (x,\frac{\tau_{0}}{2})$, and $f$ is  the identity inclusion. Let $r$ (sometimes $r_{x}$) denote  the Euclidean distance  to the singular set $\{O_{1},...,O_{m_{0}}\}$ in the reference coordinate chart respectively. 
\end{rmk}
\begin{rmk}\label{rmk relation between Euc and Spherical norm} Any bundle-valued $k-$form $\xi$  without $dr-$ component (defined over $\R^{n}\setminus O$) can be viewed as a $r-$dependent  bundle-valued $k-$form over $S^{n-1}(1)$. Let $|\xi|$ denote  the usual Euclidean norm of $\xi$ (as a form over $\R^{n}\setminus O$) , and  $|\xi|_{S}$ denote   the norm on the unit sphere with respect to the standard round metric (as a spherical form). The relation is 
\begin{equation}
|\xi|^{2}=\frac{1}{r^{2k}}|\xi|_{S}^{2},\ \textrm{for any}\ \xi. 
\end{equation}
\end{rmk}
\begin{Def}\label{Def Admissable connection with polynomial or exponential convergence}(Admissible connections) Given a smooth bundle $E\rightarrow M$ with finite many singular points , and a smooth $G_{2}-$structure $(\phi,\psi)$ over $M$, a connection $A$ of $E$ is called an admissible connection  of order $k_{0}$, if it satisfies the following conditions.
\begin{itemize}
\item $A$ is smooth away from the $O_{j}$'s. 
\item There exist a $\mu_{1}>0$, such that  for any $O$ among the  $O_{j}$'s,  there is smooth  connection $A_{O}$ on $E\rightarrow S^{n-1}$  such that  the following holds in the reference coordinate chart. 
  \begin{equation}\label{equ condition on rate of convergence for an admissable instanton}
\Sigma_{j=0}^{k_{0}}r^{j+1}|\nabla^{j}_{A_{O}}(A-A_{O})|\leq C(-\log r)^{-\mu_{1}},
\end{equation}
 where we view $A_{O}$ as the pulled-back connection over  $\R^{7}\setminus O$.
 \end{itemize}

For the purpose of quantization, $A$ is also  said to be  \textbf{of  polynomial rate  $\mu_{1}$ at $O_{j}$} (we omit the $O_{j}$ if the rate holds at every singular point). 

 Suppose for some constant $C$, $A$ satisfies (\ref{equ condition on rate of convergence for an admissable instanton}) with right hand side replaced by $Cr^{\mu_{0}}$ ( at $O_{j}$),  $\mu_{0}>0$, then $A$ is said to be   \textbf{of exponential  rate $\mu_{0}$ at $O_{j}$}.

  \end{Def}
 \begin{rmk} When $A$ is admissible and  satisfies the instanton equation away from the singularities, we call it 
 an \textbf{admissable instanton}. In practice, the coordinate near the singularities are normal coordinate of $g_{\phi}$ (see Remark \ref{rmk in practice, we usually consider normal coordinate}).
 \end{rmk}
  \begin{Def}\label{Def condition SAp} A connection $\underline{A}$ is  satisfies \textbf{Condition ${\circledS_{A,p}}$} if the following holds with respect to the reference instanton $A$.  
\begin{itemize}
\item $\underline{A} $ is an  admissible connection of order $3$.
\item $\underline{A} $  is close  to $A$ in $H_{p}$ (Definition \ref{Def Hybrid spaces} and \ref{Def abbreviation of notations for spaces}). Consequently, 
\item the tangent connections of  $\underline{A}$ at each $O_{j}$ is the same as that of $A$;  
\item   $\underline{A}$ is with the same  polynomial  rate as $A$ at each $O_{j}$. Moreover, if $A$ is with exponential rate $\mu_{0}>0$ at $O_{j}$, then $\underline{A}$ is with  exponential rate $\min\{\mu_{0}, -\frac{3}{2}-p\}$ at the same point.
\end{itemize}
\end{Def}
 Near any singular point $O$ (among the $O_{j}$'s), the bundle $E$ is trivialized by 2 coordinate patches  $U_{+},U_{-}$ of
$S^{n-1}$, then we  choose the cover of $B_{O}$ as $V_{+,O} (V_{-,O})=U_{+}(U_{-})\times [0,\tau_{0}]$. In these coordinates, we can easily define the weighted Schauder norms for sections of $\Xi$  without involving any connection. 
\begin{Def}\label{Def local Schauder norms} As in Definition 2.4  of \cite{Myself2013}, we don't even need a connection to define the Schauder norms. Let $r_{x,y}=\min\{r_{x},r_{y}\}$, $\underline{r_{x,y}}=\max\{r_{x},r_{y}\}$.  Near a  singular point $O$, let $\Gamma$ be a locally defined matrix-valued  tensor in a coordinate chart of $\Xi$ (Definition \ref{Def admissable open cover}), we define the following.
\begin{equation}\label{equ Def local Schauder norm 1}
[\Gamma]^{(\mu,b)}_{\alpha, \mathfrak{U}}= \left\{ \begin{array}{cc}\sup_{x,y\in \mathfrak{U}}(-\log \underline{r_{x,y}})^{b}r^{\mu+\alpha}_{x,y}\frac{|\Gamma(x)-\Gamma(y)|}{|x-y|^{\alpha}},& \textrm{when}\ \mu+\alpha\geq 0 \\
\sup_{x,y\in \mathfrak{U}}(-\log \underline{r_{x,y}})^{b}(\underline{r_{x,y}})^{\mu+\alpha}\frac{|\Gamma(x)-\Gamma(y)|}{|x-y|^{\alpha}},& \textrm{when}\ \mu+\alpha< 0.
\end{array}\right. 
\end{equation}
\begin{equation}\label{equ Def local Schauder norm 2}
[\Gamma]^{(\mu,b)}_{0, \mathfrak{U}}= \sup_{x\in \mathfrak{U}}(-\log r_{x})^{b}r^{\mu}_{x}|\Gamma(x)|.
\end{equation}
The idea of (\ref{equ Def local Schauder norm 1}) is to choose the weight function "as small as possible". Note that we allow $\mu$ to be any real number, while in Lemma 3 in \cite{Nirenberg}, the power is required to be non-negative. We usually let $\mathfrak{U}$ be $V_{+,O}$ ($V_{-,O}$) or a ball contained therein.  We then define  
\begin{equation}\label{equ def top order seminorm in the sector}
|\xi|^{(\gamma,b)}_{2,\alpha, V_{+,O}}\triangleq [\nabla^{2}\xi]^{(2+\gamma,b)}_{\alpha,V_{+,O}}+|\nabla^{2}\xi|^{(2+\gamma,b)}_{0,V_{+,O}}+|\nabla\xi|^{(1+\gamma,b)}_{0,V_{+,O}}+|\xi|^{(\gamma,b)}_{0,V_{+,O}}.
\end{equation}
where the $\nabla$ is just the usual gradient in Euclidean coordinates. Moreover, by abuse of notation (which we adopt through out this article in this case), the  "$\xi$" in (\ref{equ def top order seminorm in the sector})  means the multi-matrix-valued function in $V_{+,O}$ representing $\xi$. $|\xi|^{(\gamma,b)}_{2,\alpha, V_{-,O}}$ is defined in the same way throughout this article, so does $|\xi|^{(\gamma,b)}_{2,\alpha, B}$ for any ball $B\subset V_{+,O}$ or $V_{-,O}$.
\end{Def}
\begin{Def}\label{Def Global Schauder norms}(Global Schauder norms) In the same context as Definition \ref{Def admissable open cover}, let $\rho_{0}>0$ be independent of $A$ such that there exists a  $\rho_{0}-$admissible cover $\mathbb{U}_{\rho_{0}}$.  We define
\begin{equation}\label{equ norm I}|\xi|^{(\gamma,b)}_{2,\alpha,M,I}\triangleq \sup_{B_{l}\in \mathbb{U}_{\rho_{0}}}|\xi|_{2,\alpha,B_{l}}+ \sup_{B_{O_{j}}\in \mathbb{U}_{\rho_{0}}}|\xi|^{(\gamma,b)}_{2,\alpha,V_{+,O_{j}}}+\sup_{B_{O_{j}}\in \mathbb{U}_{\rho_{0}}}|\xi|^{(\gamma,b)}_{2,\alpha,V_{-,O_{j}}}.\end{equation}

The $|\xi|_{2,\alpha,B_{l}}$ are the unweighted Schauder norms defined in (4.5),(4.6) in \cite{GilbargTrudinger}. Actually we have  2 other ways to define the Schauder norms. One is by using the smaller cover:
\begin{equation}\label{equ norm II}
|\xi|^{(\gamma,b)}_{2,\alpha,M,II}\triangleq \sup_{B_{l}\in \mathbb{U}_{\rho_{0}}}|\xi|_{2,\alpha,\frac{B_{l}}{100}}+ \sup_{B_{O_{j}}\in \mathbb{U}_{\rho_{0}}}|\xi|^{(\gamma,b)}_{2,\alpha,\frac{V_{+,O_{j}}}{100}}+\sup_{B_{O_{j}}\in \mathbb{U}_{\rho_{0}}}|\xi|^{(\gamma,b)}_{2,\alpha,\frac{V_{-,O_{j}}}{100}}.
\end{equation}
The third definition is by using the naturally weighted Schauder norms in (4.17) of \cite{GilbargTrudinger} (away from the singularity): 
\begin{equation}\label{equ norm III}
|\xi|^{(\gamma,b)}_{2,\alpha,M,III}\triangleq \sup_{B_{l}\in \mathbb{U}_{\rho_{0}}}|\xi|^{\star}_{2,\alpha,\frac{B_{l}}{50}}+ \sup_{B_{O_{j}}\in \mathbb{U}_{\rho_{0}}}|\xi|^{(\gamma,b)}_{2,\alpha,\frac{V_{+,O_{j}}}{50}}+\sup_{B_{O_{j}}\in \mathbb{U}_{\rho_{0}}}|\xi|^{(\gamma,b)}_{2,\alpha,\frac{V_{-,O_{j}}}{50}}.
\end{equation}
\end{Def}
An easy but important lemma is 
\begin{lem} The 3 norms in (\ref{equ norm I}), (\ref{equ norm II}), (\ref{equ norm III}) are equivalent.
\end{lem}
\begin{proof}This is an easy exercise by definition. For the reader's convenience, we still point out the crucial detail.  Obviously  norm I is stronger than norm III, and norm III is stronger than norm II. 
We only need to show norm II is stronger than norm I. This is because of the last item in Definition \ref{Def admissable open cover}: $V_{+,O_{j}}\setminus \frac{V_{+,O_{j}}}{100}$ is covered by the $\frac{B_{l}}{100}'s$. Since the transition functions are smooth,  then the Schauder norm of $\xi$ over $V_{+,O_{j}}\setminus \frac{V_{+,O_{j}}}{100}$
is controlled by the supreme of Schauder norms on the $\frac{B_{l}}{100}'s$. 

The same holds for $V_{-,O_{j}}\setminus \frac{V_{-,O_{j}}}{100}$ and  $B_{l}\setminus \frac{B_{l}}{100}$ away from the singularities.\end{proof}
\begin{Def}\label{Def Schauder spaces} The \textbf{weighted Schauder-space} $C^{k,\alpha}_{(\gamma,b)}(M)$  consists of sections with the norm (\ref{equ norm I}) being finite. This notation also applies to any domain.
\end{Def}
\begin{Def}\label{Def local Schauder adapted to local perturbation} For the local perturbation in Theorem \ref{Thm Deforming local instanton}, on $B_{O}(R)$, we need a  Schauder space whose  weights near  $O$ and $\partial B_{O}(R)$  are different. To be precise, we define
the space $C^{k,\alpha}_{\{\gamma\},t}[B_{O}(R)]$ by the norm 
\begin{eqnarray*}
|\xi|^{\{\gamma\},t}_{k,\alpha,B_{O}(R)}
 &\triangleq & \Sigma_{j=0}^{k}\sup_{x\in V_{+,O}(R)}\min\{r_{x}^{\gamma+j},(R-r_{x})^{t+j}\}|\nabla^{j}\xi|(x)
\\& & + \sup_{x,y\in V_{+,O}(R)}\min\{r_{x}^{\gamma+k+\alpha},(R-r_{x})^{t+k+\alpha}\}\frac{|\nabla \xi(x)-\nabla \xi(x)|}{|x-y|^{\alpha}}\\& &+\ \textrm{the same in}\ V_{-,O}(R).
\end{eqnarray*}
\end{Def}
\begin{Def}\label{Def deformation of the G2 structure}(Admissible $\delta_{0}-$deformations of $G_{2}-$structures) A $G_{2}-$structure $(\underline{\phi},\underline{\psi})$ is called an admissible $\delta_{0}-$deformation of $\phi$ if
\begin{itemize}
\item $\underline{\phi}$ is smooth and
\begin{equation}\label{equ def of admissable deformations of the G2 structure}
|\underline{\phi}-\phi|_{C^{5}(M)}\leq \delta_{0}.
\end{equation}
where the $C^{5}(M)-$norm is defined by the base $G_{2}-$structure $\phi$;
\item $\underline{\phi}=\phi$ at each $O_{j}$.
\end{itemize} Then  we automatically have 
\begin{equation}
|\underline{\phi}-\phi|(x)\leq Cr_{x}
\end{equation}
when $x$ is close to the singularities.

Since the $\underline{\phi}$ determines a smooth metric $g_{\underline{\phi}}$ , and a smooth co-associative form $\underline{\psi}$ (see \cite{SalamonWalpuski} and \cite{Bryant}),  we also obtain a small deformation of the base form $\psi$ such that 
\begin{equation}
|\underline{\psi}-\psi|_{C^{5}(M)}\leq C\delta_{0}. 
\end{equation} 

 Note that we don't require $\underline{\phi}$ to be closed, but when we want  an instanton, we have to assume it's co-closed i.e.  $\underline{\psi}$ is closed.  
\end{Def}
\begin{Def}\label{Def Dependence of the constants} (General constants)   \textit{The background data in this article is the dimension $n$ (in most cases it's 7), the manifold $M$ and bundle $E$ ($\Xi$) with a fixed coordinate system, the $p,b$ in the weights, the reference $G_{2}-$structures $\phi$ and $\psi$, the tangent cone connections $A_{O_{j}}$ (and the bundle $E$ ($\Xi$) on the sphere), the H\"older-exponent $\alpha$, and the base connection $A$. Without further specification, the  constants  "$C$", $\delta_{0}$, $\epsilon_{0}$, $\mu_{1}$, $\vartheta_{1}$... in each estimate  means a constant depending (at most) on the above data.    We add sub-letters to the "$C$" when it depends on more data than the above, or when we want to emphasize the dependence on some specific factor. The "$C$'s" in different places might be different. The $\delta_{0}$, $\epsilon_{0}$, $\mu_{1}$, etc  are  usually small enough with respect to the above data. There are some auxiliary small numbers like  $\epsilon$, $\delta$, which we usually let tend to $0$.}
      
      When a bound  depends only on the above data, we say it's  uniform. 
\end{Def}
\begin{Def}\label{Def special constants} (Special constants)  \textit{For any $O$ among the singular points,  we let $\bar{C}$ denote any constant depending \textbf{only} on the weights $p,b$, and  the underlying cone connection $A_{O}$ ( and the sub-symbol if there is any).} 
      
      In particular, these $\bar{C}$'s do not depend on the radius of the underlying balls, so  our requirements are fulfilled. They mainly appear in Section \ref{section Solutions  to the  ODEs on the  Fourier-coefficients} and \ref{section Local solutions for the model cone connection}. 
\end{Def}
\begin{Def}\label{Def Fredholm operators and isomorphisms}( $(p,b)-$Fredholm operators and isomorphisms) In the space of $L_{loc}^{2}$-sections to $\Xi\rightarrow M\setminus \cup_{j}O_{j}$, consider the inner product given by the weighted space  $L^{2}_{p,b}$. As in page 49 of \cite{Donaldson}, let $H$ and $N$ be Banach spaces of sections to the bundle $\Xi$, and  $L$ is a bounded linear operator $H\rightarrow N$. $L$ is called a $(p,b)-$Fredholm operator if following conditions are satisfied.  
\begin{itemize}
\item Both $H$ and $N$ are subspaces of $L^{2}_{p,b}$. Let $\perp$ be the orthogonal complement with respect the $L^{2}_{p,b}$-inner product.
\item $Image L$ is closed in $N$. Both $KerL$ and  $coker L=N/ Image L$ are  finite dimensional.
\item   $Coker L$ is isomorphic to $Image^{\perp} L\cap N$, and under this isomorphism,  $N$ admits a direct-sum decomposition $$N=Image  L\oplus_{p,b} coker L,$$  where  $\oplus_{p,b}$ is orthogonal with respect to $L^{2}_{p,b}$. 
\item  $L:\ Ker^{\perp}L\cap H \rightarrow Image  L$ is an isomorphism (under the norms of $H$ and $N$). The "isomorphism" means $L$ is bijective (restricted to the 2 closed subspaces),  and both $L$ and $L^{-1}$ are bounded,
\end{itemize}

\end{Def}
\begin{Def}\label{Def abbreviation of notations for spaces} (Abbreviation of notations for the spaces of sections). When the log-power $b$ is equal to $0$, we abbreviate all the notations $$W^{1,2}_{p,b},L^{2}_{p,b}, H_{p,b},N_{p,b}, J_{p,b}, Q_{p,b,A_{O}},Q_{A,p,b},etc$$ as $$W^{1,2}_{p},L^{2}_{p}, H_{p},N_{p}, J_{p} , Q_{p,A_{O}},Q_{A,p},etc.$$ 
\end{Def}
\begin{Def}\label{Def tensor product} (Tensor products) The sign "$\otimes$" means a tensor product depending on (some of and at most) the reference $G_{2}-$structure $\phi,\psi$, the metric $g_{\psi}$, the Euclidean metric in the coordinates, or some other $G_{2}-$structure, manifold, or bundle.  Thus the norms of these $\otimes$'s are bounded with respect to the above data. The  $\otimes$'s in different places might be different. When we are considering some specific tensor product, we  add sub-letter or symbol to the $\otimes$ (like in Lemma \ref{lem formula for LA squared} and proof of Proposition \ref{prop seperation of variable for general cone}). 
\end{Def}
\begin{Def}\label{Def A generic} Given an admissible connection $A$, let $p$ be a real number. $p$ is called \textbf{$A-$generic} if $1-p$ and $-p$ do not belong to the  $v-$spectrum of any $\Upsilon_{A_{O_{j}}}$ (see (\ref{equ relation between v and beta}) and Definition \ref{Def v spectrum}). This means neither  $7.25-(1-p)^{2}$ nor $7.25-p^{2}$ is an eigenvalue of any $\Upsilon_{A_{O_{j}}}$.
\end{Def}
\section{Local theory}
\subsection{Seperation of variable for  the system in the model case. \label{section Seperation of variable for  the system in the model case}}

By abuse of notation, we still let $\Xi$ denote the space of sections to the bundle $\Xi$ etc. By the monopole equation (\ref{equ instantonequation without cokernel}), the linearised operator with respect to $\sigma\in \Omega^{0}_{adE}$ and $a\in \Omega^{1}_{adE}$ (at $(0,0)\in \Omega^{0}_{adE}\oplus \Omega^{1}_{adE}=\Xi$ when $\underline{\psi}=\psi$) is 
\begin{equation}\label{equ introduction formula for deformation operator}L_{A}[\begin{array}{c}
\sigma   \\
  a   
\end{array}]=[\begin{array}{c}
d_{A}^{\star}a   \\
  d_{A}\sigma+\star(d_{A}a\wedge \psi)  
\end{array}]\end{equation}
where $\psi$ is the base co-associative form in Theorem \ref{Thm Deforming instanton}. Let $L_{A_{O}}$ 
denote the deformation operator of $A_{O}$ and Euclidean $G_{2}-$structure. If the operator depends on any different $G_{2}-$structure than the Euclidean one and $\phi,\psi$, we add sub-symbol. 

 Thus    $L_{A_{O}}^{2}$ is still an operator from $\Xi$ to itself.  To achieve seperation of variable for this  operator, we should understand the bundle  in another way. Working in general dimension $n\geq 4$, given any $\xi=\left|\begin{array}{c}
\sigma  \\
 a\\
 \end{array}\right |\in \Xi$, we write  
 \begin{equation}\label{equ decompose of 1-form to radial and spherical part}
\sigma=\frac{\zeta}{r},\ a=a_{r}\frac{dr}{r}+a_{s},  
\end{equation}
where $a_{s}$ does not have radial component, and $a_{r}$ is a $r-$dependent section  of $\Omega^{0}_{E}(S^{n-1})$. In another word, we want to view sections of $\Xi$ as $r-$dependent sections of the bundle (over $S^{n-1}$)
\begin{equation}\label{equ splitting of xi over the sphere}
\Xi=\Omega^{0}_{adE}(S^{n-1})\oplus \Omega^{0}_{adE}(S^{n-1}) \oplus \Omega^{1}_{adE}(S^{n-1})\ \textrm{under the basis in}\ (\ref{equ decompose of 1-form to radial and spherical part}).
\end{equation}
 
Let  $\nabla_{S}$ denote the covariant derivative with respect to the connection $A_{O}$, viewed as a connection over $S^{n-1}$. 
 
  For the $0-$form  $\sigma$,  the well known cone formula for the rough Laplacian reads as
\begin{equation}
-\nabla^{\star}\nabla \sigma= \frac{\partial^{2}\sigma}{\partial r^{2}} +\frac{n-1}{r}\frac{\partial \sigma}{\partial r}+ \frac{\Delta_{s}\sigma}{r^{2}}.
\end{equation}

Let $\zeta=r\sigma$, by Claim \ref{clm weight change on the ODE}, we have 
\begin{equation}\label{equ cone formula for the 0form with proper basis}
-r\nabla^{\star}\nabla (\frac{\zeta}{r})= \frac{\partial^{2}\zeta}{\partial r^{2}} +\frac{n-3}{r}\frac{\partial \zeta}{\partial r}+\frac{(3-n)\zeta}{r^{2}}+ \frac{\Delta_{s}\zeta}{r^{2}},
\end{equation}
where $\Delta_{s}$ is negative of the rough Laplacian of $A_{O}$ over $S^{n-1}$. 

On $1-$forms, we have the following polar coordinate formula.
\begin{lem}\label{lem  cone formula for laplacian} Suppose $A_{O}$ is a cone connection over $ \R^{n}\setminus O$. Then 
\begin{eqnarray*}
& &-\nabla^{\star}\nabla a\ \ \ \ \ \ \ \ \ \ \ \ \ \ \ \ \ a\ \textrm{is written as in}\ (\ref{equ decompose of 1-form to radial and spherical part}),
\\&=& (\frac{\partial^{2} a_{r}}{\partial r^{2}}+\frac{n-3}{r}\frac{\partial a_{r}}{\partial r}-\frac{2(n-2)a_{r}}{r^{2}}+\frac{\Delta_{s}a_{r}+2d_{s}^{\star}a_{s}}{r^{2}})\frac{dr}{r}
\\& & +\nabla_{r}(\nabla_{r} a_{s})+\frac{n-1}{r}\nabla_{r} a_{s}-\frac{a_{s}}{r^{2}}+\frac{\Delta_{s}a_{s}+2d_{s}a_{r}}{r^{2}}
\end{eqnarray*}
where $\Delta_{s}$  is the negative of the rough laplacian of $A_{O}$ on $S^{n-1}$. 

\end{lem}
\textbf{The  proof of Lemma \ref{lem  cone formula for laplacian} will be deferred to Section \ref{Appendix B: proof of  cone formula for laplacian}.}

 Next we return to dimension 7. By the formula in Lemma \ref{lem formula for LA squared}, $-L_{A_{O}}^{2}$ is the rough Laplacian of $A_{O}$ plus some algebraic operators, thus the  polar coordinate  formula  naturally involves the $SU(3)$-structure of $S^{6}$ (see \cite{FoscoloHaskins}, \cite{Xu}) for the formulas we need. Let $(\omega,\Omega)$ (as in \cite{FoscoloHaskins}) be the standard $SU(3)-$structure over $S^{6}$, where $\omega$ is the standard  Hermitian metric with respect to the almost complex structure, and $\Omega$ is the $(3,0)-$form. They satisfies 
 \begin{equation}\label{equ SU3 structure of S6}
 d\omega=3Re\Omega,\ dIm\Omega=-2\omega^{2}.
 \end{equation}
 Moreover, the standard $G_{2}-$forms can be written as 
 \begin{equation}\label{equ G2 forms in terms of nearly Kahler forms on the link}
 \phi_{0}=r^{2}dr\wedge \omega+r^{3}Re\Omega,\ \psi_{0}=-r^3 dr \wedge Im
 \Omega+ \frac{r^{4}}{2}\omega^{2}.
\end{equation}

A necessary algebraic definition is the following. 
\begin{Def}\label{Def Two tensor products}(Some specific  tensor products) Let $\theta$ be a $p-$form and $\Theta$ be a $q-$form, both are possibly $adE-$valued.

 Suppose $q>p$, then we define $\theta \lrcorner \Theta$ as the $q-p$ form 
 $$\theta \lrcorner \Theta(Y_{1}...Y_{q-p})=\Sigma_{i_{1},...,i_{p}}\theta(v_{i_{1}},...,v_{i_{p}})\Theta(v_{i_{1}},...,v_{i_{p}},Y_{1}...Y_{q-p}),$$ where $v_{j}$'s form an orthogonal basis of the underlying metric. 
 
  Suppose $q<p$, similar to the previous paragraph,  we define $\theta \llcorner \Theta$ as the $p-q$ form  $\theta \llcorner \Theta(Y_{1}...Y_{p-q})=\Sigma_{i_{1},...,i_{q}}\theta(v_{i_{1}},...,v_{i_{q}},Y_{1}...Y_{p-q})\Theta(v_{i_{1}},...,v_{i_{q}})$.
 
 The order of multiplication is important, since they are matrix-valued.

The symbol "$\underline{\otimes}$" means the tensor product   
$$F\underline{\otimes} a=[F(e_{i},e_{j}),a(e_{i})]e^{j},$$  
where $F$ is an $adE-$valued $2-$form,  $a$ is an $1-$form, and the $e_{i}'s$ form an orthonormal frame of the underline metric (which is  the Euclidean metric in the model case).

  The symbol "$\underline{\otimes}_{S}$" means  the $\underline{\otimes}$ over $S^{6}$ with respect to the standard round  metric, so do the symbols $\lrcorner_{S}$ and $\llcorner_{S}$.

\end{Def}
Routine computation gives the main result in this section. 
 \begin{prop}\label{prop seperation of variable for general cone}Under the basis in (\ref{equ decompose of 1-form to radial and spherical part}), the equation 
 \begin{equation}-L_{A_{O}}^{2}(\begin{array}{ccc}
\frac{1}{r}& 0& 0  \\ 
0& \frac{dr}{r} & Id  
\end{array})\left |\begin{array}{c}
\zeta  \\
 a_{r}\\
 a_{s} \end{array}\right |=
(\begin{array}{ccc}
\frac{1}{r}& 0& 0  \\ 
0& \frac{dr}{r} & Id  
\end{array})\left |\begin{array}{c}
\underline{f}_{0}  \\
\underline{f}_{1} \\
\underline{f}_{2}  \end{array}\right |
\end{equation}
is  equivalent to 
\begin{equation}\label{equ cone formula for square of dirac}\left \{\begin{array}{c}
\frac{\partial^{2} \zeta}{\partial r^{2}}+\frac{4}{r}\frac{\partial \zeta}{\partial r}-\frac{5\zeta}{r^{2}}+\frac{\Upsilon_{A_{O},0}}{r^{2}}=\underline{f}_{0}  \\
\frac{\partial^{2} a_{r}}{\partial r^{2}}+\frac{4}{r}\frac{\partial a_{r}}{\partial r}-\frac{5a_{r}}{r^{2}}+\frac{\Upsilon_{A_{O},1}}{r^{2}}=\underline{f}_{1} \\
\nabla_{r}(\nabla_{r} a_{s})+\frac{6}{r}\nabla_{r} a_{s}-\frac{a_{s}}{r^{2}}+\frac{\Upsilon_{A_{O},2}}{r^{2}}=\underline{f}_{2} \end{array} \right.,
\end{equation}
where 
\begin{eqnarray}\label{equ prop cone formula general one}& &
\Upsilon_{A_{O}}\left |\begin{array}{c}
\zeta\\
a_{r}   \\
  a_{s}   
\end{array}\right|=\left |\begin{array}{c}
\Upsilon_{A_{O},0}\\
\Upsilon_{A_{O},1}   \\
\Upsilon_{A_{O},2}   
\end{array}\right|
\\&=&\left |
\begin{array}{c} \Delta_{s}\zeta+\zeta+[F_{A_{O}},a_{s}]\lrcorner_{S}Re\Omega+[F_{A_{O}},a_{r}]\lrcorner_{S}\omega  \\
\Delta_{s}a_{r}-5a_{r}+2d_{s}^{\star}a_{s}+[F_{A_{O}},a_{s}]\lrcorner_{S}Im\Omega-[F_{A_{O}},\zeta]\lrcorner_{S}\omega  \\
 \{\Delta_{s}a_{s}+2d_{s}a_{r}-F_{A_{O}}\underline{\otimes}_{S}a_{s}-[F_{A_{O}},a_{r}]\lrcorner_{S}Im\Omega
 \\+[F_{A_{O}},a_{s}]\lrcorner_{S}\frac{\omega^{2}}{2}-[F_{A_{O}},\zeta]\lrcorner_{S}Re\Omega\}
\end{array}\right| \nonumber
\end{eqnarray}
 The operator $\Upsilon_{A_{O}}$ is a smooth self-adjoint  elliptic operator over $\Xi\rightarrow S^{6}$.
  
 When $A_{O}$ is a $\psi_{0}-$instanton i.e. $F_{A_{O}}\wedge \psi_{0}=0$ as a connection pulled back to $\R^{7}\setminus O$ (equivalent to that $A_{O}$ is Hermitian Yang-Mills  on $S^{6}$ \cite{Xu}), we have  
 \begin{equation}\label{equ prop cone formula special one}
\Upsilon_{A_{O}}\left |\begin{array}{c}
\zeta\\
a_{r}   \\
  a_{s}   
\end{array}\right|=\left |
\begin{array}{c} \Delta_{s}\zeta+\zeta \\
\Delta_{s}a_{r}-5a_{r}+2d_{s}^{\star}a_{s} \\
\Delta_{s}a_{s}+2d_{s}a_{r}-2F_{A_{O}}\underline{\otimes}_{S}a_{s}
\end{array}\right| 
\end{equation}
 \end{prop}
\begin{proof} We only prove (\ref{equ prop cone formula general one}), equation (\ref{equ prop cone formula special one}) is a special case and is implied by Lemma \ref{lem formula for LA squared}, or by   Lemma 2.4 in \cite{Xu} (using (\ref{equ prop cone formula general one})). It suffices to  combine Lemma \ref{lem  cone formula for laplacian}, Lemma \ref{lem formula for LA squared} (and the proof of it), formulas (\ref{equ cone formula for the 0form with proper basis}), (\ref{equ SU3 structure of S6}), (\ref{equ G2 forms in terms of nearly Kahler forms on the link}).  We say a few more for the readers' convenience.
  
  First , since $A_{O}$ is a cone connection, then
\begin{equation}\label{equ formula for F tensor a}F_{A_{O}}\underline{\otimes} a=F_{A_{O}}\underline{\otimes} a_{s}=\frac{1}{r^{2}}F_{A_{O}}\underline{\otimes}_{S} a_{s}.
\end{equation}

   Second, formula (\ref{equ G2 forms in terms of nearly Kahler forms on the link}) directly implies 
   \begin{equation}
   [F_{A_{O}},a]\lrcorner \psi=-\frac{1}{r^{2}} [F_{A_{O}},a_{r}]\lrcorner_{S} Im\Omega+\frac{dr}{r^{3}} [F_{A_{O}},a_{s}]\lrcorner_{S} Im\Omega+\frac{1}{2r^{2}} [F_{A_{O}},a_{s}]\lrcorner_{S} \omega^{2},
   \end{equation}
    \begin{equation}
   \star([F_{A_{O}},a]\wedge\psi)= [F_{A_{O}},a]\lrcorner\phi_{0}=\frac{1}{r^{3}} [F_{A_{O}},a_{s}]\lrcorner_{S}Re\Omega+\frac{1}{r^{3}} [F_{A_{O}},a_{r}]\lrcorner_{S} \omega,\  \textrm{and}
   \end{equation}
\begin{equation}
  \star([F_{A_{O}},\sigma]\wedge \psi)=[F_{A_{O}},\sigma]\lrcorner \phi_{0}=\frac{dr}{r^{2}} [F_{A_{O}},\sigma]\lrcorner_{S} \omega+\frac{1}{r} [F_{A_{O}},\sigma]\lrcorner_{S} Re\Omega.
   \end{equation}
   
   The proof of (\ref{equ prop cone formula general one}) is complete. 
   
   To show the self adjointness  of $\Upsilon_{A_{O}}$ as an operator over $S^{6}$, it suffices to  note that  [$F_{A_{O}},a_{s}]\lrcorner_{S}Re\Omega$ is adjoint to $-[F_{A_{O}},\zeta]\lrcorner_{S}Re\Omega$, $[F_{A_{O}},a_{r}]\lrcorner_{S}\omega$ is adjoint to  $-[F_{A_{O}},\zeta]\lrcorner_{S}\omega$, $2d_{s}a_{r}$ is adjoint to $2d_{s}^{\star}a_{s}$, and $[F_{A_{O}},a_{s}]\lrcorner_{S}Im\Omega$ is adjoint to $-[F_{A_{O}},a_{r}]\lrcorner_{S}Im\Omega$. Moreover, both $F_{A_{O}}\underline{\otimes}_{S}a_{s}$ and $[F_{A_{O}},a_{s}]\lrcorner_{S}\frac{\omega^{2}}{2}$ are self-adjoint. A very import formula for verifying these relations  is 
 \begin{clm}For any $adE-$valued $p-$form $a_{1}$, $adE-$valued $q-$form $a_{2}$, and ordinary $(p+q)-$form $B$, we have 
 \begin{equation}
 <a_{1}\lrcorner B,a_{2}>= (-1)^{pq}<a_{1},a_{2}\lrcorner B>.
 \end{equation}
 \end{clm}
The  assumption that  $\Xi\rightarrow S^{6}$ is a $SO(m)-$bundle  implies $[F_{A},\cdot]$ is anti-symmetric with respect to the inner product of the Lie-algebra of $so(m)$. \end{proof}
We denote the eigenvalues of $\Upsilon_{A_{O}}$ as $\beta$,  and the corresponding eigensection as $\Psi_{\beta}$ (there might be  multiplicities) i.e 
 \begin{displaymath}
\Upsilon_{A_{O}} \Psi_{\beta}=\beta \Psi_{\beta},\  \Psi_{\beta}=(\begin{array}{c}\phi_{0,\beta}  \\
\phi_{1,\beta}  \\
  \phi_{2,\beta}   
\end{array}),\ \phi_{0,\beta},\ \phi_{1,\beta}\in \Omega_{ad E}^{0}(S^{6}),\ \phi_{2,\beta}\in \Omega_{ad E}^{1}(S^{6}).
\end{displaymath}
We require $\Psi_{\beta}$ to be  an orthonormal basis in $L_{\Xi}^{2}(S^{6})$, which is the space of $L^{2}-$sections to $\Xi\rightarrow S^{6}$, with respect to the natural inner product of the direct sums in (\ref{equ splitting of xi over the sphere}). By (\ref{equ cone formula for square of dirac}) and (\ref{equ radial derivative of a spherical form}), the equation 
\begin{equation}\label{equ Def of equation for square of model LAO}-L_{A_{O}}^2 \xi= \underline{f}
\end{equation} is then equivalent to 
\begin{equation}\label{equ raw ODE for 1 forms}
\frac{d^{2} \underline{\xi}_{\beta}}{d r^{2}}+\frac{4}{r}\frac{d \underline{\xi}_{\beta}}{d r}+\frac{(\beta-5)\underline{\xi}_{\beta}}{r^{2}}=\underline{f}_{\beta},\ \xi=\Sigma_{\beta}\underline{\xi}_{\beta}\Psi_{\beta},\ \textrm{where}\ \underline{f}=\Sigma_{\beta}\underline{f}_{\beta}\Psi_{\beta}.
\end{equation} 
Let 
\begin{equation}\label{equ relation between v and beta}
-v^{2}=\beta-\frac{29}{4}.
\end{equation}  $v$ is either a non-negative real number or a purely imaginary number. To reduce the above equation into the form we are most familiar with, we consider $\xi_{v}=r^{\frac{3}{2}}\underline{\xi}_{\beta}$, $f_{v}=r^{\frac{3}{2}}\underline{f}_{\beta}$, then (\ref{equ raw ODE for 1 forms})  becomes 
\begin{equation}\label{equ the ODE of v for 1 forms}
\frac{d^{2} \xi_{v}}{d r^{2}}+\frac{1}{r}\frac{d \xi_{v}}{d r}-v^{2}\frac{\xi_{v}}{r^{2}}=f_{v},\ \textrm{where}\ \xi_{v}\ \textrm{and}\ f_{v}\ \textrm{only depend on}\ r.
\end{equation}

 By abuse of notation,  we shall study the ordinary differential equation
\begin{equation}\label{equ the standard ODE}
\frac{d^{2} u}{d r^{2}}+\frac{1}{r}\frac{d u}{d r}-v^{2}\frac{u}{r^{2}}=f.
\end{equation}
 \begin{Def}\label{Def v spectrum}($v-$spectrum) Since the  $v$'s are determined by $\beta$ via (\ref{equ relation between v and beta}), we call them  $v-$spectrum of the tangential operators. By abuse of notation, we write $\Psi_{\beta}$ as $\Psi_{v}$, $\underline{f}_{\beta}$ as $\underline{f}_{v}$, $\underline{\xi}_{\beta}$ as $\underline{\xi}_{v}$ etc.
\end{Def}
\subsection{Solutions  to the  ODEs on the  Fourier-coefficients \label{section Solutions  to the  ODEs on the  Fourier-coefficients}}

In this section, for any singular point $O$ of the connection and bundle, we shall solve  (\ref{equ Def of equation for square of model LAO})  locally for the cone connection $A_{O}$. This is equivalent to solving the ODEs (\ref{equ the standard ODE}) of the Fourier-coefficients.  By proving  Theorem \ref{thm existence of good solutions to the uniform ODE with estimates}, we show the existence of good solutions to (\ref{equ the standard ODE}) with the correct  and optimal $L^{2}-$estimates. \textbf{Choosing different formulas for different spectrum,   the solution for each $v$ are given by (\ref{equ solution when v is real and  p>1-v is integral from 1 to r}), (\ref{equ solution formula when v positive and p less than 1-v}), (\ref{equ solution when v is purely imaginary}),    (\ref{equ solution when v=0}).} These solutions possess the properties for building up a deformation theory for singular connections. 

We can not prove Theorem \ref{thm existence of good solutions to the uniform ODE with estimates} only by "potential" estimates. To be precise, when $v=0$,  Proposition \ref{prop existence of solution with lowest order estimate}, a result by potential estimate, is not the optimal estimate we want in Theorem \ref{thm existence of good solutions to the uniform ODE with estimates}. Nevertheless, the interesting thing is that the 2 terms in     (\ref{equ solution when v=0}) actually enjoy some magic cancellation. 
This allows us  to use integral identity to improve the non-optimal estimate in Proposition \ref{prop existence of solution with lowest order estimate} to optimal estimate in Propositions (\ref{prop good solution with lowest order estimate when v is nonzero}). Since this technique involves integration by parts, we need to choose the  solution   properly with respect to the  weight, so that the boundary terms in Lemma \ref{lem rough L2 identity integration by parts } vanish, so   we have  the identity (\ref{equ strong L2 identity integration by parts}).

 This is similar to Theorem 9.9 of \cite{GilbargTrudinger}: though  the weak $L^{2,1}-$estimate can be done by Calderon-Zygmund potential estimate, the $W^{2,2}-$estimate  still requires integration by parts. 
\begin{thm}\label{thm existence of good solutions to the uniform ODE with estimates}Suppose  $f$ is supported in $(0,\frac{1}{100}]$ and vanishes near $r=0$.   Suppose $p<0$ is $A_{O}-$ generic and $b\geq 0$. Then for any    $v$ among the $v-$spectrum of  $\Upsilon_{A_{O}}$ (see Definition \ref{Def v spectrum}),  there exists a solution $u$ to (\ref{equ the standard ODE}) with the following uniform estimate. 
\begin{eqnarray}\label{equ in thm of ODE optimal estimate v neq 1-p} \int^{\frac{1}{4}}_{0} u^{2} (-\log r)^{2b}r^{2p-3} dr
\leq  \bar{C} \int^{\frac{1}{2}}_{0}f^{2}r^{2p+1}(-\log r)^{2b} dr. 
\end{eqnarray}
$\bar{C}$ is as in Definition \ref{Def special constants}. 

\end{thm}
\begin{rmk}Suppose $p$ is not $A_{O}-$generic i.e.  there is some $v$ such that $v=1-p$,   there is  a $f$ which violates the conclusion of Theorem \ref{thm existence of good solutions to the uniform ODE with estimates}. Namely,  let $v=\frac{5}{2}$, $p=-\frac{3}{2}$, and 
$f=r^{\frac{1}{2}}(-\log r)^{-b-1-\epsilon},\ b>0, \frac{1}{100}>\epsilon>0$, then $\int^{\frac{1}{2}}_{0}f^{2}r^{-2}(-\log r)^{2b}dr<\infty$. However,   there is no solution $u$ to 
(\ref{equ the standard ODE}) such that $\int^{\frac{1}{2}}_{0}u^{2}r^{-6}(-\log r)^{2b}dr<\infty$. Using a limiting argument, this means we can't find solution which satisfies the optimal bound in (\ref{equ in thm of ODE optimal estimate v neq 1-p}).
\end{rmk}
\begin{proof}[Proof of Theorem \ref{thm existence of good solutions to the uniform ODE with estimates}:] 
 It's a combination of Proposition \ref{prop good solution with lowest order estimate when v is nonzero},  \ref{prop good solution with lowest order estimate when v is imaginary}, and  \ref{prop final ODE W22 estimate when v=0}.\end{proof}
\begin{prop}\label{prop good solution with lowest order estimate when v is nonzero} Under the same conditions in Theorem \ref{thm existence of good solutions to the uniform ODE with estimates}, suppose   $v>0$,    then there exists a solution $u$ to (\ref{equ the standard ODE}) such that
$$\int^{\frac{1}{2}}_{0}u^{2}r^{2p-3}(-\log r)^{2b}dr \leq \frac{\bar{C}}{1+|v|^{3}}\int^{\frac{1}{2}}_{0}f^{2}(x)x^{2p+1}(-\log x)^{2b}  dx.$$

\end{prop}
\begin{prop}\label{prop good solution with lowest order estimate when v is imaginary} Under the same conditions in Theorem \ref{thm existence of good solutions to the uniform ODE with estimates}, suppose    $v$ is purely imaginary, then there exists a solution $u$ to (\ref{equ the standard ODE}) such that
$$\int^{\frac{1}{2}}_{0}u^{2}r^{2p-3}(-\log r)^{2b}dr \leq \bar{C}\int^{\frac{1}{2}}_{0}f^{2}(x)x^{2p+1}(-\log x)^{2b}  dx.$$
\end{prop}
\begin{prop}\label{prop existence of solution with lowest order estimate}  Under the same conditions in Theorem \ref{thm existence of good solutions to the uniform ODE with estimates}, suppose   $v=0$, then there exists a solution $u$ to (\ref{equ the standard ODE}) such that
$$\int^{\frac{1}{2}}_{0}u^{2}r^{2p-3}(-\log r)^{2b-2}dr\leq \bar{C}\int^{\frac{1}{2}}_{0}f^{2}x^{2p+1}(-\log x)^{2b}  dx.$$
\end{prop}
\begin{proof}[Proof of Proposition \ref{prop good solution with lowest order estimate when v is nonzero}:]
  Case 1:  $v>1-p$.  We choose a  solution to (\ref{equ the standard ODE}) as 
\begin{equation} \label{equ solution when v is real and  p>1-v is integral from 1 to r}
u=\frac{1}{-2v}\{r^{v}\int_{r}^{\frac{1}{2}}x^{-v+1}f(x)dx+r^{-v}\int_{0}^{r}x^{v+1}f(x)dx\}.
\end{equation}
We denote $u_{I}=\frac{r^{v}}{-2v}\int_{r}^{\frac{1}{2}}x^{-v+1}f(x)dx$, $u_{II}=\frac{r^{-v}}{-2v}\int_{0}^{r}x^{v+1}f(x)dx$.

There exists a  $q$ such that $p>q>1-v$. By H\"older inequality, 
\begin{eqnarray*}
& & u^{2}_{I}
\\&\leq &\frac{\bar{C}r^{2v}}{1+|v|^{2}}[\int^{\frac{1}{2}}_{r}x^{-2v+2}x^{-2q-1}(-\log r)^{-2b}dx][\int^{\frac{1}{2}}_{r}f^{2}(x)x^{2q+1}(-\log r)^{2b}dx]
\\&\leq & \frac{\bar{C}}{(1+|v|^{2})|v|}r^{2v}r^{-2v-2q+2}(-\log r)^{-2b}[\int^{\frac{1}{2}}_{r}f^{2}(x)x^{2q+1}(-\log r)^{2b}dx],\ q>1-v.
\end{eqnarray*}
\begin{eqnarray*}
\textrm{Thus}& & \int^{\frac{1}{2}}_{0}u^{2}_{I}r^{2p-3}(-\log r)^{2b}dr
\\&\leq & \frac{\bar{C}}{(1+|v|^{2})^{\frac{3}{2}}}\int^{\frac{1}{2}}_{0} r^{2p-2q-1}dr[\int^{\frac{1}{2}}_{r}f^{2}(x)x^{2q+1}(-\log x)^{2b}dx]
\\&= & \frac{\bar{C}}{(1+|v|^{2})^{\frac{3}{2}}}\int^{\frac{1}{2}}_{0}f^{2}(x)x^{2q+1}(-\log x)^{2b}dx\int^{x}_{0}r^{2p-2q-1}dr,\ p>q
\\&= & \frac{\bar{C}}{(1+|v|^{2})^{\frac{3}{2}}}\int^{\frac{1}{2}}_{0}f^{2}(x)x^{2q+1}(-\log x)^{2b} x^{2p-2q} dx,\ p>q.
\\&= & \frac{\bar{C}}{(1+|v|^{2})^{\frac{3}{2}}}\int^{\frac{1}{2}}_{0}f^{2}(x)x^{2p+1}(-\log x)^{2b}  dx.
\end{eqnarray*}
Since $p<0$,  we directly use  (\ref{eqnarray Prop ODE solution v>0 1}) and (\ref{eqnarray Prop ODE solution v>0 2})   replacing the "$q$" there by $0$, "$v$" by $-v$. Hence
\begin{equation}
\int^{\frac{1}{2}}_{0}u^{2}_{II}r^{2p-3}(-\log r)^{2b}dr
\leq \frac{\bar{C}}{(1+|v|^{2})^{\frac{3}{2}}}\int^{\frac{1}{2}}_{0}f^{2}(x)x^{2p+1}(-\log x)^{2b}dx.
\end{equation}
The proof is complete when $v>1-p$.

Case 2: $0<v<1-p$. In this case we take the solution as 
\begin{equation}\label{equ solution formula when v positive and p less than 1-v}
u=\frac{1}{-2v}\{-r^{v}\int_{0}^{r}x^{-v+1}f(x)dx+r^{-v}\int_{0}^{r}x^{v+1}f(x)dx\}.
\end{equation}
Let $u_{I}$ in this case be $=\frac{r^{v}}{2v}\int_{0}^{r}x^{-v+1}f(x)dx$.    Choose $q$ such that $p<q<1-v$, by H\"older inequality and Lemma \ref{lem log estimate}, we calculate
\begin{eqnarray}\label{eqnarray Prop ODE solution v>0 1}
& & u^{2}_{I}
\\&\leq &r^{2v}[\int^{r}_{0}x^{-2v+2}x^{-2q-1}(-\log x)^{2b}dx][\int^{r}_{0}f^{2}(x)x^{2q+1}(-\log x)^{2b}dx]\nonumber
\\&\leq & \bar{C}r^{-2q+2}(-\log r)^{-2b}[\int^{r}_{0}f^{2}(x)x^{2q+1}(-\log x)^{2b}dx],\ q<1-v.\nonumber
\end{eqnarray}
\begin{eqnarray}\label{eqnarray Prop ODE solution v>0 2}
\textrm{Thus}& & \int^{\frac{1}{2}}_{0}u^{2}_{I}r^{2p-3}(-\log r)^{2b}dr
\\&\leq & \bar{C}\int^{\frac{1}{2}}_{0} r^{2p-2q-1}dr[\int^{r}_{0}f^{2}(x)x^{2q+1}(-\log x)^{2b}dx]\nonumber
\\&= & \bar{C}\int^{\frac{1}{2}}_{0}f^{2}(x)x^{2q+1}(-\log x)^{2b}dx\int^{\frac{1}{2}}_{x}r^{2p-2q-1}dr,\ p<q\nonumber
\\&= & \bar{C}\int^{\frac{1}{2}}_{0}f^{2}(x)x^{2q+1}(-\log x)^{2b} x^{2p-2q} dx,\ p<q.\nonumber
\\&= & \bar{C}\int^{\frac{1}{2}}_{0}f^{2}(x)x^{2p+1}(-\log x)^{2b}  dx.\nonumber
\end{eqnarray}
The estimate of the other term is similar. The proof of Proposition \ref{prop good solution with lowest order estimate when v is nonzero} is complete.\end{proof}
\begin{proof}[Proof of Proposition \ref{prop good solution with lowest order estimate when v is imaginary}:] In this case we write  i.e $-v^{2}=\mu^{2}>0$, where $\mu$ is real and positive. Hence equation (\ref{equ the standard ODE}) becomes 
\begin{equation}\label{equ ODE when v is pure imaginary}
\frac{d^{2} u}{d r^{2}}+\frac{1}{r}\frac{d u}{d r}+\frac{\mu^{2}u}{r^{2}}=f.
\end{equation}
Since the solutions to the homogeneous equation are $\sin (\mu\log r)$
, $\cos (\mu\log r)$, then 
we choose a particular solution to (\ref{equ ODE when v is pure imaginary})   as
\begin{eqnarray}\label{equ solution when v is purely imaginary}
& &u
\\&=&\frac{\sin (\mu\log r)}{\mu}\int^{r}_{0}x\cos (\mu\log x)f(x)dx-\frac{\cos (\mu\log r)}{\mu} \int^{r}_{0}x\sin (\mu\log x)f(x)dx.\nonumber
\end{eqnarray}
Since the solutions to homogeneous equations are bounded, the estimate is  easier than that of Proposition \ref{prop good solution with lowest order estimate when v is nonzero}. Choosing $q$ such that $p<q<0$, for the estimate of $u_{I}=\frac{\sin (\mu\log r)}{\mu}\int^{r}_{0}x\cos (\mu\log x)f(x)dx$, we  only need to use (\ref{eqnarray Prop ODE solution v>0 1}) and (\ref{eqnarray Prop ODE solution v>0 2}) by replacing the "$r^{v}$" there by  $\sin (\mu\log r)$, "$x^{-v}$" by $\cos (\mu\log x)$. The estimate of the other term is similar. \end{proof}
\begin{proof}[Proof of Proposition \ref{prop existence of solution with lowest order estimate}:] When $v=0$, the solutions to the homogeneous equation 
\begin{equation}\label{equ the standard homo ODE v=0}
\frac{d^{2} u}{d r^{2}}+\frac{1}{r}\frac{d u}{d r}=0
\end{equation}
are $1$ and $\log r$, then we choose a particular solution to 
\begin{equation}\label{equ solution when v=0}
\frac{d^{2} u}{d r^{2}}+\frac{1}{r}\frac{d u}{d r}=f\ \textrm{as}\
 u=-\int_{0}^{r}x\log x f(x)dx+\log r\int_{0}^{r}xf(x)dx.\
\end{equation} 
Let $u_{II}=\log r\int_{0}^{r}xf(x)dx$, we compute
\begin{equation}
u_{II}^{2}\leq (\log r)^{2}(\int^{r}_{0}f^{2}xdx)(\int^{r}_{0}xdx)=\frac{r^{2}(\log r)^{2}}{2}\int^{r}_{0}f^{2}xdx.
\end{equation}
\begin{eqnarray*}
\textrm{Then}& &\int^{\frac{1}{2}}_{0}u_{II}^{2}r^{2p-3}(-\log r)^{2b-2}dr
\\& \leq & \int_{0}^{r}f^{2}xdx \int_{x}^{\frac{1}{2}}
r^{2p-1}(-\log r)^{2b}dr \leq \bar{C}(\int_{0}^{r}f^{2}xdx) [x^{2p}(-\log x)^{2b}]
\\&= &\bar{C}\int_{0}^{r}f^{2}x^{2p+1}(-\log x)^{2b}dx.
\end{eqnarray*}
The estimate for the other term in (\ref{equ solution when v=0}) is similar.\end{proof}
Next we use integration by parts to improve Proposition \ref{prop existence of solution with lowest order estimate} to   \ref{prop final ODE W22 estimate when v=0}. We verify the following identity, it doesn't harm to have identities for every $v$ instead of only for $v=0$.
\begin{clm}\label{clm weight change on the ODE} Suppose $B$, $V^{2}$, $p$ are  real numbers. Suppose $u$ is a solution to the following differential equation
\begin{equation}
\frac{d^{2} u}{d r^{2}}+\frac{B}{r}\frac{d u}{d r}-\frac{V^{2}u}{r^{2}}=f,
\end{equation}
then the function $\bar{u}=r^{p}u$ satisfies 
\begin{equation}
\frac{d^{2} \bar{u}}{d r^{2}}+\frac{(B-2p)}{r}\frac{d \bar{u}}{d r}-[V^{2}+p(p-1)+p(B-2p)]\frac{\bar{u}}{r^{2}}=fr^{p}.
\end{equation}
In particular, when $B=1$, we have 
\begin{equation}
\frac{d^{2} \bar{u}}{d r^{2}}+\frac{(1-2p)}{r}\frac{d \bar{u}}{d r}-[V^{2}-p^{2}]\frac{\bar{u}}{r^{2}}=fr^{p}.
\end{equation}
\end{clm}
Suppose $u$ solves (\ref{equ the standard ODE}), consider 
$\bar{u}$ as in Claim \ref{clm weight change on the ODE}.  Then $u$ satisfies
\begin{equation}\label{weighted ODE}
\frac{d^{2} \bar{u}}{d r^{2}}+\frac{k}{r}\frac{d \bar{u}}{d r}-\frac{a^{2}\bar{u}}{r^{2}}=\bar{f},
\end{equation}
where  $\bar{f}=r^{p}f$, $k=1-2p$, $a^{2}=v^{2}-p^{2}$. Then 
\begin{equation}
k^{2}+2a^{2}=2v^{2}+1-4p+2p^{2},\ a^{4}-2ka^{2}-2a^{2}=(v^{2}-p^{2})(v^{2}-(p-2)^{2}). 
\end{equation}
Moreover, we consider the weight $w_{0}$ as 
\begin{equation}\label{equ w0 is decreasing}
w_{0}=(\frac{1}{2}-r)^{10}(-\log r)^{2b}.\ \textrm{Notice that}\ \frac{d w_{0}}{d r}\leq 0.
\end{equation}
Using  $\bar{u}=r^{p}u$  and  integrating the square of (\ref{weighted ODE}), we directly verify
\begin{lem}Under the same condition on $f$ as in Theorem \ref{thm existence of good solutions to the uniform ODE with estimates}, suppose in each case we choose the solutions to (\ref{weighted ODE}) ((\ref{equ the standard ODE})) as in (\ref{equ solution when v is real and  p>1-v is integral from 1 to r}), (\ref{equ solution formula when v positive and p less than 1-v}), (\ref{equ solution when v is purely imaginary}),  (\ref{equ solution when v=0}). Then all the boundary terms in (\ref{equ rough L2 identity integration by parts}) are 0. Consequently,  
\begin{eqnarray}\label{equ strong L2 identity integration by parts}
& &\int^{\frac{1}{2}}_{0}\bar{f}^{2}rw_{0} dr
\\&=& \int^{\frac{1}{2}}_{0}|\frac{d^{2} \bar{u}}{d r^{2}}|^{2}rw_{0} dr+(2v^{2}+1-4p+2p^{2})\int^{\frac{1}{2}}_{0}|\frac{d \bar{u}}{d r}|^{2}\frac{w_{0}}{r} dr\nonumber
\\& &+ (v^{2}-p^{2})(v^{2}-(p-2)^{2})\int^{\frac{1}{2}}_{0}\frac{ \bar{u}^{2} w_{0}}{ r^{3}} dr-(1-2p)\int^{\frac{1}{2}}_{0}|\frac{d \bar{u}}{d r}|^{2}\frac{d w_{0}}{d r}dr\nonumber
\\& &+(3-2p)(v^{2}-p^{2})\int^{\frac{1}{2}}_{0}\frac{ \bar{u}^{2} }{ r^{2}} \frac{d w_{0}}{d r} dr-(v^{2}-p^{2})\int^{\frac{1}{2}}_{0}\frac{ \bar{u}^{2} }{ r}\frac{d^{2} w_{0}}{d r^{2}} dr.\nonumber
\end{eqnarray}
\end{lem}

\begin{prop}\label{prop final ODE W22 estimate when v=0}   Under the same conditions in Theorem \ref{thm existence of good solutions to the uniform ODE with estimates}, suppose $v=0$.
Let $\bar{u}=r^{p}u$, $u$ is the solution  to (\ref{equ the standard ODE}) in (\ref{equ solution when v=0}), then 
\begin{eqnarray}\label{eqnarray equivalent statement for the L2 estimate v=0}\int^{\frac{1}{2}}_{0}\bar{u}^{2} w_{0}r^{-3} dr
 \leq   \bar{C}\int^{\frac{1}{2}}_{0}\bar{f}^{2}rw_{0} dr . 
\end{eqnarray}
 Consequently, $u$ satisfies (\ref{equ in thm of ODE optimal estimate v neq 1-p}).
\end{prop}
\begin{proof}[Proof of Proposition \ref{prop final ODE W22 estimate when v=0}:]

By (\ref{equ w0 is decreasing}) and $p<0$, the terms on the right hand side of (\ref{equ strong L2 identity integration by parts}) are all non-negative except the term $-(0-p^{2})\int^{\frac{1}{2}}_{0}\frac{ \bar{u}^{2} }{ r}\frac{d^{2} w_{0}}{d r^{2}} dr$ (even this term is non-negative when $b\geq 1$, but we do not  assume this in general). We compute  
\begin{eqnarray}\label{eqnarray second derivative of w0}
\frac{d^{2} w_{0}}{d r^{2}}=\textrm{non-negative terms}+\frac{2b(2b-1)}{r^{2}}(\frac{1}{2}-r)^{10}(-\log r)^{2b-2}.\nonumber
\end{eqnarray}
Hence (\ref{equ strong L2 identity integration by parts}) implies
\begin{eqnarray}\label{eqnarray L2 estimate v=0 with the junk term}& &\int^{\frac{1}{2}}_{0}\bar{u}^{2}r^{-3}w_{0} dr 
 \leq   \bar{C}\{\int^{\frac{1}{2}}_{0}\bar{f}^{2}rw_{0} dr+\int^{\frac{1}{2}}_{0}\frac{\bar{u}^{2}}{r^{3}} (\frac{1}{2}-r)^{10}(-\log r)^{2b-2} dr\}\nonumber
\\& \leq &  \bar{C}\{\int^{\frac{1}{2}}_{0}\bar{f}^{2}r(-\log r)^{2b} dr+\int^{\frac{1}{2}}_{0}\frac{\bar{u}^{2}}{r^{3}}(-\log r)^{2b-2} dr\}. 
\end{eqnarray}
\begin{equation}\label{equ integral of u with log weight -2 is bounded}\textrm{Proposition \ref{prop existence of solution with lowest order estimate} implies}\ 
\int^{\frac{1}{2}}_{0} \bar{u}^{2} (-\log r)^{2b-2}r^{-3} dr\leq  \bar{C}\int^{\frac{1}{2}}_{0}\bar{f}^{2}r(-\log r)^{2b} dr.\nonumber
\end{equation}
The proof of (\ref{eqnarray equivalent statement for the L2 estimate v=0})  is complete by combining the above two inequalities.\end{proof}

\subsection{Local solutions for the model cone connection \label{section Local solutions for the model cone connection}}

Our goal in this section is to solve the following  deformation equation in model case, and obtain optimal Sobolev-estimates.
\begin{equation}\label{equ model equation of laplacian for cone}
L_{A_{O}}\xi=\underline{f}.
\end{equation}

\begin{Def}Let $w=r^{2p}(-\log r)^{2b}$,  $p<0$. Let $\mathfrak{B}$ be an open set in $\R^{n}\setminus O$, and   $\xi$ be a smooth  section of $ \Xi$ over $\mathfrak{B}$, we define
\begin{eqnarray*}
& &|\xi|^{2}_{W^{2,2}_{p,b,A_{0}}(\mathfrak{B})}
 \triangleq \int_{\mathfrak{B}}|\nabla_{A_{0}}\nabla_{A_{0}} \xi|^{2}wdV+\int_{\mathfrak{B}}\frac{1}{r^{2}}|\nabla_{A_{0}}\xi|^{2}wdV+\int_{\mathfrak{B}}\frac{|\xi|^{2}}{r^{4}}wdV
\end{eqnarray*}
  \begin{equation*}
|\xi|^{2}_{W^{1,2}_{p,b,A_{0}}(\mathfrak{B})}
 \triangleq \int_{\mathfrak{B}}|\nabla_{A_{0}} \xi|^{2}wdV+\int_{\mathfrak{B}}\frac{|\xi|^{2}}{r^{2}}wdV;\ 
 |\xi|^{2}_{L^{2}_{p,b}(\mathfrak{B})}
 \triangleq \int_{\mathfrak{B}}|\xi|^{2}wdV.
\end{equation*}

Since $A$ is admissible,  the norm $W^{2,2}_{p,b,A}(\mathfrak{B})$ (with connection $A$ instead of $A_{O}$) is equivalent to $W^{2,2}_{p,b,A}(\mathfrak{B})$, etc. Thus, by deleting the symbol of the connections, we denote the spaces as   $W^{2,2}_{p,b}(\mathfrak{B})$, $W^{1,2}_{p,b}(\mathfrak{B})$.

When the bundle is trivial over $B_{O}$ and the connection is smooth across $O$,  our $W^{2,2}_{p,b}-$norm is  stronger than the usual $W^{2,2}-$norm weighted by $w$. It fits into our setting better, in the sense that  the estimates in (\ref{equ bounding L2 norm of gradient for the model cone laplace equation}), Proposition \ref{prop bounding L2 norm of Hessian for the model cone laplace equation}, and Corollary \ref{Cor solving model laplacian equation over the ball without the compact support RHS condition} do not depend on the radius of the balls. \end{Def}
\begin{Def}\label{Def of L22 norm model case}

 The space $W^{2,2}_{p,b}(\mathfrak{B})$ is the completion of the section space  
 $$\{\xi|\xi \in C^{2}(\mathfrak{B}\setminus O),\ |\xi|_{W^{2,2}_{p,b} (\mathfrak{B})}<\infty\}\ \textrm{under the}\ W^{2,2}_{p,b} (\mathfrak{B})- \textrm{norm}. \ \ \ \ \ \ \ \  $$ 
The space $W^{1,2}_{p,b}(\mathfrak{B})$ is the completion of the section space  $$\{\underline{f}|\underline{f} \in C^{1}(\mathfrak{B}\setminus O),\ |\underline{f}|_{W^{1,2}_{p,b} (\mathfrak{B})}<\infty\}\ \textrm{under the}\ W^{2,2}_{p,b} (\mathfrak{B})- \textrm{norm}. \ \ \ \ \ \ \ \  $$ 
The space $L^{2}_{p,b}(\mathfrak{B})$ is the completion of the following under the $L^{2}_{p,b} (\mathfrak{B})-$norm.
$$\{ \underline{f} | \underline{f}  \in C^{\infty}_{c}(\mathfrak{B}\setminus O),  \textrm{only  finite terms in the  series}\ (\ref{equ raw ODE for 1 forms})\ \textrm{are non-zero}\}.$$ 

\end{Def}

We define the space $\mathfrak{W}^{1,2}_{p,b}$ ($\mathfrak{W}^{2,2}_{p,b}$, $\mathfrak{W}^{0,2}_{p,b}$) as the following
\begin{equation}
\mathfrak{W}^{1,2}_{p,b}(B_{O})=\{\xi\in W^{1,2}_{loc} \ \textrm{in the coordinates}\ V_{+,O},\ V_{-,O}| |\xi|_{W^{1,2}_{p,b}(B_{O})}<\infty\}.
\end{equation}
 \begin{lem}\label{lem density of smooth functions in weighted L2 space} For any $\rho>0$,  $L^{2}_{p,b} [B_{O}(\rho)]$ is the space of measurable functions on $B_{O}(\rho)$ which are square integrable with respect to the weight $w$.
\end{lem}
\begin{lem}\label{lem H=W}$W^{k,2}_{p,b}[B_{O}(\rho)]=\mathfrak{W}^{k,2}_{p,b}[B_{O}(\rho)],\ k=0,1,2.$
\end{lem}
\begin{rmk} \textbf{The proofs of the above two lemmas  are deferred to Section \ref{section Appendix D: Density and smooth convergence of   Fourier Series}}. Lemma \ref{lem density of smooth functions in weighted L2 space} reduces Theorem \ref{thm W22 estimate on 1-forms} to a finite-dimensional problem. On Lemma \ref{lem H=W}, we expect that when the bundle $\Xi$ is trivial,   $ \mathfrak{W}^{1,2}_{p,b}[B_{O}(\rho)]$ can even be approximated by sections that are smooth across the singularity. This  is because 
$w$ and $r^{-2}w$ are $A_{p}-$weights (see Theorem 1 in \cite{Goldstein}). 
\end{rmk}
Our main result in this section is the following. \textbf{The crucial observation  is that  the $L^{2}-$estimate (given by Theorem \ref{thm existence of good solutions to the uniform ODE with estimates}) and  simple integration by parts yield the $W^{2,2}-$estimate}.
\begin{thm}\label{thm W22 estimate on 1-forms}Suppose $p<0$ is $A_{O}-$generic and $b\geq 0$. Then there is a bounded linear operator $\mathfrak{Q}_{p,b,A_{O}}$ from $L^{2}_{p,b}[B_{O}(\frac{1}{4})]$ to $W^{2,2}_{p,b}[B_{O}(\frac{1}{4})]$ such that 
\begin{itemize}
\item $-L^{2}_{A_{0}}\mathfrak{Q}_{p,b,A_{O}}=Id$ from  $L^{2}_{p,b}[B_{O}(\frac{1}{4})]$  to itself,
\item the bound on $\mathfrak{Q}_{p,b,A_{O}}$ is less than a  $\bar{C}$  as in Definition \ref{Def special constants}.
\end{itemize}
\end{thm}
\begin{proof}[Proof of Theorem \ref{thm W22 estimate on 1-forms}:]
This is a direct application of Theorem \ref{thm existence of good solutions to the uniform ODE with estimates}. We first prove it  assuming that $\underline{f}$ satisfies the conditions in Definition \ref{Def of L22 norm model case} for  $L^{2}_{p,b}[B_{O}(\frac{1}{4})]$. We write  $\underline{f}=\Sigma_{v<v_{0}}r^{-\frac{3}{2}}f_{v}\Psi_{v}$ for some $0<v_{0}<\infty$ (see Definition \ref{Def v spectrum}). For each $f_{v}$, we  define $\xi_{v}$ as the solution to (\ref{equ the ODE of v for 1 forms}) in (\ref{equ solution when v is real and  p>1-v is integral from 1 to r}), (\ref{equ solution formula when v positive and p less than 1-v}), (\ref{equ solution when v is purely imaginary}), (\ref{equ solution when v=0}), and let
\begin{equation}\label{equ Def of mathfrak Q right inverse}
\mathfrak{Q}_{p,b,A_{O}}\underline{f}\triangleq \Sigma_{v}r^{-\frac{3}{2}}\xi_{v}\Psi_{v}\triangleq \xi.
\end{equation}
It suffices to bound the $L^{2}_{p-2,b}[B_{O}(\frac{1}{4})]-$norm of $\xi$.     By  Theorem \ref{thm existence of good solutions to the uniform ODE with estimates}, we find 
 \begin{eqnarray}\label{eqnarray L2 estimate in model  case for laplacian}
& &\int_{B_{O}(\frac{1}{4})}\frac{|\xi|^{2}}{r^{4}}wdV
 =  \Sigma_{v}\int^{\frac{1}{4}}_{0}\int_{S^{6}(1)}\xi_{v}^{2}r^{-9}|\Psi_{v}|^{2}_{S}r^{6}wd\theta dr\nonumber
= \Sigma_{v}\int^{\frac{1}{4}}_{0}|\xi_{v}|^{2} r^{-3}wdr
\\&\leq &  \bar{C} \Sigma_{v}\int^{\frac{1}{4}}_{0}f_{v}^{2} rwdr=\bar{C}\int_{B_{O}(\frac{1}{4})}|\underline{f}|^{2}wdV. 
 \end{eqnarray} 
 
 By  Definition \ref{Def of L22 norm model case}, (\ref{eqnarray L2 estimate in model  case for laplacian}), (\ref{equ bounding L2 norm of gradient for the model cone laplace equation}),   and Proposition \ref{prop bounding L2 norm of Hessian for the model cone laplace equation}, the proof is complete when $\underline{f}$ satisfies the a priori conditions in the first paragraph of this proof. 
 
 In the general case, for any $\underline{f}\in L^{2}_{p,b}$, by Lemma \ref{lem density of smooth functions in weighted L2 space}, there exists a sequence   $\underline{f}_{j}\rightarrow \underline{f}$ in $L^{2}_{p,b}-$topology, and $\underline{f}_{j}$ satisfy the a priori conditions. We denote $\xi_{j}$ as $\mathfrak{Q}_{p,b,A_{O}}\underline{f}_{j}$.  
 By the a priori estimate proved (in the first step) and the linearity of $\mathfrak{Q}_{p,b,A_{O}}$,  $\xi_{j}$ is a Cauchy-sequence in $W^{2,2}_{p,b}[B_{O}(\frac{1}{4})]$. By completeness, $\xi_{j}$ converges to $\xi_{\infty}$ in $W^{2,2}_{p,b}[B_{O}(\frac{1}{4})]$.  Moreover, $\xi_{\infty}$ satisfies the estimate in Theorem \ref{thm W22 estimate on 1-forms}. We thus define $\mathfrak{Q}_{p,b,A_{O}}\underline{f} $  as $\xi_{\infty}$, the  bounds in Theorem \ref{thm W22 estimate on 1-forms} implies that this definition does not depend on the approximation. The proof  is complete. \end{proof}
 \begin{cor}\label{Cor solving model laplacian equation over the ball without the compact support RHS condition} Let $p,b$ be as in Theorem \ref{thm W22 estimate on 1-forms}. There exists a local right inverse $Q_{p,b,A_{O}}$  of $L_{A_{O}}$  with the following properties. For any  $\tau\leq \frac{1}{10}$, 
 \begin{itemize}
 \item $Q_{p,b,A_{O}}$ is bounded from $L^{2}_{p,b}[B_{O}(\tau)]$ to $W^{1,2}_{p,b}[B_{O}(\tau)]$. The bound is less than a  $\bar{C}$  as in Definition \ref{Def special constants}. In particular, it does not depend on $\tau$. 
\item $L_{A_{0}}Q_{p,b,A_{O}}=Id$ from  $L^{2}_{p,b}[B_{O}(\tau)]$  to itself.
\end{itemize}

\end{cor}

\begin{proof}[Proof of Corollary \ref{Cor solving model laplacian equation over the ball without the compact support RHS condition}:] By extending to vanish outside $B_{O}(\tau)$,  $\underline{f}$ can be viewed as a section in $L^{2}_{p,b}[B_{O}(\frac{1}{4})]$. It suffices to take $Q_{p,b,A_{O}}=-L_{A_{O}}\mathfrak{Q}_{p,b,A_{O}}$ and restrict it to $B_{O}(\tau)$. Under this extension,   $Q_{p,b,A_{O}}$ does not depend on $\tau$.   \end{proof}

 
 

 Next, we establish two crucial building-blocks of Theorem \ref{thm W22 estimate on 1-forms}.
\begin{lem}\label{lem bounding L2 norm of gradient for the model cone laplace equation}Under the same conditions in Theorem \ref{thm W22 estimate on 1-forms},  suppose $\underline{f}$ satisfies the a priori conditions in the first paragraph of proof of Theorem \ref{thm W22 estimate on 1-forms}. Let $\xi=\mathfrak{Q}_{p,b,A_{O}}\underline{f}$. Then  for any $ \varrho \in (0,1]$, the following bound holds. 
\begin{equation}\label{equ bounding L2 norm of gradient for the model cone laplace equation}\int_{B_{O}(\frac{\varrho}{4.5})}\frac{|\nabla_{A_{O}}\xi|^{2}}{r^{2}}wdV\leq \bar{C} (\int_{B_{O}(\frac{\varrho}{4})}\frac{|\xi|^{2}}{r^{4}}wdV+\int_{B_{O}(\frac{\varrho}{4})}|\underline{f}|^{2} wdV).
\end{equation}
 $\bar{C}$ is as in Definition \ref{Def special constants}, thus is \textbf{independent of $\varrho$}. 
\end{lem}
\begin{rmk} The estimate is independent of $\varrho$  because (\ref{equ bounding L2 norm of gradient for the model cone laplace equation}) is   scaling-correct. This is  important for (\ref{equ Thm global apriori est 1}). When $\underline{f}$ satisfies the a priori conditions, $\mathfrak{Q}_{p,b,A_{O}}\underline{f}$ is smooth away from $O$, thus every term in the proofs of Theorem \ref{thm W22 estimate on 1-forms}, Lemma \ref{lem bounding L2 norm of gradient for the model cone laplace equation}, Proposition \ref{prop bounding L2 norm of Hessian for the model cone laplace equation} makes sense.
\end{rmk}
\begin{proof}  Let $\eta_{\epsilon}$ be the standard cut-off function (of the singular point $O$) which  vanishes in $B_{O}(\epsilon)$, and is identically $1$ when $r\geq 2\epsilon$. We have 
\begin{equation}\label{equ cut-off function bound near the singular point}
\epsilon|\nabla \eta_{\epsilon}|+\epsilon^{2}|\nabla^{2} \eta_{\epsilon}|< \bar{C},\ \textrm{when}\ \epsilon<<\varrho. 
\end{equation}

 Let $\chi$ be the standard cut-off function  supported in $B_{O}(\frac{\varrho}{4})$ and identically $1$ over $B_{O}(\frac{\varrho}{4.5})$, $|\nabla \chi|\leq \frac{\bar{C}}{\varrho}$.  Lemma \ref{lem formula for LA squared} implies 
\begin{equation}\label{equ in L2 bound on the gradient consequence of Bochner formula}
\nabla^{\star}_{A_{O}}\nabla_{A_{O}}\xi=-\underline{f}+F_{A_{O}}\otimes \xi.
\end{equation} We compute 
\begin{eqnarray}\label{eqnarray in lemma bounding L2 norm of gradient for the model cone laplace equation}& &
\int_{B_{O}(\frac{\varrho}{4})}\frac{|\nabla_{A_{O}}\xi|^{2}}{r^{2}}\eta_{\epsilon}\chi^{2} wdV
\\&=& \int_{B_{O}(\frac{\varrho}{4})}\frac{<\xi, \nabla^{\star}_{A_{O}}\nabla_{A_{O}}\xi>}{r^{2}}\eta_{\epsilon}w\chi^{2} dV- \int_{B_{O}(\frac{\varrho}{4})}<\xi, [\nabla(\frac{\eta_{\epsilon}\chi^{2} w}{r^{2}})]\lrcorner\nabla_{A_{O}}\xi> dV\nonumber
\\&=& -\int_{B_{O}(\frac{\varrho}{4})}\frac{<\xi,\underline{f}>}{r^{2}}\eta_{\epsilon}w\chi^{2} dV+\int_{B_{O}(\frac{\varrho}{4})}\frac{<\xi,F_{A_{O}}\otimes \xi>}{r^{2}}\eta_{\epsilon}w\chi^{2} dV\nonumber
\\& & -\int_{B_{O}(\frac{\varrho}{4})}<\xi, (\nabla \eta_{\epsilon})\lrcorner\nabla_{A_{O}}\xi> \frac{\chi^{2} w}{r^{2}}dV-\int_{B_{O}(\frac{\varrho}{4})}<\xi, (\nabla \frac{\chi^{2} w}{r^{2}})\lrcorner\nabla_{A_{O}}\xi> \eta_{\epsilon}dV
\nonumber
\end{eqnarray}

By Definition \ref{Def of L22 norm model case}, the cheap estimate  $|F_{A_{O}}|\leq \frac{\bar{C}}{r^{2}}$, and Cauchy-Schwartz inequality, the first 2 terms on the most right hand side of (\ref{eqnarray in lemma bounding L2 norm of gradient for the model cone laplace equation})   are    bounded by the right hand side of (\ref{equ bounding L2 norm of gradient for the model cone laplace equation}), uniformly in $\epsilon$. Note 
\begin{equation}\label{equ gradient of the cutoff function and weight}|\nabla w|\leq \frac{Cw}{r}\Longrightarrow|\nabla (\frac{\chi^{2} w}{r^{2}})|\leq \frac{\bar{C}\chi^{2} w}{r^{3}}+2\chi|\nabla\chi|\frac{ w}{r^{2}},\ 
\end{equation} 
hence we obtain the following bound on the last term.
  \begin{eqnarray}\label{eqnarray in model L12 bound 1}& &\int_{B_{O}(\frac{\varrho}{4})}<\xi, (\nabla \frac{\chi^{2} w}{r^{2}})\lrcorner\nabla_{A_{O}}\xi> \eta_{\epsilon}dV
\\&\leq & \bar{C}\int_{B_{O}(\frac{\varrho}{4})} |\xi||\nabla_{A_{O}}\xi|\frac{\chi^{2}w\eta_{\epsilon}}{r^{3}} dV+ \bar{C}\int_{B_{O}(\frac{\varrho}{4})} |\xi||\nabla_{A_{O}}\xi|\frac{\chi|\nabla \chi|w\eta_{\epsilon}}{r^{2}} dV \nonumber
\\&\leq &  \vartheta \int_{B_{O}(\frac{\varrho}{4})}\frac{|\nabla_{A_{O}}\xi|^{2}}{r^{2}}\eta_{\epsilon}\chi^{2} wdV+\bar{C}_{\vartheta}\int_{B_{O}(\frac{\varrho}{4})}\frac{|\xi|^{2}}{r^{4}}\chi^{2}w\eta_{\epsilon}dV\nonumber
\\& &+\bar{C}_{\vartheta}\int_{B_{O}(\frac{\varrho}{4})}|\nabla\chi|^{2}\frac{|\xi|^{2}}{r^{2}}w\eta_{\epsilon}dV.\nonumber
\end{eqnarray}
On the last term above, we notice that in $B_{O}(\frac{\varrho}{4})$, $\frac{1}{\varrho}\leq \frac{1}{r}$. By 
$|\nabla \chi|\leq \frac{\bar{C}}{\varrho}$,  we obtain the following bound 
\begin{equation}\label{equ in model L12 bound 1}
\bar{C}_{\vartheta}\int_{B_{O}(\frac{\varrho}{4})}|\nabla\chi|^{2}\frac{|\xi|^{2}}{r^{2}}w\eta_{\epsilon}dV\leq \bar{C}_{\vartheta}\int_{B_{O}(\frac{\varrho}{4})}\frac{|\xi|^{2}}{r^{4}}w\eta_{\epsilon}dV.
\end{equation}

Therefore (\ref{eqnarray in model L12 bound 1}) and (\ref{equ in model L12 bound 1}) imply 
\begin{eqnarray}\label{eqnarray the L12 estimate model case with a small error term on the right to be absorbed}
& &\int_{B_{O}(\frac{\varrho}{4})}<\xi, (\nabla \frac{\chi^{2} w}{r^{2}})\lrcorner\nabla_{A_{O}}\xi> \eta_{\epsilon}dV
\\&\leq &\vartheta \int_{B_{O}(\frac{\varrho}{4})}\frac{|\nabla_{A_{O}}\xi|^{2}}{r^{2}}\eta_{\epsilon}\chi^{2} wdV+\bar{C}_{\vartheta}\int_{B_{O}(\frac{\varrho}{4})}\frac{|\xi|^{2}}{r^{4}}wdV\nonumber.
\end{eqnarray}

Let $\vartheta=\frac{1}{10}$,  plugging (\ref{eqnarray the L12 estimate model case with a small error term on the right to be absorbed}) in (\ref{eqnarray in lemma bounding L2 norm of gradient for the model cone laplace equation}) and using the remark under (\ref{eqnarray in lemma bounding L2 norm of gradient for the model cone laplace equation}), we find 
\begin{eqnarray}\label{eqnarray L2 bound on the gradient model case crucial decomposition}
& &\int_{B_{O}(\frac{\varrho}{4})}\frac{|\nabla_{A_{O}}\xi|^{2}}{r^{2}}\eta_{\epsilon}w\chi^{2} dV
\\&\leq & \bar{C}\{\int_{B_{O}(\frac{\varrho}{4})}\frac{|\xi|^{2}}{r^{4}}wdV+\int_{B_{O}(\frac{\varrho}{4})}|\underline{f}|^{2} wdV+|\int_{B_{O}(\frac{\varrho}{4})}<\xi,\nabla_{A_{O},\nabla \eta_{\epsilon}}\xi> \frac{\chi^{2} w}{r^{2}}dV|\}. \nonumber
\end{eqnarray}
It suffices to show $\Pi_{1}$ approaches $0$ as $\epsilon\rightarrow 0$. The condition on $\underline{f}$ implies   $L^{2}_{A_{O}}\xi=0$ in $B_{O}(r_{\underline{f}})$, for some $r_{\underline{f}}>0$. Since  $\xi=\mathfrak{Q}_{p,b,A_{O}}\underline{f}\in  L^{2}_{p-2,b}[B_{O}(\frac{\varrho}{4})])$,  Lemma \ref{lem bound on C3 norm of solution to laplace equation when f is smooth and vainishes near O} (with "$p$" replaced by $p-1$) gives 
 \begin{equation}
 |\xi|\leq \frac{C_{\underline{f}}}{r^{\frac{3}{2}+p-\lambda}},\ |\nabla_{A_{O}}\xi|\leq \frac{C_{\underline{f}}}{r^{\frac{5}{2}+p-\lambda}},\ \lambda>0.\ \textrm{Hence when}\ \epsilon\ \textrm{goes to}\ 0,
\end{equation}
\begin{equation}\label{equ integration by parts holds true in the case of L12 model estimate wrt to cone}
\Pi_{1}\leq C_{\underline{f}}\int_{B_{O}(2\epsilon)-B_{O}(\epsilon)}|\nabla_{A_{O}}\xi||\xi|r^{2p-3}(-\log r)^{2b}dV\leq C_{\underline{f}}\epsilon^{2\lambda}(-\log \epsilon)^{2b}\rightarrow 0\nonumber
\end{equation} The proof of (\ref{equ bounding L2 norm of gradient for the model cone laplace equation}) is complete. \end{proof}
 Using almost the same technique, we obtain
   \begin{prop}\label{prop bounding L2 norm of Hessian for the model cone laplace equation}Let $p,b$,$\underline{f}$,$\xi,\varrho$  be as in Lemma \ref{lem bounding L2 norm of gradient for the model cone laplace equation}. Then 
  \begin{eqnarray*}& &  \int_{B_{O}(\frac{\varrho}{5})}|\nabla^{2}_{A_{O}}\xi|^{2}wdV
   \\& \leq & \bar{C}\int_{B_{O}(\frac{\varrho}{4.5})}\frac{|\nabla_{A_{O}}\xi|^{2}}{r^{2}}wdV+\bar{C} \int_{B_{O}(\frac{\varrho}{4.5})}\frac{|\xi|^{2}}{r^{4}}wdV+\bar{C}\int_{B_{O}(\frac{\varrho}{4.5})}|\underline{f}|^{2} wdV. 
 \end{eqnarray*}
  \end{prop}
  \textbf{For the reader's convenience, we still do the full proof of Proposition \ref{prop bounding L2 norm of Hessian for the model cone laplace equation} in Section \ref{section Appendix E: Various integral identities and proof}}.
  \begin{Def}\label{Def global weight and Sobolev spaces}  Given an admissible  connection $A$,  let $\tau_{0}$ be small enough. Let $p<0$ be $A-$generic and $b\geq 0$.   Let $w_{p,b}$ be the smooth function  such that for any  singular point  $O_{j}$,
\begin{equation}  w_{p,b}=\left \{\begin{array}{cc}
  1& \textrm{when}\ r\geq 2\tau_{0}, \\
  r^{2p}(-\log r)^{2b}& \textrm{when}\ r\leq \tau_{0},\end{array} \right. \end{equation}
   $r$ is the distance to $O_{j}$ in local coordinates (by abuse of notation). Away from the coordinate neighbourhoods of the singular points, $w_{p,b}\equiv 1$.
    
 Then we define the global $L^{2}_{p,b}-$space as the completion of smooth functions (away from the singular points) under the norm $|\cdot|_{L^{2}_{p,b}}$:
\begin{equation}\label{equ Def of global weighted L2 space for tame connections}
|\xi|^{2}_{L^{2}_{p,b}}=\int_{M}|\xi|^{2}w_{p,b}dV.
\end{equation}

We define the space $W^{1,2}_{p,b}$ as the completion of smooth functions (away from the singular points) under the norm 
\begin{equation}
|\xi|_{W^{1,2}_{p,b}}=|\nabla_{A}\xi|_{L^{2}_{p,b}}+|\xi|_{L^{2}_{p-1,b}}.
\end{equation}
 
\textbf{Convention of section-spaces}: the global norms over $M$ are denoted just as $W^{1,2}_{p,b}$ or $L^{2}_{p,b}$ without any symbol on the domain, the local norms are usually with 
a symbol indicating the domain (c.f. Definition \ref{Def of L22 norm model case}).
  \end{Def}
  \begin{Def}\label{Def volume forms} By abuse of notation, let $dV$ denote all the volume forms of our integrations. The convention is: locally, it usually means the Euclidean volume form; globally, it usually means the volume form determined by  $\phi$.
  
  Anyway, the $dV$ in various cases are equivalent up to a constant depending on the (reference) $G_{2}-$structure.
  \end{Def}

 \section{Global Theory}
By abuse of notation, from now on let $f$ denote the image.
  \subsection{Global apriori estimate. \label{section Global apriori estimate}} 
\begin{Def}\label{Def spectrum gap} For any real number $\tau$,  let $\vartheta_{\tau}$ denote the $v-$spectrum gap of $\Upsilon_{A_{O_{j}}}'$s at $\tau$ i.e the distance from $\tau$ to the closest $v-$eigenvalue (of any $\Upsilon_{A_{O_{j}}}$) other than $\tau$ itself. When the gap$>1$, we let $\vartheta_{\tau}=1$. \end{Def} 
The following bootstrapping lemma for the model operator is important especially for Theorem \ref{thm global apriori L22 estimate}. 
 \begin{lem}\label{lem bound on C3 norm of solution to laplace equation when f is smooth and vainishes near O} Let $ \tau_{0} \in (0,\frac{1}{10})$, $p< 0$, $b\geq 0$. Suppose   $\xi\in L^{2}_{p-1,b}[B_{O}(2\tau_{0})]$, and  $L^{2}_{A_{O}}\xi=0$ in $B_{O}(2\tau_{0})$. Then $\xi$ is actually  in $ L^{2}_{p-1-\lambda,b}[B_{O}(2\tau_{0})]$  for all $0\leq \lambda <\vartheta_{-p}$, and we have the following estimates. 
 \begin{equation}\label{equ local kernel regularity 0}
 |\xi|_{L^{2}_{p-1-\lambda,b}[B_{O}(2\tau_{0})]}\leq \frac{\bar{C}}{\tau_{0}^{1+\lambda}}|\xi|_{L^{2}_{p,b}[B_{O}(2\tau_{0})]}.
 \end{equation}
\begin{equation}\label{equ local kernel regularity 0.5}
r|\xi|+r^{2}|\nabla_{A_{O}}\xi|+r^{3}|\nabla^{2}_{A_{O}}\xi|+r^{4}|\nabla^{3}_{A_{O}}\xi|\leq C_{\xi,\tau_{0}}r^{-\frac{3}{2}-p+\lambda}\ \textrm{when}\ r<\tau_{0}.  
\end{equation}
$\bar{C}$ is \textbf{independent of $\tau_{0}$}.

\end{lem}
\begin{rmk} We don't need $p$ to be $A_{O}-$generic, since  by  condition (\ref{equ local regularity L2p-1 bound}),   any "harmonic" section in $W^{1,2}_{p,b}$ does not have non-trivial component on $v=-p$.
\end{rmk}

\begin{proof}[Proof of Lemma \ref{lem bound on C3 norm of solution to laplace equation when f is smooth and vainishes near O}:] \textbf{The idea of proof is to use Fourier-expansion to rule out some bad eigenvalues}. We write
 \begin{equation}\label{equ local kernel regularity 1}
 \xi=r^{-\frac{3}{2}}U_{v}\Psi_{v}\ (\textrm{see Definition}\ \ref{Def v spectrum}).
 \end{equation}
  Since $\xi$ is harmonic, and $U_{v}$ is the spherical inner product,  by (\ref{equ cone formula for square of dirac}) (with right hand side as $0$), we directly verify  
 \begin{equation}
 \frac{\partial ^{2} U_{v} }{\partial r^{2}}+\frac{1}{r}\frac{\partial  U_{v} }{\partial r}-\frac{v^{2}U_{v}}{r^{2}}=0.
 \end{equation}
 This means 
 \begin{equation} \label{equ local kernel regularity 2}  U_{v}=\left \{
 \begin{array}{ccc} c_{1,v}r^{-v}+c_{2,v}r^{v};&\ v>0.\\
 c_{1,v}+c_{2,v}\log r;&\ v=0.\\
 c_{1,v}\sin(v\log r)+c_{2,v}\cos(v\log r),&\  v<0.
 \end{array}\right.
 \end{equation}
$c_{1,v}$, $c_{2,v}$ are constants.  The condition $\xi\in  L^{2}_{p-1,b}[B_{O}(2\tau_{0})]$ implies 
\begin{equation} \label{equ local regularity L2p-1 bound}
\int^{2\tau_{0}}_{0} U_{v}^{2}r^{2p-1}(-\log r)^{2b}dr<\infty.\end{equation}
   Since $p< 0$, the terms $1,\log r,\sin (v\log r),\cos (v\log r),r^{-v}$ can not appear. Moreover, $r^{v}$ can not appear  if $2v+2p\leq 0$. Thus,  only those $r^{v}$ with $v>-p$ will appear. Moreover, by the discreteness of the spectrum,  only those $r^{v}$ with $v\geq -p+\vartheta_{-p}$ could appear. In this case, we have  a $v-$ independent $L^{2}_{p-1-\lambda,b}-$estimate by the $L^{2}_{p,b}-$norm. 
 \begin{eqnarray*} \int^{2\tau_{0}}_{0}| r^{v}|^{2}r^{2p-2\lambda-1}(-\log r)^{2b}dr\leq  \frac{(2\tau_{0})^{2v+2p-2\lambda}(-\log 2\tau_{0})^{b}}{2v+2p-2\lambda}.
 \end{eqnarray*}

 On the other hand, \begin{equation}\int^{2\tau_{0}}_{0}|r^{v}|^{2}r^{2p+1}(-\log r)^{2b}dr
 \geq (-\log 2\tau_{0})^{b}\frac{(2\tau_{0})^{2v+2p+2}}{2v+2p+2}.
 \end{equation}
  Then 
 \begin{equation}\label{equ local kernel regularity 3}
 \int^{2\tau_{0}}_{0}|r^{v}|^{2}r^{2p-2\lambda-1}(-\log r)^{2b}dr\leq \bar{C}_{\lambda}\tau_{0}^{-2-2\lambda}\int^{2\tau_{0}}_{0}|r^{v}|^{2}r^{2p+1}(-\log r)^{2b}dr.
 \end{equation}
  Using  (\ref{equ local kernel regularity 1}) and (\ref{equ local kernel regularity 2}), (\ref{equ local kernel regularity 3}) is equivalent to   (\ref{equ local kernel regularity 0}). The estimate  (\ref{equ local kernel regularity 0.5}) is a direct consequence of  (\ref{equ 0 estimate local version}) and  Lemma \ref{lem Schauder estimate in local coordinates in small balls} (with $k=3$, $A=A_{O}$). 
\end{proof}
 
 \begin{clm}\label{clm local regularity L12 for global apriori estimate} Under the same conditions in Lemma \ref{lem bound on C3 norm of solution to laplace equation when f is smooth and vainishes near O}, we  have
   \begin{equation}\label{equ Thm global apriori estimate}|\xi|_{W^{1,2}_{p,b}[B_{O}(\tau_{0})]}\leq \frac{\bar{C}}{\tau_{0}}|\xi|_{L^{2}_{p,b}[B_{O}(2\tau_{0})]},\ \bar{C}\ \textrm{is independent of}\
    \tau_{0}.
 \end{equation}
 \end{clm}
 \begin{proof} It suffices to show $\int_{B_{O}(\tau_{0})}|\nabla_{A_{O}}\xi|^{2}wdV\leq \bar{C} \int_{B_{O}(2\tau_{0})}\frac{|\xi|^{2}}{r^{2}}wdV.$ It is a much easier version of  Lemma \ref{lem bounding L2 norm of gradient for the model cone laplace equation}, we only need to run the argument through with the measure $\eta_{\epsilon}\chi^{2}wdV$ instead of $\frac{\eta_{\epsilon}\chi^{2}w}{r^{2}}dV$. \end{proof}
 \textbf{Our  main theorem in this section implies the image is closed.}
 \begin{thm}(Global a priori estimate)\label{thm global apriori L22 estimate} Suppose $\xi\in W^{1,2}_{p,b}$, then $$|\xi|_{W^{1,2}_{p,b}}\leq C(|L_{A}\xi|_{L^{2}_{p,b}}+|\xi|_{L^{2}_{p,b}}).$$
 \end{thm}
 \begin{proof} \textbf{The observation is that we can reduce the estimate for $L_{A}$ to estimate of the model operator}. We only need to derive this estimate near the singularity. Away from the singularity it follows from the standard estimates, then we patch up the estimates in each piece. 
 
 For any $\epsilon_{0}$, when $\tau_{0}-$small enough,  given $\xi \in W^{1,2}_{p,b}[B_{O}(2\tau_{0})]$,  by Corollary \ref{Cor solving model laplacian equation over the ball without the compact support RHS condition}, there exists  a $\eta$  such that 
  \begin{eqnarray}\label{eqnarray Thm global apriori estimate}
  & &L_{A_{O}}\eta=L_{A_{O}}\xi\  \textrm{and}
 \\& &  |\eta|_{W^{1,2}_{p,b}[B_{O}(2\tau_{0})]}\leq \bar{C}|L_{A_{O}}\xi|_{L^{2}_{p,b}[B_{O}(2\tau_{0})]};\ \bar{C}\ \textrm{is independent of}\ \epsilon_{0}\ \textrm{and}\ \tau_{0}.\nonumber
\\& \leq &  \bar{C}|(L_{A}-L_{A_{O}})\xi|_{L^{2}_{p,b}[B_{O}(2\tau_{0})]}+ \bar{C}|L_{A}\xi|_{L^{2}_{p,b}[B_{O}(2\tau_{0})]}\nonumber
\\&\leq &  \bar{C}|L_{A}\xi|_{L^{2}_{p,b}[B_{O}(2\tau_{0})]}+ \epsilon_{0}|\xi|_{W^{1,2}_{p,b}[B_{O}(2\tau_{0})]}.\nonumber
  \end{eqnarray}
 Then we estimate 
 \begin{eqnarray}\label{equ priliminary L22 apriori estimate }
& & |\xi|_{W^{1,2}_{p,b}[B_{O}(\tau_{0})]}\leq |\xi-\eta|_{W^{1,2}_{p,b}[B_{O}(\tau_{0})]}+|\eta|_{W^{1,2}_{p,b}[B_{O}(\tau_{0})]}\nonumber
\\&\leq & \bar{C} |L_{A}\xi|_{L^{2}_{p,b}[B_{O}(2\tau_{0})]}+ \epsilon_{0}|\xi|_{W^{1,2}_{p,b}[B_{O}(2\tau_{0})]}+ |\xi-\eta|_{W^{1,2}_{p,b}[B_{O}(\tau_{0})]}.
 \end{eqnarray}
 $\xi-\eta$ satisfies $L_{A_{O}}(\xi-\eta)=0$ in $B_{O}(2\tau_{0})$. Then it's smooth away from the singularity, and we have 
 \begin{equation}
-L^{2}_{A_{O}}(\xi-\eta)=0.
 \end{equation}
 
 Thus (\ref{equ priliminary L22 apriori estimate })  and Claim \ref{clm local regularity L12 for global apriori estimate} (for $\xi-\eta$) yield 
 \begin{equation}\label{equ Thm Global apriori estimate -1}
 |\xi|_{W^{1,2}_{p,b}[B_{O}(\tau_{0})]}\leq \bar{C}|L_{A}\xi|_{L^{2}_{p,b}[B_{O}(2\tau_{0})]}+ \epsilon_{0}|\xi|_{W^{1,2}_{p,b}[B_{O}(2\tau_{0})]}+ \frac{\bar{C}}{\tau_{0}}|\xi-\eta|_{L^{2}_{p,b}[B_{O}(2\tau_{0})]}.
 \end{equation}
 
Within $B_{O}(2\tau_{0})$, we have $\frac{1}{\tau_{0}}\leq \frac{2}{r}$, then by definition and (\ref{eqnarray Thm global apriori estimate}), we obtain
 \begin{equation}\label{equ Thm global apriori est 0}
\frac{|\eta|_{L^{2}_{p,b}[B_{O}(2\tau_{0})]}}{\tau_{0}}
 \leq
 4|\eta|_{L^{2}_{p-1,b}[B_{O}(2\tau_{0})]}\leq \bar{C}|L_{A}\xi|_{L^{2}_{p,b}[B_{O}(2\tau_{0})]}+ 4\epsilon_{0}|\xi|_{W^{1,2}_{p,b}[B_{O}(2\tau_{0})]},
 \end{equation}
 where $\bar{C}$ \textbf{does not depend on $\epsilon_{0}$ or $\tau_{0}$}. 
 Then (\ref{equ Thm global apriori estimate}), (\ref{equ Thm Global apriori estimate -1}), and (\ref{equ Thm global apriori est 0}) imply 
 \begin{equation}\label{equ Thm global apriori est 1}
 |\xi|_{W^{1,2}_{p,b}[B_{O}(\tau_{0})]}\leq \bar{C}|L_{A}\xi|_{L^{2}_{p,b}[B_{O}(2\tau_{0})]}+ \bar{C}\epsilon_{0}|\xi|_{W^{1,2}_{p,b}[B_{O}(2\tau_{0})]}+ \frac{\bar{C}}{\tau_{0}}|\xi|_{L^{2}_{p,b}[B_{O}(2\tau_{0})]}.
 \end{equation}
 
 Away from the singularities we have
  \begin{equation}\label{equ Thm global apriori est away from singularities}
 |\xi|_{W^{1,2}_{p,b}(M_{\frac{\tau_{0}}{2}})}\leq C_{\epsilon_{0},\tau_{0}}|L_{A}\xi|_{L^{2}_{p,b}(M_{\frac{\tau_{0}}{4}})}+ C_{b,p,\tau_{0}}|\xi|_{L^{2}_{p,b}(M_{\frac{\tau_{0}}{4}})}.
 \end{equation}
Adding (\ref{equ Thm global apriori est 1}) and (\ref{equ Thm global apriori est away from singularities}), we find 
 \begin{equation}
 |\xi|_{W^{1,2}_{p,b}}\leq C|L_{A}\xi|_{L^{2}_{p,b}}+ \bar{C}\epsilon_{0}|\xi|_{W^{1,2}_{p,b}}+ C|\xi|_{L^{2}_{p,b}}.
 \end{equation}
 Choosing the $\epsilon_{0}$ (initially in (\ref{eqnarray Thm global apriori estimate})) to be less than $\frac{1}{20\bar{C}}$ (we let $\tau_{0}$ be small enough with respect to $\epsilon_{0}$), the proof of Theorem \ref{thm global apriori L22 estimate} is complete.\end{proof}

\begin{thm}\label{thm existence of a good solution when f is in the image} Let $A,p,b$ be as in  Theorem \ref{Thm Fredholm}, then $Ker L_{A}$ in $W^{1,2}_{p,b}$ is finite-dimensional. Moreover, for any $f\in Image L_{A}$, there exists a solution $\xi$ to  $L_{A}\xi=f$  such that $\xi\in Ker^{\perp}L_{A}\subset W^{1,2}_{p,b}$ and $|\xi|_{W^{1,2}_{p,b}}\leq C|f|_{L^{2}_{p,b}}$.

\end{thm}
\begin{proof} We first show that $ker L_{A}$ is finite-dimensional.  Were this not true, there exist countably many $\xi_{k}$'s in $ker L_{A}$, such that for any $k$, $\xi_{k}$ is not in the span of the preceding vectors.  Then using the $L^{2}_{p,b}-$inner product,  the Gram-Schmidt process produces  an orthonormal sequence $\widehat{\xi}_{k}$ of sections in $ker L_{A}$. On the other hand, by Theorem \ref{thm global apriori L22 estimate}, we have $|\widehat{\xi}_{k}|_{W^{1,2}_{p,b}}\leq C$, then Lemma \ref{lem compact imbedding} implies $\xi_{k}$ converges in $L^{2}_{p,b}$, which contradicts the orthogonality.

Since $Ker L_{A}\subset\ W^{1,2}_{p,b}$ is now shown to be  finite dimensional,  we consider the projection of $\underline{\xi}$ onto  $Ker L_{A}$ (with respect to the $L^{2}_{p,b}-$inner product) as $\underline{\xi}\parallel_{Ker L_{A}}$. Given any $\underline{\xi}$ such that $L_{A}\underline{\xi}=f$, we consider 
\begin{equation}
\xi=\underline{\xi}-[\underline{\xi}\parallel_{Ker L_{A}}],\ \textrm{then}\ \xi\in  Ker^{\perp}L_{A}. 
\end{equation}

To prove  the estimate in Theorem \ref{thm existence of a good solution when f is in the image}, by Theorem \ref{thm global apriori L22 estimate},  it suffices to show
\begin{equation}\label{equ apriori L2 estimate when f is in the image and xi perdicular to kernel}
|\xi|_{L^{2}_{p,b}}\leq C|f|_{L^{2}_{p,b}}.
\end{equation}
 Were (\ref{equ apriori L2 estimate when f is in the image and xi perdicular to kernel}) not true,  there exists a sequence $\xi_{i}\in W^{1,2}_{p,b}$, $f_{i}=L_{A}\xi_{i}\in Image L_{A}\subset L^{2}_{p,b}$, such that 
\begin{equation}
|\xi_{i}|_{L^{2}_{p,b}}=1,\ \xi_{i}\in ker^{\perp} L_{A},\ \textrm{but}\ 
|f_{i}|_{L^{2}_{p,b}}\rightarrow 0.
\end{equation}

By  Theorem \ref{thm global apriori L22 estimate} and  Lemma \ref{lem compact imbedding}, 
$\xi_{i}$ converges  in $L^{2}_{p,b}$ to $\xi_{\infty}$. By the  linearity of $L_{A}$, these in turn  imply $\xi_{i}$ is a Cauchy-Sequence in $W^{1,2}_{p,b}$: 
\begin{equation}
|\xi_{i}-\xi_{j}|_{W^{1,2}_{p,b}}\leq C|\xi_{i}-\xi_{j}|_{L^{2}_{p,b}}+C|f_{i}-f_{j}|_{L^{2}_{p,b}}\rightarrow 0.
\end{equation}
Then $\xi_{i}$ converges to $\xi_{\infty}$ in  $W^{1,2}_{p,b}$,  $\xi_{\infty}\in W^{1,2}_{p,b}$, $|\xi_{\infty}|_{L^{2}_{p,b}}=1$, and 
$L_{A}\xi_{\infty}=0$. But $\xi_{i}\in ker^{\perp} L_{A}$ implies $\xi_{\infty}\perp ker L_{A}$. This is a contradiction. \end{proof}
 
\subsection{Hybrid space and $C^{0}-$estimate.\label{section Hybrid space and C0-estimat}}

Using an interpolation trick, the $C^{0}-$estimate is a direct corollary of the $W^{1,2}_{p,b}$-estimate.
To see this, for any point $q$ close enough to a singular point $O$, \textbf{$\xi\in W^{1,2}_{p,b}\subset L^{2}_{p-1,b}$ implies the average of $|\xi|$ over $B_{q}(\frac{r_{q}}{2})$ is bounded correctly i.e. by $r_{q}^{-\frac{5}{2}-p}(-\log r_{q})^{-b}$. Since $\xi$ satisfies an elliptic equation, the interpolation trick gives the $C^{0}-$bound.}  For second order uniformly elliptic equations of divergence form, this can be done by the well known Nash-Moser iteration.

\begin{Def}\label{Def Hybrid spaces}(Hybrid spaces) The hybrid spaces are defined  as (their norms are defined in the parenthesis on the left hand side of "$<\infty$") 
\begin{itemize}
\item $H_{p,b}=\{\xi\ \textrm{is}\ C^{2,\alpha}\ \textrm{away from}\ O|\ |\xi|_{W^{1,2}_{p,b}}+|\xi|^{(\frac{5}{2}+p,b)}_{2,\alpha,M}<\infty\}$ and 
\item $N_{p,b}=\{\xi\ \textrm{is}\ C^{1,\alpha}\ \textrm{away from}\ O|\ |\xi|_{L^{2}_{p,b}}+|\xi|^{(\frac{7}{2}+p,b)}_{1,\alpha,M}<\infty\}$.
\end{itemize}
\end{Def}

\begin{lem}\label{lem multiplicative property of hybrid spaces} Suppose $b\geq 0$, $p\leq -\frac{3}{2}$. Suppose $\xi_{1},\ \xi_{2}\in H_{p,b}$, then for any smooth tensor product $\otimes$, $|\xi_{1}\otimes \xi_{2}|\in N_{p,b}$ and 
\begin{equation}
|\xi_{1}\otimes \xi_{2}|_{N_{p,b}}\leq C|\xi_{1}|_{H_{p,b}}|\xi_{2}|_{H_{p,b}},\ C\ \textrm{depends on the }\ C^{2}-\textrm{norm of}\ \otimes. 
\end{equation}
\end{lem}
This multiplicative property works for the quadratic non-linearity of (\ref{equ instantonequation without cokernel}). 
\begin{proof}By Definition \ref{Def Global Schauder norms}, \ref{Def local Schauder norms}, and the conditions on $p,b$,  we have $\xi_{1}\otimes \xi_{2}\in C^{1,\alpha}_{(\frac{7}{2}+p,b)}$ and 
\begin{equation}
|\xi_{1}\otimes \xi_{2}|_{C^{1,\alpha}_{(\frac{7}{2}+p,b)}}\leq C|\xi_{1}|_{C^{1,\alpha}_{(\frac{5}{2}+p,b)}}|\xi_{2}|_{C^{1,\alpha}_{(\frac{5}{2}+p,b)}}\leq C|\xi_{1}|_{H_{p,b}}|\xi_{2}|_{H_{p,b}}.
\end{equation}

\textbf{The $L^{2}_{p,b}-$bound is estimated by making use of the $C^{0}-$norms:}
\begin{eqnarray*}& &
\int_{M}|\xi_{1}\otimes \xi_{2}|^2w_{p,b}dV\ \ \ \ \ \ \ \ \ \ \ \ \ \ \ \ \ \ \ \ p\leq \frac{3}{2}
\\&\leq & C|\xi_{1}|_{C^{1,\alpha}_{(\frac{5}{2}+p,b)}}\int_{M}\frac{|\xi_{2}|^2}{r^{5+2p}(-\log r)^{2b}}w_{p,b}dV\leq C|\xi_{1}|_{H_{p,b}}\int_{M}\frac{|\xi_{2}|^2}{r^{2}}w_{p,b}dV
\\&\leq & C|\xi_{1}|_{H_{p,b}}|\xi_{2}|_{H_{p,b}}. 
\end{eqnarray*}
\end{proof}

\begin{thm}\label{thm C0 est} Let  $b\geq 0$, $0<\alpha<1$, and $\gamma$ be any real number. Suppose  $f\in C^{\alpha}_{(\gamma,b)}(M)$, and $\xi$ is  $C^{2,\alpha}$ away from the singularities. Suppose $L_{A}\xi=f$ or $L_{A}^{\star}\xi=f$, and  $\xi\in L^{2}_{-\frac{9}{2}+\gamma,b}$. Then $\xi$  satisfies 
\begin{equation}\label{equ Thm C0 estimate 1}
|\xi|^{(\gamma-1,b)}_{0,M}\leq C\{|f|^{(\gamma,b)}_{\alpha,M}+|\xi|_{L^{2}_{-\frac{9}{2}+\gamma,b}}\}.
\end{equation}
Consequently,
\begin{equation}\label{equ Thm C0 estimate 2}
 |\xi|^{(\gamma-1,b)}_{2,\alpha,M}\leq C\{|f|^{(\gamma,b)}_{1,\alpha,M}+|\xi|_{L^{2}_{-\frac{9}{2}+\gamma,b}}\}.
\end{equation}
 
\end{thm}

\begin{proof} We only prove it for $L_{A}$, the proof for $L_{A}^{\star}$ is the same. 

Let $|\cdot|^{[y]}_{k,\alpha,B}$ denote the weighted norm in (6.10) of \cite{GilbargTrudinger} with respect to $B$. Notice that this is \textbf{different} from $|\cdot|^{(y)}_{k,\alpha,B}$ (Definition \ref{Def Schauder spaces}) on which the weight depends only on the distance to the singular point.

By Lemma \ref{lem Schauder estimate in local coordinates in small balls} and multiplication of weight, in $B\triangleq B_{q}(\frac{r_{q}}{100})$, we have
\begin{equation}\label{equ interior estimate in balls away from singularity with weight half of n}
|\xi|^{[\frac{7}{2}]}_{1,\alpha,B}\leq C|L_{A}\xi|^{[\frac{9}{2}]}_{\alpha,B}+C|\xi|^{[\frac{7}{2}]}_{0,B}.
\end{equation}
It suffices to prove that (\ref{equ intermediate C0 estimate}) holds for any $\mu<\frac{1}{10}$. 

\begin{equation}\label{equ intermediate C0 estimate}
|\xi|^{[\frac{7}{2}]}_{0,B}\leq \mu|\nabla \xi|^{[\frac{9}{2}]}_{0,B}+C_{\mu}\frac{r_{q}^{\frac{9}{2}-\gamma}}{(-\log r_{q})^{b}}|\xi|_{L^{2}_{-\frac{9}{2}+\gamma,b}(B)}.
\end{equation}

Assuming (\ref{equ intermediate C0 estimate}), we go on to prove Theorem \ref{thm C0 est}. By 
(\ref{equ interior estimate in balls away from singularity with weight half of n}), we obtain 
\begin{equation}
|\xi|^{[\frac{7}{2}]}_{1,\alpha,B}\leq C|L_{A}\xi|^{[\frac{9}{2}]}_{\alpha,B}+\mu C|\nabla \xi|^{[\frac{9}{2}]}_{0,B}+C_{\mu}\frac{r_{q}^{\frac{9}{2}-\gamma}}{(-\log r_{q})^{b}}|\xi|_{L^{2}_{-\frac{9}{2}+\gamma,b}(B)}.
\end{equation}
\begin{equation}\label{equ 1 alpha estimate local version}
\textrm{Let}\ \mu C<\frac{1}{10},\  \textrm{then}\ |\xi|^{[\frac{7}{2}]}_{1,\alpha,B}\leq C|f|^{[\frac{9}{2}]}_{\alpha,B}+C_{\mu}\frac{r_{q}^{\frac{9}{2}-\gamma}}{(-\log r_{q})^{b}}|\xi|_{L^{2}_{-\frac{9}{2}+\gamma,b}(B)}.
\end{equation}
In particular, on the $C^{0}-$norm,  we have 
\begin{equation}
r^{\frac{7}{2}}_{q}|\xi|(q)\leq C|f|^{[\frac{9}{2}]}_{\alpha,B}+C_{\mu}\frac{r_{q}^{\frac{9}{2}-\gamma}}{(-\log r_{q})^{b}}|\xi|_{L^{2}_{-\frac{9}{2}+\gamma,b}(B)}.
\end{equation}
\begin{equation}\textrm{By definition we have}\  \
r_{q}^{-\frac{7}{2}}|f|^{[\frac{9}{2}]}_{\alpha,B}\leq \frac{C_{\mu}}{r^{\gamma-1}_{q}(-\log r_{q})^{b}}|f|^{(\gamma,b)}_{\alpha,B}.\ \ \ \ \ \ \ \ \ \ \ \ \ \ \ \ \ \ \ \ \ \ 
\end{equation}
\begin{equation}\label{equ 0 estimate local version}\textrm{Then}\ \
|\xi|(q)\leq \frac{C}{r^{\gamma-1}_{q}(-\log r_{q})^{b}}\{|f|^{(\gamma,b)}_{\alpha,B}+|\xi|_{L^{2}_{-\frac{9}{2}+\gamma,b}(B)}\}. \ \ \ \ \ \ \ \ \ \ \ \ \ \ \ \ \ \ \ \ \ \ 
\end{equation}

Since $q$ is an arbitrary point near the singularity, and $|f|^{(\gamma,b)}_{\alpha,B}\leq |f|^{(\gamma,b)}_{\alpha,M}$, the proof of Theorem \ref{thm C0 est} is complete (assuming (\ref{equ intermediate C0 estimate})). 

Now we prove (\ref{equ intermediate C0 estimate}) using simple interpolation. For any $x\in B$, consider $B_{x}(\mu d_{x})\ (\mu<\frac{1}{10})$, then there is a point $x_{0}$ in $B_{x}(\mu d_{x})$ such that 
\begin{eqnarray*}
& &|\xi|(x_{0})\leq [\oint_{B_{x}(\mu dx)}|\xi|^{2}dy]^{\frac{1}{2}}\leq \frac{C[\int_{B_{x}(\mu dx)}|\xi|^{2}r^{-9+2\gamma}(-\log r)^{2b}dy]^{\frac{1}{2}}]}{r_{x}^{-\frac{9}{2}+\gamma}((-\log r_{x})^{b})(\mu d_{x})^{\frac{7}{2}}}
\end{eqnarray*}

For any $x\in B$, $r_{x}$ is comparable to $r_{q}$, we then compute
\begin{eqnarray*}& & d^{\frac{7}{2}}_{x}|\xi|(x)
=d^{\frac{7}{2}}_{x}[|\xi|(x)-|\xi|(x_{0})]+d^{\frac{7}{2}}_{x}|\xi|(x_{0})
\\&\leq &2\mu d_{x}^{\frac{9}{2}}\sup_{y\in B_{x}(\mu d_{x})}|\nabla \xi|(y)+\frac{C_{\mu}r_{q}^{\frac{9}{2}-\gamma}}{(-\log r_{q})^{b}}|\xi|_{L^{2}_{-\frac{9}{2}+\gamma,b}(B)}.
\end{eqnarray*}

Replacing $"2\mu"$ by $\mu$, by  definition, we deduce (\ref{equ Thm C0 estimate 1}). Then (\ref{equ Thm C0 estimate 2}) is a Corollary of Proposition \ref{prop log weighted Schauder estimate} and (\ref{equ Thm C0 estimate 1}). \end{proof}

\subsection{Compact imbedding} 

\begin{lem}\label{lem compact imbedding} Suppose $p_{1}-1<p_{2}$, or  $p_{1}-1=p_{2}$ and $b_{1}>b_{2}$. Then for any ball $B$ such that $\partial B$ does not intersect the singularities, the imbedding $W^{1,2}_{p_{1},b_{1}}(B) \rightarrow L^{2}_{p_{2},b_{2}}(B)$  is compact. 

Consequently,  the imbedding $W^{1,2}_{p_{1},b_{1}} \rightarrow L^{2}_{p_{2},b_{2}}$ (of global spaces) is compact.
\end{lem}

\begin{proof} It suffices to assume $B$ is centred at a singular point $O$, and does not contain any other singular point. We only  prove the case when $p_{1}-1=p_{2}$ and $b_{1}=b_{2}+1$.  The proof in general  is  the same, except that we have to spell out more notations.  By definition, the imbedding from $W^{1,2}_{p_{1},b_{1}}(B)$ to $L^{2}_{p_{1}-1,b_{1}}(B)$ is bounded. For any concentric  and smaller ball $B(R)$, We have 
 \begin{equation}
\int_{B(R)}|u|^2w_{p_{2},b_{2}}dx \leq  \frac{1}{(-\log R)^{2}}\int_{B(R)}|u|^2w_{p_{1}-1,b_{1}}dx.
\end{equation}
Then suppose $|u|_{W^{1,2}_{p_{1},b_{1}}(B)}\leq C_{1}$,  we can  choose $R_{m}$ depending on $b$ and $C_{1}$ such that 
\begin{equation}
(\int_{B(R_{m})}|u|^2w_{p_{2},b_{2}}dx )^{\frac{1}{2}}\leq \frac{1}{m}. 
\end{equation}

Given a sequence $u_{i}$ such that $|u_{i}|_{W^{1,2}_{p_{1},b_{1}}(B)}\leq C$, using compactness of the imbedding $W^{1,2}(B\setminus B(R_{m}))\rightarrow L^{2}(B\setminus B(R_{m}))$, we obtain a Cauchy-subsequence $u_{m,j}$ in   $L^{2}_{p_{2},b_{2}}(B\setminus B(R_{m}))$ and 
$$|u_{m,j}|_{L^{2}_{p_{2},b_{2}}[B(R_{m})]}\leq \frac{1}{m}.$$
This means for any $m$ and the $u_{m,j}$, there is a $N_{m}$ such that  $j,l>N_{m}$ implies
\begin{equation}
|u_{m,j}-u_{m,l}|_{L^{2}_{p_{2},b_{2}}(B)}<\frac{4}{m}. 
\end{equation}
 When $b>a$, we choose $u_{b,j}$ as a subsequence of $u_{a,j}$. Then we can write $u_{aa}$ as $u_{m,j_{a}}$, $u_{bb}$ as $u_{m,j_{b}}$. Since they are subsequence of $u_{m,j}$, then 
$j_{a}>a$, $j_{b}>b$. Consider the diagonal sequence $u_{aa}$. By the discussion above,  for any $m$, when $a,\ b$ $>m+N_{m}$, we have
\begin{equation}
|u_{aa}-u_{bb}|_{L^{2}_{p_{2},b_{2}}(B)}<\frac{4}{m}.
\end{equation}
Thus the diagonal sequence $u_{aa}$ is a Cauchy-Sequence in $L^{2}_{p_{2},b_{2}}(B)$.\end{proof}
\subsection{Fredholm Theory \label{section Global Fredholm and Schauder Theory}}
\textbf{Using the local inverse in Corollary \ref{Cor solving model laplacian equation over the ball without the compact support RHS condition} and the above compact Sobolev-imbedding, it's almost standard to prove $L_{A}$ is Fredholm}.
\begin{equation}\label{equ deformation equation for A on  a small ball}
\textrm{We  consider the  linearized equation}\ L_{A}\xi=f\ \textrm{on a small ball}.
\end{equation}

\begin{prop}\label{prop solving the deformation equation locally for the noncone connection} Let $A,p,b$ be as in Theorem \ref{Thm Fredholm}. There is a $\tau_{0}>0$ such that for    any singular point $O$,  there exists a bounded linear map $Q_{A,p,b}$ from $L^{2}_{p,b}[B_{O}(\tau_{0})]$ to $W^{1,2}_{p,b}[B_{O}(\tau_{0})]$ with the following properties. 
 \begin{itemize}
\item $L_{A}Q_{A,p,b}=Id$ from  $L^{2}_{p,b}[B_{O}(\tau_{0})]$  to itself.
\item The bound on $Q_{A,p,b}$ is less than a  $\bar{C}$  as in Definition \ref{Def special constants}.
\end{itemize}
Consequently, equation  (\ref{equ deformation equation for A on  a small ball}) admits a solution $\xi$ such that 
\begin{equation}\label{equ lem local inverse near O is bounded 2}|\xi|_{W^{1,2}_{p,b}[B_{O}(\tau_{0})]}\leq \bar{C} |f|_{L^{2}_{p,b}[B_{O}(\tau_{0})]}.\ \textrm{In particular},\  \bar{C}\ \textrm{does not depend on}\ \tau_{0}.
\end{equation}

\end{prop}

\begin{rmk}Note that for  $W^{1,2}_{p,b}-$estimate, we don't have to shrink domain. This  is similar to  the case of standard Laplace equation (Theorem 9.9 in \cite{GilbargTrudinger}). 
\end{rmk}
\begin{proof}  The proof is  similar to that of Theorem 5.2 in \cite{GilbargTrudinger}. 

 Equation (\ref{equ deformation equation for A on  a small ball})
 is equivalent to 
 \begin{equation}\label{equ lem local inverse near O 1}
 \xi=Q_{p,b,A_{O}}f+Q_{p,b,A_{O}}P_{A_{O},A}\xi\triangleq \square \xi,\ P_{A_{O},A}=L_{A_{O}}-L_{A},
 \end{equation}
where $Q_{p,b,A_{O}}$ is the one in Corollary \ref{Cor solving model laplacian equation over the ball without the compact support RHS condition} ($\tau=\tau_{0}$).

By Definition \ref{Def Admissable connection with polynomial or exponential convergence} and Lemma \ref{lem bounding local perturbation of deformation operator }, for any $\epsilon_{0}$, when $\tau_{0}$ is small enough with respect to $\epsilon_{0}$ and $\psi$,  we obtain the following  $\textrm{for any}\ \xi_{1},\xi_{2}\in W^{1,2}_{p,b}[B_{O}(\tau_{0})]$.
\begin{equation}\label{equ lem local inverse near O 2}
|P_{A_{O},A}(\xi_{1}-\xi_{2})|_{L^{2}_{p,b}[B_{O}(\tau_{0})]}\leq \bar{C}\epsilon_{0}|\xi_{1}-\xi_{2}|_{W^{1,2}_{p,b}[B_{O}(\tau_{0})]}.
\end{equation}
By the  optimal $W^{1,2}_{p,b}-$ estimate of  $Q_{p,b,A_{O}}$, we obtain 
\begin{equation}\label{equ contract mapping for linear equation with Cepsilon0}
|\square (\xi_{1}-\xi_{2})|_{W^{1,2}_{p,b}[B_{O}(\tau_{0})]}\leq \bar{C}\epsilon_{0}|\xi_{1}-\xi_{2}|_{W^{1,2}_{p,b}[B_{O}(\tau_{0})]}.
\end{equation}

Let $ \bar{C}\epsilon_{0}<\frac{1}{4}$, $\square$ is a contract mapping from 
$W^{1,2}_{p,b}[B_{O}(\tau_{0})]$ to itself, then there must be a unique fixed point $\xi$ which   solves (\ref{equ deformation equation for A on  a small ball}). We  define 
$Q_{A,p,b}f$ as this $\xi$, the uniqueness  also implies \textbf{$Q_{A,p,b}$ is linear}. The condition $ \bar{C}\epsilon_{0}<\frac{1}{4}$, (\ref{equ lem local inverse near O 1}), (\ref{equ lem local inverse near O 2}), and Corollary \ref{Cor solving model laplacian equation over the ball without the compact support RHS condition} imply 
\begin{equation}\label{equ lem local inverse near O is bounded 1}
|\xi|_{W^{1,2}_{p,b}[B_{O}(\tau_{0})]}\leq |Q_{p,b,A_{O}}f|_{W^{1,2}_{p,b}}+ \frac{1}{4}|\xi|_{W^{1,2}_{p,b}[B_{O}(\tau_{0})]}.
\end{equation} 
Therefore (\ref{equ lem local inverse near O is bounded 2}) is true.\end{proof}
The following Lemma is well known.
\begin{lem}\label{lem parametrice in the smooth part}(Lemma 1.3.1 of \cite{Gilkey}, Theorem 4.3 of \cite{Lawson}). Under the same assumptions in Proposition \ref{prop solving the deformation equation locally for the noncone connection} on $A$,  for the $\tau_{0}$ obtained there, the $\tau_{0}-$admissible cover $\mathbb{U}_{\tau_{0}}$ satisfies the following property.  For any ball $B_{l}\in \mathbb{U}_{\tau_{0}}$ away from the singularity,   there exists a local parametrix $Q_{l}:\ L^{2}(B_{l})\rightarrow W^{1,2}(B_{l})$ such that 
$L_{A}Q_{l}=Id+K_{l}, \ K_{l} \ \textrm{is  compact}\ (L^{2}(B_{l})\rightarrow L^{2}(B_{l})).
$

\end{lem}
\begin{proof}For the reader's convenience we mention a little bit: any section $\xi\in L^{2}(B_{l})$ can be extended as $0$ outside $B_{l}$, thus gives a section in "$H^{0}$" ($L^{2}$) defined in page 6 of \cite{Gilkey}. Choosing the "$\phi$" in (a) of Lemma 1.3.1 in \cite{Gilkey} to be the standard cutoff function which is identically $1$ in  $B_{l}$ but vanishes outside $2B_{l}$,  Lemma 1.3.1 in \cite{Gilkey} says $\phi(L_{A}Q_{l}-Id)$  is an infinitely-smoothing operator (defined in first line of page 12 in \cite{Gilkey}). Then by restricting   $\phi(L_{A}Q_{l}-Id)\xi$ onto $B_{l}$, the proof is complete.\end{proof}
\begin{thm}\label{thm global parametrix}Let $A,p,b$ be as in Theorem \ref{Thm Fredholm},  there is a  bounded linear operator $Q\ (L^{2}_{p,b}\rightarrow W^{1,2}_{p,b})$ such that  $K_{p,b}\triangleq L_{A}Q-id\ (L^{2}_{p,b}\rightarrow L^{2}_{p,b})$ is compact.
\end{thm}
\begin{proof}  We consider the $\tau_{0}-$admissible cover  $\mathbb{U}_{\tau_{0}}$ in Definition (\ref{Def admissable open cover}), for the $\tau_{0}$ in Proposition \ref{prop solving the deformation equation locally for the noncone connection}. In  the  $B_{O_j}$'s (near the singular points), we use  the right inverse inverse  $Q_{j}$ constructed in Proposition \ref{prop solving the deformation equation locally for the noncone connection}. In the  $B_{l}$'s (away from the singular points), we use  the  $Q_{l}$'s  in Lemma \ref{lem parametrice in the smooth part}.   Then let
\begin{equation}\label{equ Thm global Sobolev parametrix}Q(f)=\Sigma_{j}\varphi_{j}Q_{j}(\chi_{j}f)+\Sigma_{l}\varphi_{l}Q_{l}(\chi_{l}f),\end{equation} where $\varphi_{j},\ \varphi_{l}$'s are the partition of unity of $\frac{\mathbb{U}_{\tau_{0}}}{5}$ (with co-centric balls with radius $\frac{1}{5}$ of the original one), and $\chi_{j}$ ($\chi_{l}$) are smooth  functions  such that 
\begin{displaymath}
\chi_{j}\ (\chi_{l}) =\left \{ \begin{array}{cc}
1, & \textrm{over}\ supp \varphi_{j}\subset \frac{B_{O_{j}}}{5}\ (supp \varphi_{l}\subset \frac{B_l}{5});\\
0,& \textrm{outside}\ \frac{B_{O_{j}}}{4}\ (\frac{B_{l}}{4}).
\end{array}\right.
\end{displaymath}
 Then $\varphi_{j}\chi_{j}=\varphi_{j}$ ($\varphi_{l}\chi_{l}=\varphi_{l}$).

For any smooth function $\varphi$ and section $\xi=\left [
\begin{array}{c}
   \sigma  \\
  a   \\
 \end{array}\right ]$, we calculate
 \begin{equation}\label{equ parametrix for LA: def of G}
 L_{A}(\varphi\xi)= \varphi L_{A}\xi+G(d\varphi,\xi),\ \textrm{where}
 \end{equation}
 \begin{equation}\label{equ Def G..} G(d\varphi,\xi)=\left [
\begin{array}{cc}
   0 & -\star (d\varphi \wedge \star a) \\
  d\varphi \wedge \sigma & \star (d\varphi \wedge a \wedge \psi)   \\
 \end{array}\right ]\ \textrm{is an algebraic operator. } \end{equation}  \begin{equation}\label{equ formula for LQ-id compact operator}\textrm{Thus}\
 L_{A}Q \xi=f+ \Sigma_{j}G[d\varphi_{j},Q_{j}(\chi_{j}f)]+\Sigma_{l}G[d\varphi_{l},Q_{l}(\chi_{l}f)]+\Sigma_{l}\varphi_{l}K_{l}(\chi_{l}f).
 \end{equation}
Since each $\varphi_{j}$ is smooth, Corollary \ref{Cor solving model laplacian equation over the ball without the compact support RHS condition} yields  
\begin{equation}
|\Sigma_{j}G[d\varphi_{j},Q_{j}(\chi_{j}f)|_{W^{1,2}_{p,b}}\leq C|f|_{L^{2}_{p,b}}.
\end{equation}
$\textrm{Lemma \ref{lem parametrice in the smooth part} implies}\ 
|\Sigma_{l}\varphi_{l}K_{l}(\chi_{l}f)|_{L_{p,b}^{1,2}}\leq C|f|_{L_{p,b}^{2}}$.
Let \begin{equation}\label{equ formula for Kpb}K_{p,b}(f)=\Sigma_{j}G[d\varphi_{j},Q_{j}(\chi_{j}f)]+\Sigma_{l}G[d\varphi_{l},Q_{l}(\chi_{l}f)]+\Sigma_{l}\varphi_{l}K_{l}(\chi_{l}f),\end{equation} we obtain 
$|K_{p,b} f|_{W^{1,2}_{p,b}}\leq C|f|_{L^{2}_{p,b}}.$

By  Lemma \ref{lem compact imbedding}, $K_{p,b}$ is compact from  $L^{2}_{p,b}$ to itself.  \end{proof}

\begin{proof}[\textbf{Proof of Theorem \ref{Thm Fredholm}}:]  By Theorem \ref{thm global parametrix},   using the argument in page 50 of \cite{Donaldson},  and in the proof of Theorem 4  in \cite{Evans},  the proof of the  Sobolev-theory part of Theorem \ref{Thm Fredholm} is complete.

  The Hybrid part in Theorem \ref{Thm Fredholm} is a direct corollary of  the  crucial $C^{0}-$  estimate in Theorem \ref{thm C0 est}. Suppose  $\bar{f}$ is in cokernel, (\ref{equ coker contained in a bigger space}) yields $L_{A}^{\star}\bar{f}=0$ away from the singularities. Then Theorem \ref{thm C0 est} for $L^{\star}_{A}$ ($\gamma=\frac{9}{2}+p$), with "$f$" there being $0$ and "$\xi$" there being $\bar{f}$) yields
   \begin{equation}
 \label{equ Thm paramatix 1}
   |\bar{f}|_{C^{1,\alpha}_{(\frac{7}{2}+p,b)}(M)}\leq C |\bar{f}|_{L^{2}_{p,b}}.
\end{equation} 
 This means  $CokerL_{A}\subset N_{p,b}$ (Definition \ref{Def Hybrid spaces}). 

Since $Imaga L_{A}|_{W^{1,2}_{p,b}}$ is closed in $L^{2}_{p,b}$, for any $f\in N_{p,b}$ , we have a resolution into parallel and perpendicular components with respect to  $Image L_{A}|_{W^{1,2}_{p,b}}$:
\begin{equation}\label{equ f=fparallel + f perp}
f=f^{\parallel}+f^{\perp},\ f^{\perp}\in CokerL_{A}.
\end{equation}  
The estimate (\ref{equ Thm paramatix 1}) means $f^{\perp}\in N_{p,b}$ and  
  \begin{equation}
 |f^{\perp}|_{C^{1,\alpha}_{\frac{7}{2}+p,b}}\leq 
C|f^{\perp}|_{L^{2}_{p,b}}\leq C|f|_{L^{2}_{p,b}}\leq C|f|_{N_{p,b}}.
 \end{equation}
The decomposition (\ref{equ f=fparallel + f perp}) implies $f^{\parallel}\in N_{p,b}$. Thus we have proved
  \begin{lem}\label{lem regularity of projection of f onto Image of L}Suppose $f\in N_{p,b}$, then 
 \begin{equation}
 |f^{\parallel}|_{N_{p,b}}\leq C|f|_{N_{p,b}},\  |f^{\perp}|_{N_{p,b}}\leq C|f|_{N_{p,b}}.
 \end{equation}
 \end{lem}
 
 Lemma \ref{lem regularity of projection of f onto Image of L} and Theorem \ref{thm existence of a good solution when f is in the image} yields a solution to $L_{A}\xi=f^{\parallel}$ which is orthogonal to the kernel and 
 \begin{equation}\label{equ proof of Thm Fredholm 1}
 |\xi|_{L^{2}_{p-1,b}}\leq  |\xi|_{W^{1,2}_{p,b}}\leq C|f|_{L^{2}_{p,b}}.
 \end{equation}
Theorem \ref{thm C0 est} ($\gamma=\frac{7}{2}+p$) and (\ref{equ proof of Thm Fredholm 1}) gives 
\begin{equation}\label{equ proof of Thm Fredholm 2}|\xi|_{2,\alpha,M}^{(\frac{5}{2}+p,b)}\leq C|f|_{N_{p,b}}. \ \textrm{Therefore}\ 
 |\xi|_{H_{p,b}}\leq C|f|_{N_{p,b}}. 
 \end{equation}
 The proof of the Hybrid-spaces part of Theorem \ref{Thm Fredholm} is complete.\end{proof}
 \begin{rmk}\label{rmk Jpb}(The pre-image space $J_{p,b}$) We define 
 \begin{equation}
 J_{p,b}=\{\xi \in H_{p,b}| \xi\perp kerL_{A}\}.
 \end{equation}
 Theorem \ref{Thm Fredholm} says $L_{A}$ is an isomorphism from $J_{p,b}$ to $Image L_{A}|_{H_{p,b}}\subset N_{p,b}$.  Denoting  $|\xi|_{H_{p,b}}$ as $|\xi|_{J_{p,b}}$ when $\xi\in J_{p,b}$,  we rewrite  (\ref{equ proof of Thm Fredholm 2}) as
 \begin{equation}\label{equ in Jpb LA  inverse is bounded}
 |L_{A}^{-1}f|_{J_{p,b}}\leq C|f|_{N_{p,b}}, \ \textrm{for any}\ f\in Image L_{A}|_{H_{p,b}}.
 \end{equation}
  \end{rmk}
  \section{Perturbation \label{section Perturbation}}
In this section we don't need the log-weight, thus we  conform to the abbreviation convention in Definition \ref{Def abbreviation of notations for spaces} i.e. \textbf{there will be no "$b$" in the symbols of the function spaces.} We prove the most precise version of Theorem \ref{Thm Deforming instanton simple version}.
\begin{thm}\label{Thm Deforming instanton}
 Let $M$ be a $7-$manifold with a smooth $G_{2}-$structure $(\phi,\psi)$, and $E\rightarrow M$ be an  admissible $SO(m)-$bundle defined away from finitely-many points (Definition \ref{Def the bundle xi}). Suppose $A$ is an admissible  $\psi-$instanton of order $4$ (Definition \ref{Def Admissable connection with polynomial or exponential convergence} and (\ref{equ instanton equation for A})).  Then for any $A-$generic $p\in (-\frac{5}{2},-\frac{3}{2})$ (Definition \ref{Def A generic}), there exists a $\delta_{0}>0$ with the following property.

 Suppose   $cokerL_{A}=\{0\}$ (Theorem \ref{thm characterizing cokernel}), 
for  any admissible $\delta_{0}$-deformation $(\underline{\phi},\underline{\psi})$ of  $(\phi,\psi)$ (Definition \ref{Def deformation of the G2 structure}),  there exists a $\underline{\psi}-G_{2}$ monopole $(\underline{A},\sigma)$  such that  $\underline{A}$ satisfies Condition ${\circledS_{A,p}}$ (Definition \ref{Def condition SAp}). In particular, the tangent  connection of $\underline{A}$ at each singular point is the same as that of $A$. When $\underline{\psi}$ is closed,   $\underline{A}$ is a  $\underline{\psi}-$instanton.
\end{thm}
\begin{prop}\label{prop C1 global bound for the error term} Under the same conditions as in Theorem \ref{Thm Deforming instanton}, for any $\lambda_{1}>0$ and $\epsilon_{1}$, there is a $\delta_{0}$ with the following property. For any admissible $\delta_{0}-$deformation  $\underline{\psi}$ of $\psi$, we have 
\begin{equation}\label{equ in prop C1 global bound for the error term}
|\star_{\underline{\psi}}(F_{A}\wedge \underline{\psi})|_{C^{1,\alpha}_{{(1+\lambda_{1})}}(M)}< \epsilon_{1},\ \textrm{for any}\ 0<\alpha\leq 1.
\end{equation}
Thus for any $\lambda_{2}>0$ and $\epsilon_{1}$, the following is true when  $\delta_{0}$ is small enough.
\begin{equation}\label{equ in prop C1 global bound for the error term 1}
|\star_{\underline{\psi}}(F_{A}\wedge \underline{\psi})|_{N_{-\frac{5}{2}+\lambda_{2}}}< \epsilon_{1}. 
\end{equation}
\end{prop}
\begin{proof} The essential issue can be elaborated by  the $C^{0}-$estimate. The instanton condition (\ref{equ instanton equation for A}) implies 
  \begin{equation}\label{equ decomposation of the full  error term}
  F_{A}\wedge  \underline{\psi}=  F_{A}\wedge  (\underline{\psi}-\psi).
  \end{equation}
  
  Let $\rho_{0}>0$ be small enough such that $B_{O}(\rho_{0})$ is within the coordinate chart near $O$.  When $r\leq \rho_{0}$, the admissible condition implies $\underline{\psi}-\psi=[\underline{\psi}-\underline{\psi}(O)]+[\psi(O)-\psi].\ \textrm{Hence}\ \ |\underline{\psi}-\psi|<Cr.$ $\textrm{Adjust}\ \rho_{0}\ \textrm{such that}\ C\rho_{0}^{\lambda_{1}}<\frac{\epsilon_{1}}{2},\ \textrm{we find}$
  \begin{equation}\label{equ C0 bound on the error step 1}
|F_{A}\wedge  (\underline{\psi}-\psi)|\leq \frac{C}{r}\leq \frac{C\rho_{0}^{\lambda_{1}}}{r^{1+\lambda_{1}}}<\frac{\epsilon_{1}}{r^{1+\lambda_{1}}}. 
  \end{equation}
  When $r\geq \rho_{0}$, still by the condition in Proposition \ref{prop C1 global bound for the error term}, we have
    \begin{equation}\label{equ C0 bound on the error step 2}
   |F_{A}\wedge  (\underline{\psi}-\psi)|\leq C\delta_{0}\rho^{-2}_{0}=C\delta_{0}(\frac{\epsilon_{1}}{2})^{-\frac{2}{\lambda_{1}}}<\epsilon_{1},\  \textrm{when}\ \delta_{0}\ \ \textrm{is small enough}.
  \end{equation}
  
 Thus we obtain the $C^{0}-$bound 
  \begin{equation}\label{equ C0 bound on the error}
   |F_{A}\wedge  \underline{\psi})|_{C^{0}_{{(1+\lambda_{1})}}(M)}\approx|\star_{\underline{\psi}}(F_{A}\wedge  \underline{\psi})|_{C^{0}_{{(1+\lambda_{1})}}(M)}< C\epsilon_{1}.
  \end{equation}
 where  "$\approx$" means equivalent up to a constant in the sense of Definition \ref{Def Dependence of the constants}.
  
The  bounds on $|\nabla_{A} \star_{\underline{\psi}}(F_{A}\wedge  \underline{\psi})|_{C^{0}_{{(2+\lambda_{1})}}(M)}$ and
 $|\nabla^{2}_{A} \star_{\underline{\psi}}(F_{A}\wedge  \underline{\psi})|_{C^{0}_{{(3+\lambda_{1})}}(M)}$ are similar.  For the reader's convenience, we still do the gradient bound. 
  \begin{eqnarray}& &\nabla_{A,\psi}\star_{\underline{\psi}}(F_{A}\wedge  \underline{\psi})=\nabla_{A,\psi}\star_{\underline{\psi}}(F_{A}\wedge  [\underline{\psi}-\psi]) \\&=&[\nabla_{\psi}(\star_{\underline{\psi}})](F_{A}\wedge  [\underline{\psi}-\psi])+\star_{\underline{\psi}}(F_{A}\wedge \nabla_{\psi}[\underline{\psi}-\psi])+\star_{\underline{\psi}}(\nabla_{A,\psi}F_{A}\wedge[\underline{\psi}-\psi]) \nonumber
 \end{eqnarray} 
 
For the first term,  by (\ref{equ C0 bound on the error}) and that $|\underline{\psi}|_{C^{5}(M)}\leq C$, we have 
 \begin{eqnarray}
 & &|[\nabla_{\psi}(\star_{\underline{\psi}})](F_{A}\wedge   [\underline{\psi}-\psi])|_{C^{0}_{{(2+\lambda_{1})}}(M)}\leq |[\nabla_{\psi}(\star_{\underline{\psi}})](F_{A}\wedge   [\underline{\psi}-\psi])|_{C^{0}_{{(1+\lambda_{1})}}(M)}\nonumber
 \\&\leq & C\epsilon_{1}.
 \end{eqnarray}
 
 For the second term, we have the following cheap estimate 
  \begin{equation}
|\star_{\underline{\psi}}(F_{A}\wedge \nabla_{\psi}[\underline{\psi}-\psi])|\leq \frac{C\delta_{0}}{r^{2}}.
 \end{equation}
 By exactly the trick (relaxing the weight a little bit) from (\ref{equ C0 bound on the error step 1}) to (\ref{equ C0 bound on the error}), we obtain for any $\lambda_{1}>0$ that
  \begin{equation}
|\star_{\underline{\psi}}(F_{A}\wedge \nabla_{\psi}[\underline{\psi}-\psi])|_{C^{0}_{{(2+\lambda_{1})}}(M)}< \epsilon_{1}\ \textrm{when}\ \delta_{0}\ \textrm{is small enough}.
 \end{equation} 
 In the same way it follows that the  $C^{0}_{{(2+\lambda_{1})}}(M)-$norm of the third term is less than $\epsilon_{1}$ (using $|\nabla_{A,\psi}F_{A}|\leq \frac{C}{r^{3}}$). Then we obtain 
\begin{equation}\label{equ error bound 2}|\nabla_{A,\psi}\star_{\underline{\psi}}(F_{A}\wedge  \underline{\psi})|_{C^{0}_{{(2+\lambda_{1})}}(M)}< C\epsilon_{1}\  \textrm{when}\ \delta_{0}\ \ \textrm{is small enough}. 
\end{equation}
\begin{equation}\label{equ error bound 3}\textrm{The proof of the Hessian estimate }\ |\nabla_{A,\psi}^{2}\star_{\underline{\psi}}(F_{A}\wedge  \underline{\psi})|_{C^{0}_{{(3+\lambda_{1})}}(M)}< C\epsilon_{1}
\end{equation}
is absolutely the same, except that we have one more negative power on $r$. 

By replacing "$C\epsilon_{1}$" by $\epsilon_{1}$ (since $\epsilon_{1}$ is arbitrary and $C$ does not depend on it), the estimates (\ref{equ C0 bound on the error}), (\ref{equ error bound 2}), (\ref{equ error bound 3}) amount to
\begin{equation}
|\star_{\underline{\psi}}(F_{A}\wedge \underline{\psi})|_{C^{2}_{{(1+\lambda_{1})}}(M)}< \epsilon_{1}.
\end{equation}

By  Lemma \ref{lem C1 C0 intepolate Calpha}, the proof of (\ref{equ in prop C1 global bound for the error term}) is complete. Using Definition \ref{Def Hybrid spaces}, the proof of (\ref{equ in prop C1 global bound for the error term 1}) is done by letting $\lambda_{1}=\frac{\lambda_{2}}{2}$ in (\ref{equ in prop C1 global bound for the error term})  and $\delta_{0}$ small enough. 
 \end{proof}

The monopole  equation with respect to $\underline{\psi}$ is  \begin{equation}\label{equ instantonequation without cokernel}\star_{\underline{\psi}}(F_{A+a}\wedge \underline{\psi})+d_{A+a}\sigma=0, \textrm{with gauge fixing equation}\ d^{\star_{\underline{\psi}}}_{A}a=0
 \end{equation}
 It's equivalent to 
 \begin{equation}\label{equ instanton equation without cokernel without LA}
\star_{\underline{\psi}}(d_{A }a\wedge \underline{\psi})+\frac{1}{2}\star_{\underline{\psi}}([a,a]\wedge \underline{\psi})+d_{A}\sigma+[a,\sigma]+\star_{\underline{\psi}}(F_{A}\wedge \underline{\psi})=0
 \end{equation}
with gauge fixing. Equation (\ref{equ instantonequation without cokernel}) and (\ref{equ instanton equation without cokernel without LA}) can be written as
 \begin{equation}\label{equ instanton column vector equation without cokernel}
 L_{A,\underline{\psi}}[\begin{array}{c}
\sigma   \\
  a   
\end{array}]=\left[\begin{array}{c}
0  \\
  -\frac{1}{2}\star_{\underline{\psi}}([a,a]\wedge \underline{\psi})
 -[a,\sigma]-\star_{\underline{\psi}}(F_{A}\wedge \underline{\psi})
\end{array}\right ]
 \end{equation}
 \begin{lem}\label{lem Launderline psi is an isomorphism from Jp to Np} Under the same conditions as in Theorem \ref{Thm Deforming instanton}, $L_{A,\underline{\psi}}$ is an isomorphism from $J_{p}$ to $N_{p}$,  and the bounds (on itself and the inverse of it) are uniform for all admissible $\delta_{0}-$deformations of $\psi$, when $\delta_{0}$ is small enough  with respect to the data in Definition \ref{Def Dependence of the constants}. 
 \end{lem}
 
 \begin{proof}The proof is exactly as that of Proposition \ref{prop solving the deformation equation locally for the noncone connection}, for the reader's convenience we include the crucial detail. 
  \begin{equation}\label{equ linearized equation without cokernel }
 \textrm{For any} \ f\in N_{p},\ \textrm{the equation}\ L_{A,\underline{\psi}}\left[\begin{array}{c}
\sigma  \\
a
\end{array}\right ] =f\ \textrm{is equivalent to}\
 \end{equation}
\begin{equation}\label{equ linearized equation for iteration  no cokernel}
 [\begin{array}{c}
\sigma   \\
  a   
\end{array}]=L^{-1}_{A}(L_{A}-L_{A,\underline{\psi}})[\begin{array}{c}
\sigma   \\
  a   
\end{array}]+L^{-1}_{A}f.
\end{equation}

The right hand side of (\ref{equ linearized equation for iteration  no cokernel}) is a contract mapping in terms of $[\begin{array}{c}
\sigma   \\
  a   
\end{array}]$ from $J_{p}$ to itself, thus iteration implies (\ref{equ linearized equation without cokernel }) can be  uniquely  solved in $J_{p}$. The bound on $L^{-1}_{A}$ follows from the last part in the proof of Proposition \ref{prop solving the deformation equation locally for the noncone connection}.
 \end{proof}
 
 By Lemma \ref{lem Launderline psi is an isomorphism from Jp to Np}, (\ref{equ instanton column vector equation without cokernel}) is  equivalent to the following equation
\begin{equation}
\left[\begin{array}{c}
\sigma  \\
a
\end{array}\right ]=L_{A,\underline{\psi}}^{-1}P\left[\begin{array}{c}
\sigma  \\
a
\end{array}\right ],
\end{equation}
where $
P\left[\begin{array}{c}
\sigma  \\
a
\end{array}\right ]$ means the right hand side of (\ref{equ instanton column vector equation without cokernel}).\begin{equation}\label{equ first iteration}
\textrm{The first iteration is}\ L_{A,\underline{\psi}}^{-1}P\left[\begin{array}{c}
0  \\
0\\

\end{array}\right ]=L_{A,\underline{\psi}}^{-1}\left[\begin{array}{c}
0  \\
-\star_{\underline{\psi}}(F_{A}\wedge \underline{\psi})
\end{array}\right ].
\end{equation}

For any $p$ as in Theorem \ref{Thm Deforming instanton}, there is a $\lambda_{2}$ such that $p>-\frac{5}{2}+\lambda_{2}$, thus  Proposition \ref{prop C1 global bound for the error term} implies the following bound on the first iteration   \begin{equation}\label{equ first iteration is small when cokernel is present}
|L_{A,\underline{\psi}}^{-1}P\left[\begin{array}{c}
0  \\
0\\
\end{array}\right ]|_{J_{p}}\leq C|\star_{\underline{\psi}}(F_{A}\wedge \underline{\psi})|_{N_{p}}< C\epsilon_{1}\ \textrm{when}\   \delta_{0}\ \textrm{is small enough}.
\end{equation}

\begin{proof}[\textbf{Proof of Theorem \ref{Thm Deforming instanton}}:] To solve the taming-pair equation (\ref{equ instantonequation without cokernel}), it suffices to show that $L_{A,\underline{\psi}}^{-1}P$ is a contract mapping when restricted to a small enough neighbourhood of $0$ in $J_{p}$, then iteration as in section 3 of  \cite{Myself2013} implies the existence of a unique solution close enough to $0$.

Since $p<-\frac{3}{2}$, this is an easy consequence of the multiplicative property of $J_{p}\ (H_{p}),N_{p}$ in Lemma \ref{lem multiplicative property of hybrid spaces}, and that  the 2 terms  \begin{equation}\label{equ quadratic terms 1 in the instanton equation with cokernel}-\frac{1}{2}\star_{\underline{\psi}}([a,a]\wedge \underline{\psi}])\  \textrm{and}\ -[a,\sigma]\end{equation} are quadratic. For the reader's convenience, we  include the crucial detail. 
\begin{eqnarray}\label{eqnarray 1 in proof of Main thm without cokernel}
\textrm{We compute} & &  P\left[\begin{array}{c}
\sigma_{1}  \\
a_{1}
\end{array}\right ]-   
P\left[\begin{array}{c}
\sigma_{2} \\
a_{2}
\end{array}\right]
\\&=&-\frac{1}{2}\{\star_{\underline{\psi}}([a_{1}-a_{2},a_{1}]\wedge \underline{\psi})+\frac{1}{2}\star_{\underline{\psi}}([a_{2},a_{2}-a_{1}]\wedge \underline{\psi})\}\nonumber\\& &-\{ [a_{1}-a_{2},\sigma_{1}]+[a_{2},\sigma_{2}-\sigma_{1}]\}\nonumber
\end{eqnarray}
Then Lemma \ref{lem multiplicative property of hybrid spaces} implies
\begin{equation}\label{equ 1 in proof of Main thm without cokernel}
|P\left[\begin{array}{c}
\sigma_{1}  \\
a_{1}
\end{array}\right ]-   
P\left[\begin{array}{c}
\sigma_{2} \\
a_{2}
\end{array}\right]|_{N_{p}}
\leq C |\left[\begin{array}{c}
\sigma_{1}  \\
a_{1}
\end{array}\right]-   
\left[\begin{array}{c}
\sigma_{2} \\
a_{2}
\end{array}\right]|_{J_{p}}
\times (\left|\begin{array}{c}
\sigma_{1}  \\
a_{1}
\end{array}\right|_{J_{p}}+  
\left|\begin{array}{c}
\sigma_{2} \\
a_{2}
\end{array}\right|_{J_{p}})
\end{equation}

Thus, letting $\epsilon_{1}$ small enough with respect to the "$C$" above and the "$C$" in (\ref{equ in Jpb LA  inverse is bounded}), the condition $(\left|\begin{array}{c}
\sigma_{1}  \\
a_{1}
\end{array}\right|_{J_{p}}+  
\left|\begin{array}{c}
\sigma_{2} \\
a_{2}
\end{array}\right|_{J_{p}})<\epsilon_{1}$ implies 
\begin{equation}
|L_{A,\underline{\psi}}^{-1}P\left[\begin{array}{c}
\sigma_{1}  \\
a_{1}
\end{array}\right ]-   
L_{A,\underline{\psi}}^{-1}P\left[\begin{array}{c}
\sigma_{2} \\
a_{2}
\end{array}\right]|_{J_{p}}
\leq \frac{1}{2} |\left[\begin{array}{c}
\sigma_{1}  \\
a_{1}
\end{array}\right]-   
\left[\begin{array}{c}
\sigma_{2} \\
a_{2}
\end{array}\right]|_{J_{p}}.
\end{equation}

The proof of the contract mapping is complete.

Denote  $A+a$  as $\underline{A}$. When $\underline{\psi}$ is closed,  note that by applying 
$d^{\star_{\underline{\phi}}}_{\underline{A}}$ to  (\ref{equ instanton equation without cokernel without LA}) away from the singularity, we obtain $d^{\star_{\underline{\phi}}}_{\underline{A}}d_{\underline{A}}\sigma=0$. Then we choose the cut-off function $\eta_{\epsilon}$  as in (\ref{equ cut-off function bound near the singular point}), and calculate
\begin{equation}\label{equ Thm deforming instanton 3}
0= \int_{M}<d^{\star_{\underline{\phi}}}_{\underline{A}}d_{\underline{A}}\sigma, \eta_{\epsilon}\sigma>dV=\int_{M}<d_{\underline{A}}\sigma,(d\eta_{\epsilon})\wedge\sigma>+\int_{M} \eta_{\epsilon}|d_{\underline{A}}\sigma|^{2}dV.
\end{equation}
$\sigma\in J_{p}$ implies $|d_{\underline{A}}\sigma|\leq \frac{C_{\sigma}}{r^{\frac{7}{2}+p}},\ |\sigma|\leq \frac{C}{r^{\frac{5}{2}+p}}$. Then 
\begin{equation}\label{equ Thm deforming instanton 4}
|\int_{M}<d_{\underline{A}}\sigma,(d\eta_{\epsilon})\wedge\sigma>|\leq C_{\sigma}\int_{B(\epsilon)-B(2\epsilon)}\frac{1}{r^{\frac{7}{2}+p}}\frac{1}{r^{\frac{5}{2}+p}}\frac{1}{r}\leq C\epsilon^{-2p}\rightarrow 0\ \textrm{as}\ \epsilon\rightarrow 0.
\end{equation}

Let $\epsilon\rightarrow 0$ in (\ref{equ Thm deforming instanton 3}), monotone convergence theorem and  (\ref{equ Thm deforming instanton 4}) implies $\int_{M}|d_{\underline{A}}\sigma|^{2}dV=0.$ This means $d_{\underline{A}}\sigma=0$, and thus (\ref{equ instantonequation without cokernel}) says $\underline{A}$ is a $\underline{\psi}-$instanton.
\end{proof}

\begin{proof}[\textbf{Proof of Theorem \ref{Thm Deforming local instanton}}:] Using  Corollary \ref{Cor solving model laplacian equation over the ball without the compact support RHS condition}, this is much easier than  Theorem \ref{Thm Deforming instanton}.  The crucial trick is to cut off the nonlinear term and error in (\ref{equ instantonequation without cokernel}) [(\ref{equ instanton column vector equation without cokernel})] i.e. let $\xi$ denote $\left[\begin{array}{c} a \\
  \sigma
\end{array}\right]$, we should consider 
the  equation  \begin{equation}\label{equ instanton column vector equation without cokernel cutted off}
 L_{A_{O},\underline{\psi}}\xi=\chi\left[\begin{array}{c} 0 \\
  \xi\otimes \xi-\star_{\underline{\psi}}(F_{A_{O}}\wedge \underline{\psi})
\end{array} ]\right. ,\ \xi\otimes \xi=-\frac{1}{2}\star_{\underline{\psi}}([a,a]\wedge \underline{\psi})
 -[a,\sigma], \end{equation}
$\chi$ is the standard cut-off function $\equiv 1$ in $B_{O}(\frac{1}{4})$ and $\equiv 0$ outside  $B_{O}(\frac{5}{16})$. Using 
(\ref{equ 0 estimate local version}) with $\gamma=\frac{7}{2}+p$ (\textbf{for any $p$ as in Theorem \ref{Thm Deforming instanton}}), and proof of Theorem 1 in \cite{Nirenberg} [in   $B_{O}(\frac{7}{16})$ and near $\partial B_{O}(\frac{7}{16})$], we obtain 
\begin{equation}\label{equ Thm deformation local instanton 1} 
|\xi|^{\{\frac{5}{2}+p\},\frac{7}{2}}_{2,\alpha,B_{O}(\frac{7}{16})}\leq \bar{C}\{|L_{A_{O}}\xi|^{\{\frac{7}{2}+p\},\frac{9}{2}}_{1,\alpha,B_{O}(\frac{7}{16})}+|\xi|_{L^{2}_{p-1}[B_{O}(\frac{7}{16})]}\}\ (\textrm{see Definition}\ \ref{Def local Schauder adapted to local perturbation}).
\end{equation}
We define
\begin{itemize}
\item $\bar{H}_{p}=\{\xi\ \textrm{is}\ C^{2,\alpha}\ \textrm{away from}\ O|\ |\xi|_{W^{1,2}_{p}(B_{O}(\frac{7}{16}))}+|\xi|^{\{\frac{5}{2}+p\},\frac{7}{2}}_{2,\alpha,B_{O}(\frac{7}{16})}<\infty\}$ and 
\item $\bar{N}_{p}=\{\xi\ \textrm{is}\ C^{1,\alpha}\ \textrm{away from}\ O|\ |\xi|_{L^{2}_{p}(B_{O}(\frac{7}{16}))}+|\xi|^{\{\frac{7}{2}+p\},\frac{9}{2}}_{1,\alpha,B_{O}(\frac{7}{16})}<\infty\}$.
\end{itemize}

Using Corollary \ref{Cor solving model laplacian equation over the ball without the compact support RHS condition} and (\ref{equ Thm deformation local instanton 1}), the $L_{A_{O},\psi_{0}}:\ \bar{H}_{p}\rightarrow \bar{N}_{p}$ is inverted by $Q_{p,A_{O}}$ which is linear and  bounded by a $\bar{C}$ as in Definition \ref{Def special constants}. Moreover, the advantage of cutting-off the monopole equation is 
\begin{clm}\label{clm multiplicative property of cutting off} We have $|\chi f|_{C^{1,\alpha}_{\{\frac{7}{2}+p\},\frac{9}{2}}[B_{O}(\frac{7}{16})]}\leq C |f|_{C^{1,\alpha}_{(\frac{7}{2}+p)}[B_{O}(\frac{6}{16})]}$. Consequently, 
$$|\chi \xi_{1}\otimes \xi_{2}|_{\bar{N}_{p}}\leq C|\xi_{1}|_{\bar{H}_{p}}|\xi_{2}|_{\bar{H}_{p}},\ \textrm{where}\ \otimes\ \textrm{and}\ \chi\ \textrm{are as in}\ (\ref{equ instanton column vector equation without cokernel cutted off}).$$
\end{clm}
The proof of the above is similar to Lemma \ref{lem multiplicative property of hybrid spaces}, the cutoff function $\chi$ plays the key role. \textbf{Claim \ref{clm multiplicative property of cutting off} means we can avoid the boundary estimate near $\partial B$}, $B$ as in Theorem \ref{Thm Deforming local instanton}. With the help of Proposition \ref{prop C1 global bound for the error term},   by  going through the proof of Lemma \ref{lem Launderline psi is an isomorphism from Jp to Np}, Theorem \ref{Thm Deforming instanton}, and  
\begin{itemize}
\item replacing the $A$ and $\psi$ (Proposition \ref{prop C1 global bound for the error term}) by  $A_{O}$ and $\psi_{0}$,  
\item replacing  the $L_{A,\underline{\psi}}$, $L_{A}$ in Lemma \ref{lem Launderline psi is an isomorphism from Jp to Np} by  $L_{A_{O},\underline{\psi}}$, $L_{A_{O}}$ respectively, 
\item replacing  the $L^{-1}_{A,\underline{\psi}}$ in proof of Theorem \ref{Thm Deforming instanton} by $L^{-1}_{A_{O},\underline{\psi}}$,
\item   replacing  the $J_{p},N_{p}$ in proof of Theorem \ref{Thm Deforming instanton} by 
$\bar{H}_{p},\bar{N}_{p}$ respectively, 
\end{itemize}
we obtain a solution $\xi$ to (\ref{equ instanton column vector equation without cokernel cutted off}) in $\bar{H}_{p}$, \textbf{for any $p$ as in Theorem \ref{Thm Deforming instanton}}. Since $\chi\equiv 1$ in $B_{O}(\frac{1}{4})$, $A_{O}+a$ solves  the monopole-equation therein [see (\ref{equ instanton column vector equation without cokernel}) and the discussion above it]. The proof of Theorem \ref{Thm Deforming local instanton} is complete.
\end{proof}
\section{Characterizing the cokernel \label{section Characterizing the cokernel}}

The formal adjoint of $L_{A}$ is
\begin{equation}\label{equ LA star formula} L_{A}^{\star}f=L_{A}f+G(\frac{d\omega_{p,b}}{\omega_{p,b}},f)\ \textrm{defined away from the singularities}. 
\end{equation}

The cokernel is defined as $Image^{\perp}L_{A}$ i.e.
\begin{equation}\label{equ Def of Coker} cokerL_{A}\triangleq \{f\in L^{2}_{p,b}|\int_{M}<L_{A}\xi,f>w_{p,b}dV=0\ \textrm{for all}\ \xi\in W^{1,2}_{p,b}\}.\end{equation}
Taking the test sections $\xi$ in (\ref{equ Def of Coker}) as smooth sections supported away from the singularities, using Theorem \ref{thm C0 est} (for $\gamma=\frac{9}{2}+p$), we deduce 
\begin{equation}\label{equ coker contained in a bigger space}
cokerL_{A}\subset \{f\in N_{p,b}|L^{\star}_{A}f=0\ \textrm{away from the singularities}\}.
\end{equation}
However, a priorily, there is no guarantee that the right hand side of (\ref{equ coker contained in a bigger space}) is finite-dimensional. \textbf{Fortunately, when $p$ is $A-$generic, the "blowing-up" rate of elements in  cokernel can be improved as follows.}
\begin{thm}\label{thm characterizing cokernel} Let $A,p,b$ be as in Theorem \ref{Thm Fredholm}. For all $0<\mu<\vartheta_{1-p}$,  
\begin{equation}\label{equ thm characterizing cokernel}
cokerL_{A}=\{f\in C^{1,\alpha}_{(\frac{7}{2}+p-\mu)}(M)| L^{\star}_{A}f=0\ \textrm{away from the singularites}\}.
\end{equation}
Moreover, $|f|^{(\frac{7}{2}+p-\mu)}_{1,\alpha,M}\leq C|f|_{L^{2}_{p,b}}$ for any $f\in cokerL_{A}$. In particular, if  $f$ satisfies the two conditions in the parentheses of (\ref{equ thm characterizing cokernel})  for some $\mu>0$, then $f$ satisfies them for all $\mu<\vartheta_{1-p}$.

\end{thm} 
\begin{proof} By Theorem 5 of Appendix D.5 in \cite{Evans},  $cokerL_{A}\subset ker(Id-K^{\star}_{p,b})|_{L^{2}_{p,b}}$.  Lemma \ref{lem  boostrap Kstar} implies any $f\in ker(Id-K^{\star}_{p,b})$ is actually in $L^{2}_{p-\mu,b}$ with uniform bound in terms of the $L^{2}_{p,b}-$ norm of $f$. Then Theorem \ref{thm characterizing cokernel} is a directly corollary of  (\ref{equ coker contained in a bigger space}), Lemma \ref{lem  boostrap Kstar}, \ref{lem cokernel contains the more vanishing sections}, and Theorem \ref{thm C0 est} (for $L_{A}^{\star}$ by taking $\gamma=\frac{9}{2}+p-\mu$).
\end{proof}

The  crucial observation is that in the setting of Theorem \ref{thm characterizing cokernel}, we have 
\begin{equation}\label{equ Crucial observation that Q is the same for neary by weights} 
Q_{A_{O_{j}}, p,b}f=Q_{A_{O_{j}}, p+\mu,b}f \ \textrm{for any} \ j.
\end{equation}
The reason is that the $v-$spectrum of the tangential operators  are fixed. Thus,  if  $v>(<) 1-p$, the same  holds with $p$ with replaced by $p+\mu$ or $p-\mu$. Hence the solution formulas (in (\ref{equ solution when v is real and  p>1-v is integral from 1 to r}), (\ref{equ solution formula when v positive and p less than 1-v}), (\ref{equ solution when v is purely imaginary}), and   (\ref{equ solution when v=0})) do not change.  

\begin{proof}It's directly implies by the Lemmas \ref{lem  boostrap K}, \ref{lem  boostrap Kstar}, and \ref{lem cokernel contains the more vanishing sections}.
\end{proof}
\begin{lem}\label{lem  boostrap K}Let $A,p,b,\mu$ be as in Theorem \ref{thm characterizing cokernel}, then     $L^{2}_{p,b}\subset L^{2}_{p+\mu,b}$ is an invariant subspace of $K_{p+\mu,b}$, and $K_{p+\mu,b}=K_{p,b}$ when restricted to  $L^{2}_{p,b}$. Consequently, for any  $f\in L^{2}_{p,b}$, we have 
   \begin{equation}
   |K_{p,b}f|_{L^{2}_{p,b}}\leq C_{\mu}|f|_{L^{2}_{p+\mu,b}}. 
   \end{equation}

\end{lem}
\begin{proof}By (\ref{equ formula for Kpb}), we only need to show the parametrices $Q_{j}$ near the singularities satisfies the property asserted. The parametrices away from the singularity do not depend on the weight chosen. 

   In the setting of (\ref{equ lem local inverse near O 2}) and (\ref{equ contract mapping for linear equation with Cepsilon0}), let $Q_{j,p,b}$ denote $Q_{A,p,b}$ near $O_{j}$.   By (\ref{equ Crucial observation that Q is the same for neary by weights}) and uniqueness of fixed point of contract mappings, we find  $Q_{j,p,b}f=Q_{j,p+\mu,b}f$ when  $f\in C^{\infty}_{c}[B_{O_{j}}\setminus O_{j}]$. Since $Q_{j,p,b}\ (Q_{j,p+\mu,b})$  is bounded from $L^{2}_{p,b}(B_{O_{j}})\ (L^{2}_{p+\mu,b}(B_{O_{j}}))$ to $W^{1,2}_{p,b}(B_{O_{j}})\ (W^{1,2}_{p+\mu,b}(B_{O_{j}}))$,  their extensions (as in proof of Theorem \ref{thm W22 estimate on 1-forms}) agree on $L^{2}_{p,b}(B_{O_{j}})\subset L^{2}_{p+\mu,b}(B_{O_{j}})$.  Hence 
   \begin{equation}
   |Q_{j,p,b}f|_{L^{2}_{p,b}(B_{O_{j}})}\leq |f|_{L^{2}_{p+\mu,b}(B_{O_{j}})}\ \textrm{when}\ f\in L^{2}_{p,b}(B_{O_{j}}). 
   \end{equation}
\end{proof}
   \begin{lem}\label{lem  boostrap Kstar} Let $A,p,b,\mu$ be as in Theorem \ref{thm characterizing cokernel}, then  $K_{p,b}^{\star}$ is a bounded operator from $L^{2}_{p,b}$ to $L^{2}_{p-\mu,b}$. The bound is uniform as in Definition \ref{Def Dependence of the constants}.

   \end{lem}
\begin{proof}   Assuming  $\xi\in L^{2}_{p,b}$ vanishes near the singularities (thus  $\xi\in L^{2}_{p,b}$ for any $p$), by Lemma \ref{lem  boostrap K}, we find 
\begin{equation}\label{equ lem boostrap Kstar 1}
\int_{M}<K_{p,b}(\frac{\xi}{r^\mu}),\ f>w_{p,b}dV\leq C|K_{p,b}(\frac{\xi}{r^\mu})|_{L^{2}_{p,b}}|f|_{L^{2}_{p,b}}\leq C|\frac{\xi}{r^\mu}|_{L^{2}_{p+\mu,b}}|f|_{L^{2}_{p,b}}.
\end{equation}

On the other hand, let  $\eta_{\epsilon}$  be the cutoff function of the singular points satisfying condition  (\ref{equ cut-off function bound near the singular point}). Let $\xi=\eta^{2}_{\epsilon}\frac{K_{p,b}^{\star}\xi}{r^\mu}$, 
we obtain \begin{equation}\label{equ lem boostrap Kstar 2}
\int_{M}<K_{p,b}(\frac{\xi}{r^\mu}),\ f>w_{p,b}dV=\int_{M}\eta^{2}_{\epsilon}\frac{|K_{p,b}^{\star}f|^{2}}{r^{2\mu}}w_{p,b}dV\geq C |\eta_{\epsilon}K_{p,b}^{\star}f|^{2}_{L^{2}_{p-\mu,b}}.
\end{equation}
Hence (\ref{equ lem boostrap Kstar 1}) and (\ref{equ lem boostrap Kstar 2}) imply 
\begin{equation}
|\eta_{\epsilon}K_{p,b}^{\star}f|^{2}_{L^{2}_{p-\mu,b}}\leq C|\eta^{2}_{\epsilon}K_{p,b}^{\star}f|_{L^{2}_{p-\mu,b}}|f|_{L^{2}_{p,b}}\leq C|\eta_{\epsilon}K_{p,b}^{\star}f|_{L^{2}_{p-\mu,b}}|f|_{L^{2}_{p,b}}.
\end{equation}
Then $|\eta_{\epsilon}K_{p,b}^{\star}f|_{L^{2}_{p-\mu,b}}\leq C|f|_{L^{2}_{p,b}}$. Let $\epsilon\rightarrow 0$, by the monotone convergence theorem of Lebesgue measure, the proof of Lemma \ref{lem  boostrap Kstar} is complete. 
\end{proof}

\begin{lem}\label{lem cokernel contains the more vanishing sections} Suppose $L_{A}^{\star}f=0$ away from the singular points and  $f\in C^{1,\alpha}_{\frac{7}{2}+p-\mu}(M)$ for some $\mu>0$.  Then $f\in coker L_{A}$ . 
\end{lem}
\begin{proof} With the same $\eta_{\epsilon}$ as in Lemma \ref{lem  boostrap Kstar},  we compute
\begin{eqnarray}
& &\int_{M}<L_{A}\xi,f>w_{p,b}dV=\lim_{\epsilon\rightarrow 0}\int_{M}<L_{A}\xi,\eta_{\epsilon}f>w_{p,b}dV\nonumber
\\&=&\lim_{\epsilon\rightarrow 0}\int_{M}<\xi,G(d\eta_{\epsilon},f)>w_{p,b}dV \triangleq \lim_{\epsilon\rightarrow 0}\Pi_{\epsilon}. \nonumber
\end{eqnarray}

By  (\ref{equ cut-off function bound near the singular point}), (\ref{equ LA star formula}), (\ref{equ Def of Coker}), and  H\"older inequality, since $\mu>0$, we obtain 
\begin{eqnarray*}& &\Pi_{\epsilon}
\leq  C|\xi|_{L^{2}_{p-1,b}}(\int_{M}|G(d\eta_{\epsilon},f)|^{2}r^{2}w_{p,b}dV)^{\frac{1}{2}}\\&\leq &  C\Sigma_{j}|\xi|_{L^{2}_{p-1,b}}(\epsilon^{2\mu}\int_{B_{O_{j}}(2\epsilon)-B_{O_{j}}(\epsilon)}|f|^{2}\frac{w_{p,b}}{r^{2\mu}}dV)^{\frac{1}{2}}
\\&\leq &  C\epsilon^{\mu}|\xi|_{L^{2}_{p-1,b}}|f|_{L^{2}_{p-\mu,b}}\rightarrow 0\ \textrm{as}\ \epsilon\rightarrow 0.
\end{eqnarray*}
The proof is complete.
\end{proof}

\section{Appendix}
\subsection{Appendix A: Weitzenb\"ock formula in the model case. \label{section Appendix A: Weitzenb formula in the model case}}
Suppose $p+q\leq n$, $\Phi$ is a $p-$form, $\nu$ is a $q-$form, and both are possibly $adE-$valued. We have \begin{equation}\label{equ appendix A star wedge}
  \star(\Phi\wedge \nu)= (\star \Phi)\llcorner\nu=(-1)^{pq}  \Phi \lrcorner(\star\nu).
\end{equation}

Let the $e_{i}$'s be the standard coordinate vectors in $\R^{7}$. 
The standard (Euclidean) $G_{2}-$structure  over $\R^{7}\setminus O$ is 
\begin{eqnarray}\label{eqnarray Euc G2 forms}& &\phi_{0}=e^{123}-e^{145}-e^{167}-e^{246}+e^{257}-e^{347}-e^{356}
\\& & \psi_{0}=-e^{1247}-e^{1256}+e^{1346}-e^{1357}-e^{2345}-e^{2367}+e^{4567}\nonumber
\end{eqnarray}

Notice  the $e_{i}$'s here are not the same as the ones in Section \ref{Appendix B: proof of  cone formula for laplacian}. We abuse notations on frames in different sections. 
 
\begin{lem}\label{lem formula for LA squared}Suppose $A_{O}$ is a cone connection over $E\rightarrow \R^{7}\setminus O$. Let $L_{A_{O}}$ be the deformation operator of $A_{O}$ with respect to $\phi_{0},\psi_{0}$ (see (\ref{equ introduction formula for deformation operator})), and $[\begin{array}{c}
\sigma   \\
  a   
\end{array}]$  be as in Section \ref{section Seperation of variable for  the system in the model case}.   Then 
\begin{equation}\label{equ bochner identity for general cone connection}
L_{A_{O}}^{2}[\begin{array}{c}
\sigma   \\
  a   
\end{array}]=[\begin{array}{c}
\nabla^{\star}\nabla \sigma  -\star([F_{A_{O}},a]\wedge\psi_{0}) \\
 \star([F_{A_{O}},\sigma]\wedge\psi_{0})+\nabla^{\star}\nabla a+F_{A_{O}}\underline{\otimes} a-[F_{A_{O}}, a]\lrcorner \psi_{0}.   
\end{array}]
\end{equation}

Suppose $A_{O}$ is a $G_{2}-$instanton with respect to the standard (Euclidean) $G_{2}-$structures i.e.
$\star(F_{A_{O}}\wedge \phi_{0})=-F_{A_{O}}$ or ($F_{A_{O}}\wedge \psi_{0}=0$), then 
\begin{equation}\label{equ bochner identity for instanton cone connection}
L_{A_{O}}^{2}[\begin{array}{c}
\sigma   \\
  a   
\end{array}]=[\begin{array}{c}
\nabla^{\star}\nabla \sigma   \\
  \nabla^{\star}\nabla a+2F_{A_{O}}\underline{\otimes} a.   
\end{array}]
\end{equation}
\end{lem}
\begin{proof} All the forms in this proof can possibly be $adE-$valued.

By (\ref{equ appendix A star wedge}) and (\ref{equ introduction formula for deformation operator}), we have $L_{A_{O}}[\begin{array}{c}
\sigma   \\
  a   
\end{array}]=[\begin{array}{c}
d^{\star}a   \\
  d\sigma+da\lrcorner \phi_{0}.   
\end{array}]$
Hence \begin{equation}
L^{2}_{A_{O}}[\begin{array}{c}
\sigma  \\
  a   
\end{array}]=[\begin{array}{c}
d^{\star}d\sigma -\star([F_{A_{O}},a]\wedge\psi_{0})  \\
 \star([F_{A_{O}},\sigma]\wedge\psi_{0})+ dd^{\star}a+\{d[da\lrcorner \phi_{0}]\}\lrcorner \phi_{0}.   
\end{array}]
\end{equation}

We first prove the general formula (\ref{equ bochner identity for general cone connection}). For any 1-form $B=B_{i}e^{i}$, we compute 
\begin{equation}
(dB\lrcorner  \phi_{0})(e_{1})=-d_{6}B_{7}+d_{7}B_{6}-d_{4}B_{5}+d_{5}B_{4}-d_{3}B_{2}+d_{2}B_{3}
\end{equation} 
Let $b$ be a $2-$form, write $b=\Sigma_{i<j}b_{ij}e^{ij}$. Let $B=b\lrcorner \phi_{0}$, then 
\begin{eqnarray*}& & B_{7}=-b_{16}+b_{25}-b_{34},\ \ B_{6}=-b_{24}+b_{17}-b_{35},\ \ B_{5}=-b_{27}-b_{14}+b_{36}.\nonumber
\\& & B_{4}=b_{15}+b_{37}+b_{26},\ \ B_{3}=-b_{47}-b_{56}+b_{12},\ \ B_{2}=-b_{13}-b_{46}+b_{57}.
\end{eqnarray*}
\begin{eqnarray}\textrm{Then} & & \{[d (b\lrcorner \phi_{0})] \lrcorner \phi_{0} \}(e_{1})
\\&=&-\Sigma_{i=2}^{7}d_{i}b_{i1}-<db,e^{625}>+<db,e^{634}>+<db,e^{427}>+<db,e^{537}>\nonumber
\\&=& d^{\star}b(e_{1})-(db\lrcorner \psi_{0})(e_{1}).\nonumber
\end{eqnarray}

Therefore, by computing the component on $e_{2},......,e_{7}$ similarly,  we  arrive at  the following intermediate result. 
\begin{clm}\label{equ intermediate formula for LA{2}} For any $adE$-valued $2-$form $b$, the following formula holds. 
\begin{equation}
[d (b\lrcorner \phi_{0})] \lrcorner \phi_{0}=d^{\star}b-db\lrcorner \psi_{0}.
\end{equation}
\end{clm}
Let $b=da$, we obtain $[d (da\lrcorner \phi_{0})] \lrcorner \phi_{0}=d^{\star}da-[F_{A_{O}},a]\lrcorner \psi_{0}$. Using the Bochner-identity $
d^{\star}da+dd^{\star}a=\nabla^{\star}\nabla a+F_{A_{O}}\underline{\otimes} a$, we find
\begin{equation}
dd^{\star}a+\{d[da\lrcorner \phi_{0}]\}\lrcorner \phi_{0}= \nabla^{\star}\nabla a+F_{A_{O}}\underline{\otimes} a-[F_{A_{O}},a]\lrcorner \psi_{0}.
\end{equation}
The proof of (\ref{equ bochner identity for general cone connection}) is complete

Next, suppose $A_{O}$ is a $G_{2}-$instanton  with respect to $\phi_{0}$, we prove  (\ref{equ bochner identity for instanton cone connection}). Notice in this case we automatically have $\star([F_{A_{O}},a]\wedge\psi_{0})=0$ and 
 $\star([F_{A_{O}},\sigma]\wedge\psi_{0})=0$. We compute 
\begin{equation} -[F_{A_{O}},a]\lrcorner \psi_{0} =-\star(\phi_{0}\wedge [F_{A_{O}},a])=\star\{-\phi_{0}\wedge F_{A_{O}}\wedge a)\}+\star\{\phi_{0}\wedge a\wedge F_{A_{O}})\}
\end{equation}

The instanton equation implies \begin{equation}
\star\{-\phi_{0}\wedge F_{A_{O}}\wedge a\}=-[\star(\phi_{0}\wedge F_{A_{O}})]\lrcorner a=F_{A_{O}}\llcorner a,
\end{equation}
and 
\begin{equation}
\star\{\phi_{0}\wedge  a\wedge F_{A_{O}}\}=-\star\{  a\wedge \phi_{0} \wedge F_{A_{O}}\}=a\lrcorner \star( \phi_{0} \wedge F_{A_{O}})=-a\lrcorner F_{A_{O}}.
\end{equation}

Then we get
\begin{equation} -[F_{A_{O}},a]\lrcorner \psi_{0} =F_{A_{O}}\llcorner a-a\lrcorner F_{A_{O}}=\Sigma_{i,j}[F_{ij},a_{i}]e^{j}\triangleq F_{A_{O}}\underline{\otimes} a.
\end{equation}

The proof of (\ref{equ bochner identity for instanton cone connection}) and Lemma \ref{lem formula for LA squared} is complete.
\end{proof}

\subsection{Appendix B: Proof of Lemma \ref{lem  cone formula for laplacian} \label{Appendix B: proof of  cone formula for laplacian}}
The polar coordinate formula for $1-$forms  is more involved. We need some preliminary identities. The Euclidean metric is equal to $dr^{2}+r^{2}g_{S^{n-1}}$.  For any    point $(\mathfrak{p},r)\in \R^{7}\setminus O$, we choose  $e_{1}......e_{n-1}$ as an orthonormal frame on $S^{n-1}$ near $\mathfrak{p}$. Furthermore, we require $e_{1}......e_{n-1}$ to  be the geodesic coordinates on the sphere at $\mathfrak{p}$.

As vector fields defined under polar coordinate near  $\mathfrak{p}\times (0,1)$, for any $r$, let $\nabla^{S}$ denote the covariant derivative induced on $S^{n-1}(r)$.  $\nabla^{S}$ is just the Levi-Civita connection of the induced metric. Since (the metric on) $S^{n-1}(r)$ differs from the unit sphere by a constant rescaling, and $e_{i}'$s are the geodesic coordinates at $\mathfrak{p}$,   then  
\begin{equation}\label{equ cone formula the ei are geodesic coordinate at p}
\nabla^{S}_{e_{j}}e_{i}=0 \ \textrm{along}\ \mathfrak{p}\times (0,1),\ \textrm{for all}\ i,\ j. 
\end{equation}
Notice  the $e_{i}$'s here are not the same as the ones in Section \ref{section Appendix A: Weitzenb formula in the model case}. We abuse these notations.

 The vector files  $\frac{\partial}{\partial r},\frac{e_{i}}{r},i=1...n-1$ form 
an orthonormal  basis for the Euclidean metric over $\R^{7}\setminus O$.  

\begin{lem} \label{lem covariant derivatives} In a neighbourhood of $\mathfrak{p}\times (0,1)$ in $\R^{n}\setminus O$, we have 
\begin{equation}
\nabla_{\frac{\partial}{\partial r}}e_{i}=\frac{e_{i}}{r};\ \ \ \ \ \ \ \nabla_{\frac{\partial}{\partial r}}e^{i}=-\frac{e^{i}}{r};\ \ \ \ \ \ \;\ \nabla_{e_{i}}dr=re^{i}. 
\end{equation}
\begin{equation}\label{equ at p ei are orthonormal basis on sphere}
\nabla_{e_{i}}e^{k}=\nabla^{S}_{e_{i}}e^{k}-\delta_{ik}\frac{dr}{r};\ \ \ \ \ \ \nabla_{e_{i}}e_{j}=\nabla^{S}_{e_{i}}e_{j}-\delta_{ij}r\frac{\partial }{\partial r}.
\end{equation}
\end{lem}
Hence we compute for any $\phi\in \Omega^{1}_{\Xi}(S^{n-1})$ that
\begin{equation}\label{equ radial derivative of a spherical form}
\nabla_{\frac{\partial}{\partial r}}\phi=\nabla_{\frac{\partial}{\partial r}}(\phi_{i}e^{i})=-\frac{\phi}{r};\ \nabla_{\frac{\partial}{\partial r}}\nabla_{\frac{\partial}{\partial r}}\phi=\frac{2\phi}{r^{2}}.
\end{equation}

\begin{proof}[Proof of Lemma \ref{lem  cone formula for laplacian}:]

We first observe that, for any bundle-valued form $b$, 
\begin{equation}
-\nabla^{\star}\nabla b=\Sigma_{k=1}^{n}\nabla^{2}b(v_{k},v_{k}),\ (v_{k})\ \textrm{is an orthonormal basis}. 
\end{equation}
This definition does not depend on the orthonormal basis chosen, then we can use $\frac{\partial}{\partial r},\frac{e_{1}}{r}......\frac{e_{n-1}}{r}$ to obtain
\begin{eqnarray}\label{eqnarray cone formula 1 forms 1}& &-\nabla^{\star}\nabla b=\nabla^{2}b(\frac{\partial}{\partial r},\frac{\partial}{\partial r})+\frac{1}{r^{2}}\Sigma_{i=1}^{n-1}\nabla^{2}b(e_{i},e_{i})
\\&=&\nabla_{\frac{\partial}{\partial r}}\nabla_{\frac{\partial}{\partial r}}b+\frac{1}{r^{2}}\Sigma_{i=1}^{n-1}\nabla_{i}\nabla_{i}b-\frac{1}{r^{2}}\nabla_{(\Sigma_{i}\nabla_{i}e_{i})}b\nonumber
\\&=&\nabla_{\frac{\partial}{\partial r}}\nabla_{\frac{\partial}{\partial r}}b+\frac{1}{r^{2}}\Sigma_{i=1}^{n-1}\nabla_{i}\nabla_{i}b-\frac{1}{r^{2}}\nabla_{(\Sigma_{i}\nabla^{S}_{i}e_{i})}b+\frac{n-1}{r}\nabla_{\frac{\partial}{\partial r}}b.\nonumber
\end{eqnarray}
The $\nabla_{i}\nabla_{i}b$ should be understood as $\nabla_{i}(\nabla_{i}b)$.

Part I: we compute the rough laplacian of the radial part $a_{r}\frac{dr}{r}$. By (\ref{eqnarray cone formula 1 forms 1}), 
\begin{eqnarray}\label{eqnarray cone formula 1 forms 2}& & -\nabla^{\star}\nabla(a_{r}\frac{dr}{r})
\\&=&\nabla_{\frac{\partial}{\partial r}}\nabla_{\frac{\partial}{\partial r}}(a_{r}\frac{dr}{r})+\frac{1}{r^{2}}\Sigma_{i=1}^{n-1}\nabla_{i}\nabla_{i}(a_{r}\frac{dr}{r})-\frac{1}{r^{2}}\nabla_{(\Sigma_{i}\nabla^{S}_{i}e_{i})}(a_{r}\frac{dr}{r})+\frac{n-1}{r}\nabla_{\frac{\partial}{\partial r}}(a_{r}\frac{dr}{r}).\nonumber
\end{eqnarray}
Term-wise computation gives 
\begin{eqnarray}\label{eqnarray termwise computation for laplacian of the radial term}
& &\nabla_{\frac{\partial}{\partial r}}(a_{r}\frac{dr}{r})
=(\nabla_{\frac{\partial }{\partial r}}a_{r}) \frac{dr}{r}-a_{r}\frac{dr}{r^2};\\& & \nabla_{\frac{\partial}{\partial r}}\nabla_{\frac{\partial}{\partial r}}(a_{r}\frac{dr}{r})=(\nabla_{\frac{\partial }{\partial r}}\nabla_{\frac{\partial }{\partial r}}a_{r}) \frac{dr}{r}-2(\nabla_{\frac{\partial}{\partial r}}a_{r})\frac{dr}{r^{2}}+\frac{2}{r^{3}}a_{r}dr.\nonumber
\end{eqnarray}
For the hardest term $\Sigma_{i=1}^{n-1}\nabla_{i}\nabla_{i}(a_{r}\frac{dr}{r})$, fix $i$,  we compute 
\begin{equation}\label{equ cone formula 1 forms 1}
\nabla_{i}\nabla_{i}(a_{r}\frac{dr}{r})=(\nabla_{i}\nabla_{i}a_{r})\frac{dr}{r}+2(\nabla_{i}a_{r})(\nabla_{i}\frac{dr}{r})+a_{r}\nabla_{i}\nabla_{i}\frac{dr}{r}.
\end{equation}

Using 
\begin{equation}
\nabla_{i}\frac{dr}{r}=e^{i},\ \nabla_{i}\nabla_{i}\frac{dr}{r}=-\frac{dr}{r}+\nabla^{S}_{i}e^{i},\ \textrm{and}\ \Sigma_{i=1}^{n-1}\nabla_{i}\nabla_{i}a_{r}=\Delta_{S}a_{r}+\nabla^{S}_{\Sigma_{i}\nabla^{S}_{i}e_{i}}a_{r},
\end{equation}
we obtain (by summing up $i$ in (\ref{equ cone formula 1 forms 1})) 
\begin{equation}\label{equ hardest term in rough laplacian of radial term}
\Sigma_{i=1}^{n-1}\nabla_{i}\nabla_{i}(a_{r}\frac{dr}{r})=(\Delta_{S}a_{r})\frac{dr}{r}+(\nabla^{S}_{\Sigma_{i}\nabla^{S}_{i}e_{i}}a_{r})\frac{dr}{r}+2d_{S}a_{r}-(n-1)a_{r}\frac{dr}{r}+a_{r}(\Sigma_{i}\nabla^{S}_{i}e^{i})
\end{equation}

By (\ref{equ cone formula the ei are geodesic coordinate at p}), (\ref{eqnarray termwise computation for laplacian of the radial term}), (\ref{eqnarray cone formula 1 forms 2}), and (\ref{equ hardest term in rough laplacian of radial term}),  along $\mathfrak{p}\times (0,1)$, we have 
\begin{eqnarray}\label{equ rough lapla of the radial term}
& &-\nabla^{\star}\nabla (a_{r}\frac{dr}{r})
\\&=&(\nabla_{\frac{\partial}{\partial r}}\nabla_{\frac{\partial}{\partial r}}a_{r})\frac{dr}{r}+\frac{n-3}{r}(\nabla_{\frac{\partial }{\partial r}}a_{r})\frac{dr}{r}+\frac{1}{r^{2}}[\Delta_{S}a_{r}-(2n-4)a_{r}]\frac{dr}{r}+\frac{2}{r^{2}}d_{S}a_{r}.\nonumber
\end{eqnarray}

Since (\ref{equ rough lapla of the radial term}) is independent of the coordinate chosen, and $\mathfrak{p}$ is arbitrary, then it holds everywhere on $\R^{7}\setminus O$.

Part II: we compute the rough laplacian of the radial part $a_{s}$ (which  does not have radial component). First we have 
\begin{equation*} -\nabla^{\star}\nabla a_{s}
=\nabla_{\frac{\partial}{\partial r}}\nabla_{\frac{\partial}{\partial r}}a_{s}+\frac{1}{r^{2}}\Sigma_{i=1}^{n-1}\nabla_{i}\nabla_{i}a_{s}-\frac{1}{r^{2}}\nabla_{(\Sigma_{i}\nabla^{S}_{i}e_{i})}a_{s}+\frac{n-1}{r}\nabla_{\frac{\partial}{\partial r}}a_{s}.
\end{equation*}
In this case, we only have to compute the crucial term $\nabla_{i}\nabla_{i}a_{s}$. We write $a_{s}=\Sigma_{i=1}^{n-1}a_{i}$, then 
\begin{equation}\nabla_{i}a_{s}=(\nabla_{i}a_{k})e^{k}+a_{k}\nabla_{i}e^{k}=\nabla^{S}_{i}a_{s}+a_{k}(\nabla_{i}e^{k}-\nabla_{i}^{S}e^{k})
=\nabla^{S}_{i}a_{s}-a_{s}(e_{i})\frac{dr}{r}.
\end{equation}

To compute $\nabla_{i}\nabla_{i}a_{s}$, it suffices to compute $\nabla_{i}\nabla_{i}^{S}a_{s}$ and $\nabla_{i}[a_{s}(e_{i})\frac{dr}{r}]$.  

\begin{equation}
\Sigma_{i}\nabla_{i}\nabla_{i}^{S}a_{s}=\Sigma_{i}\nabla^{S}_{i}\nabla_{i}^{S}a_{s}-\Sigma_{i}(\nabla_{i}^{S}a_{s})(e_{i})\frac{dr}{r}=\Delta_{S}a_{s}+\nabla^{S}_{\nabla^{S}_{i}e_{i}}a_{s}+d^{\star}a_{s}\frac{dr}{r}.
\end{equation} 
\begin{eqnarray}\textrm{
On the other hand}& &\Sigma_{i}\nabla_{i}[a_{s}(e_{i})\frac{dr}{r}]=\Sigma_{i}\{[\nabla^{S}_{i}a_{s}(e_{i})]\frac{dr}{r}+\frac{1}{r}a_{i}\nabla_{i}dr\}\nonumber
\\&=& \Sigma_{i}\{(\nabla^{S}_{i}a_{s})(e_{i})\frac{dr}{r}+a_{s}(\nabla_{i}^{S}e_{i})\frac{dr}{r}+a_{i}e^{i}\}\nonumber
\\&=& -(d^{\star}a_{s})\frac{dr}{r}+a_{s}(\nabla_{i}^{S}e_{i})\frac{dr}{r}+a_{s}.
\end{eqnarray}
\begin{equation}\label{equ cone formula 1 forms 2}\textrm{Then}\
\Sigma_{i}\nabla_{i}\nabla_{i}a_{s}=\Delta_{S}a_{s}+2(d^{\star}a_{s})\frac{dr}{r}+\nabla^{S}_{\nabla^{S}_{i}e_{i}}a_{s}-a_{s}(\nabla_{i}^{S}e_{i})\frac{dr}{r}-a_{s}.
\end{equation}
By (\ref{equ at p ei are orthonormal basis on sphere}), (\ref{equ cone formula 1 forms 2}), and (\ref{equ cone formula the ei are geodesic coordinate at p}),  the following holds along $\mathfrak{p}\times (0,1)$
\begin{equation}\label{equ rough lapla of the spherical part} -\nabla^{\star}\nabla a_{s}=
 \nabla_{\frac{\partial}{\partial r}}\nabla_{\frac{\partial}{\partial r}}a_{s}+\frac{n-1}{r}\nabla_{\frac{\partial}{\partial r}}a_{s}
+\frac{1}{r^{2}}(\Delta_{S}a_{s}-a_{s})+(2d^{\star}a_{s})\frac{dr}{r^{3}}.
\end{equation}
It does not depend on the coordinates chosen, thus holds everywhere. The proof of Lemma \ref{lem  cone formula for laplacian} is completed by combining (\ref{equ rough lapla of the radial term}) and (\ref{equ rough lapla of the spherical part}).
\end{proof}
\subsection{Appendix C: Fundamental facts on elliptic systems}

\textbf{In this section, we work under the same conditions as in Theorem \ref{Thm Fredholm}}.
For any  $q$ near a singular point $O$, denote
\begin{equation}\label{equ abbreviation for the ball}B=B_{q}(\frac{r_{q}}{100}),
\end{equation}
then $B$ lies in one coordinate sector. Let $L_{\psi}$ be the local elliptic operator with $A=0$ in $B$ i.e.
 $L_{\psi}[\begin{array}{c}
\sigma   \\
  a   
\end{array}]=[\begin{array}{c}
d^{\star}a   \\
  d\sigma+da\lrcorner \phi  
\end{array}]$.  For any locally defined $G_{2}-$structure $(\phi,\psi)$, we have the following weighted estimate   due to Nirenberg-Douglis \cite{Nirenberg}. 
\begin{equation}
|\xi|^{\star}_{2,\alpha,B}\leq C|L_{\psi} \xi|^{(1)}_{1,\alpha,B}+C|\xi|_{0,B}, 
\end{equation}
where "$C$" depends at most on the $C^{5}-$norm  and the non-degeneracy of   $\phi$.

An easy but important building-block  is 
\begin{lem}\label{lem Schauder estimate in local coordinates in small balls} Suppose $\xi\in C^{k,\alpha}(B)$, the following estimate holds. \begin{equation*}
|\xi|^{\star}_{k,\alpha,B}\leq C|L_{A}\xi|^{[1]}_{k-1,\alpha,B}+C|\xi|_{0,B},\ k=2,3,4.
\end{equation*}
The estimate holds the same with $L_{A}$ replaced by $L_{A}^{\star}$.
\end{lem}
\begin{proof} We only prove it for $L_{A}$ when $k=2$. On  $L_{A}^{\star}$, note that
 Definition \ref{Def global weight and Sobolev spaces} implies $|\frac{dw_{p,b}}{w_{p,b}}|\leq \frac{C}{r}$, thus we  verify that $|\frac{dw_{p,b}}{w_{p,b}}|^{[1]}_{2,0,B}\leq |\frac{dw_{p,b}}{w_{p,b}}|^{(1)}_{2,0,B}<C$. Then  formula (\ref{equ LA star formula}) implies that the proof for $L_{A}$ directly carries over to $L_{A}^{\star}$.

 The admissible conditions implies   $|A|^{[1]}_{2,0,B}
\leq  C$, hence $|A|^{[1]}_{1,\alpha,B}
\leq  C$.  Using $|(L_{A}-L_{\psi})\xi|=|A\otimes_{g_{\phi}}\xi|$ and  splitting of weight, we have 
$$|A\otimes_{g}\xi|^{[1]}_{1,\alpha,B}\leq |A|^{[1]}_{1,\alpha,B}|\xi|^{\star}_{1,\alpha,B}\leq C|\xi|^{\star}_{1,\alpha,B}.$$
Thus $|(L_{A}-L_{\psi})\xi|^{[1]}_{1,\alpha,B}\leq C|\xi|^{\star}_{1,\alpha,B}$, and this implies
\begin{equation}\label{equ Schauder est with junk terms}
|\xi|^{\star}_{2,\alpha,B}\leq C(|L_{A}\xi|^{[1]}_{1,\alpha,B}+|\xi|_{0,B})+|(L_{A}-L_{\psi})\xi|^{[1]}_{1,\alpha,B}\leq C(|L_{A}\xi|^{[1]}_{1,\alpha,B}+|\xi|^{\star}_{1,\alpha,B})
\end{equation}
By  Lemma 6.32 in \cite{GilbargTrudinger} (standard interpolation), for any $\mu\in (0,\frac{1}{100})$, we have
 \begin{equation}\label{equ intepolation for Schauder est}
|\xi|^{\star}_{1,\alpha,B}\leq \mu[\nabla^{2}\xi]^{[2]}_{0,B}+C_{\mu}|\xi|_{0,B}.
\end{equation}
Then   (\ref{equ Schauder est with junk terms}) and (\ref{equ intepolation for Schauder est})  imply Lemma \ref{lem Schauder estimate in local coordinates in small balls}. 
\end{proof}

\begin{prop}\label{prop log weighted Schauder estimate} Let $\gamma$ be any real number,  suppose $\xi\in $ is $C^{2,\alpha}$ away from the singularity, $L_{A}\xi\in C_{(\gamma,b)}^{1,\alpha}(M)$, $\xi\in C_{(\gamma-1,b)}^{0}(M) $. Then $\xi\in C_{(\gamma-1,b)}^{2,\alpha}(M)$, and 
\begin{equation*} 
|\xi|^{(\gamma-1,b)}_{2,\alpha,M}\leq C|L_{A}\xi|^{(\gamma,b)}_{1,\alpha,M}+C|\xi|^{(\gamma-1,b)}_{0,M}.
\end{equation*}
\end{prop}
\begin{proof} We only consider the case when $1+\gamma+\alpha\geq 0$. By our choice of $x$ and $y$ (the paragraph above (\ref{equ Prop Apriori Schauder 1 Appendix C})), the proof is similar when it's negative.  Without loss of generality, for any singular point $O$, we only consider  $x\in B_{O}(\rho_{0})$ and  $B_{x}(\frac{r_{x}}{1000})$. For any $y\in B_{x}(\frac{r_{x}}{1000})$, the distance from $y$ to $\partial B_{x}(\frac{r_{x}}{1000})$ is less than $r_{y}$, and  we have $2r_{y}>r_{x}>\frac{r_{y}}{2}$, $10(-\log r_{y})>-\log r_{x}>\frac{-\log r_{y}}{10}$. Thus $r_{x,y}\simeq r_{y} \simeq r_{x}\simeq \underline{r_{x,y}}$. Hence, using the proof of 4.20 in Theorem 4.8 of \cite{GilbargTrudinger} and Lemma  \ref{lem Schauder estimate in local coordinates in small balls}, multiplying both sides of the estimate in Lemma\ref{lem Schauder estimate in local coordinates in small balls} by $r_{q}^{\gamma-1}(-\log r_{q})^{b}$, we obtain the following   lower order estimate 

\begin{equation} \label{equ C20 estimate in the apriori Schauder}
|\xi|^{(\gamma-1,b)}_{2,0,M}\leq C|L_{A}\xi|^{(\gamma,b)}_{1,\alpha,M}+C|\xi|^{(\gamma-1,b)}_{0,M}.
\end{equation}
 The term left to  estimate is the following one with highest order.
  $$Q_{x,y}=(-\log \underline{r_{x,y}})^{b}r^{1+\gamma+\alpha}_{x,y}\frac{|\nabla^{2}\xi(x)-\nabla^{2}  \xi(y)|}{|x-y|^{\alpha}}.$$
This can be done in a standard way.  We can assume $r_{x}\geq r_{y}$, since otherwise we only need to interchange them. We note that this choice is different from the paragraph above (6.15) in \cite{GilbargTrudinger}.  Suppose $y\in B_{x}(\frac{r_{x}}{2000})$, the proof of (\ref{equ C20 estimate in the apriori Schauder})  implies one more conclusion:
\begin{equation}\label{equ Prop Apriori Schauder 1 Appendix C}
Q_{x,y}\leq C|L_{A}\xi|^{(\gamma,b)}_{1,\alpha,M}+C|\xi|^{(\gamma-1,b)}_{0,M}.
\end{equation}
When  $y\notin B_{x}(\frac{r_{x}}{2000})$ (we only need to consider $y$ in $B_{O}(\rho_{0})$),  we find
\begin{eqnarray}
& &Q_{x,y}
\\&\leq &  (-\log r_{x})^{b}r^{1+\gamma}_{x}|\nabla^{2}\xi(x)|+(-\log r_{y})^{b}r^{1+\gamma}_{y}|\nabla^{2}  \xi(y)|\leq C|\xi|^{(\gamma-1,b)}_{2,0,M}\nonumber
\\& \leq &   C|L_{A}\xi|^{(\gamma,b)}_{\alpha,M}+C|\xi|^{(\gamma-1,b)}_{0,M},\ \textrm{by}\ (\ref{equ C20 estimate in the apriori Schauder}).\nonumber
\end{eqnarray}

The proof of Proposition \ref{prop log weighted Schauder estimate} is complete.
\end{proof}

Working on each coordinate patch separately, by  the proof of Lemma 6.32 in \cite{GilbargTrudinger} with  slight modification (on log weight as Proposition \ref{prop log weighted Schauder estimate}), we obtain 
\begin{lem}(Intepolation)\label{lem C1 C0 intepolate Calpha} For any $\mu<\frac{1}{3000}$, $b\geq 0$,  $0<\alpha<1$, and real numbers  $k$, there exists a constant $C_{\mu,k,\alpha,b}$ with the following property.   For any section $\xi$ and non-negative integer $j$, the following  intepolation holds.
\begin{equation}\label{equ intepolation 1 for calpha norm}
|\nabla^{j}\xi|^{(k,b)}_{\alpha,M}\leq \mu |\nabla^{j+1} \xi|^{(k+1,b)}_{0,M}+C_{\mu,k,\alpha,b} |\nabla^{j}\xi|^{(k,b)}_{0,M}, 
\end{equation}
where $\nabla^{j}\xi$ is viewed as a combination of locally defined matrix-valued tensors in each chart of $\mathbb{U}_{\rho_{0}}$. 
\end{lem}

\begin{lem}\label{lem regularity of solution to the deformation equation with C1alpha rhs} Suppose $B$ is a ball such that  $2B$ is contained in a single coordinate neighbourhood away from the singularity. Suppose $\xi\in W^{1,2}(2B)$ and $L_{A}\xi\in C^{1,\alpha}(2B)$ (in the sense of strong solution). Then $u\in C^{2,\alpha}(B)$. 
\end{lem}
\begin{rmk} We believe Lemma \ref{lem regularity of solution to the deformation equation with C1alpha rhs} is in literature. Since the author can not find an exact reference, we still give a proof for the readers' convenience.
\end{rmk}
\begin{proof} $\textrm{We apply}\ -L_{A}\ \textrm{to the equation again to obtain}\
-L^{2}_{A}\xi=-L_{A}f\in C^{\alpha}(2B).$
From the proof of Lemma \ref{lem formula for LA squared}, we see that the difference of the Weitzenb\"ock formula between the model case and the general case is some lower order term (concerning at most first derivative of $\xi$ in local coordinates). Then 
\begin{equation}\label{equ in regularity main lem}
\Delta_{g_{\phi}}\xi=\nabla\xi\otimes T_{\phi,\psi,A,1}+\xi\otimes T_{\phi,\psi,A,0}-L_{A}f,
\end{equation}
where the $T$'s are tensors depending algebraicly on $\phi$, $\psi$, $A$, and their  derivatives. The $T$'s might be only  locally defined, but this is sufficient. 

The important point  is that $\Delta_{g_{\phi}}\xi$ means the metric Laplacian of each entry of $\xi$ in local coordinates, thus (\ref{equ in regularity main lem}) is actually a bunch of scalar equations. 
Then Lemma \ref{lem regularity of solution to the deformation equation with C1alpha rhs}  follows by applying Lemma \ref{lem regularity of the laplace equation} successively.
\end{proof}
\begin{lem}(\cite{GilbargTrudinger})\label{lem regularity of the laplace equation} Under the same conditions on  $B$ in  Lemma \ref{lem regularity of solution to the deformation equation with C1alpha rhs},  suppose $\xi \in W^{1,2}(2B)$ is a weak solution to  
\begin{equation}\label{equ standard laplace equation in lem regularity of laplace equation}
\Delta_{g_{\phi}}\xi=h\ \textrm{in}\ 2B.
\end{equation}
Suppose $h\in L^{p}(2B)$, $p\geq 2$, then $\xi\in L^{2,p}(B)$. Suppose $h\in C^{\alpha}(2B)$, $0<\alpha<1$, then $\xi\in C^{2,\alpha}(B)$.
\end{lem}
\begin{proof} It suffices to construct local solutions (with optimal regularity) to  (\ref{equ standard laplace equation in lem regularity of laplace equation}).  Viewing  (\ref{equ standard laplace equation in lem regularity of laplace equation}) as a bunch of 
 scalar equations,   let the boundary value over $\partial B$ be $0$, when  $h\in L^{p}(2B)$ and  $p\geq 2$, Theorem 9.15 in \cite{GilbargTrudinger} implies that (\ref{equ standard laplace equation in lem regularity of laplace equation}) admits a solution $\bar{\xi}\in L^{2,p}(B)$ (in $B$). When $h\in C^{\alpha}(2B)$, Theorem 6.14 in \cite{GilbargTrudinger}  gives a solution $\bar{\xi}\in C^{2,\alpha}(B)$ to (\ref{equ standard laplace equation in lem regularity of laplace equation}).    In both cases, $\Delta_{g}(\bar{\xi}-\xi)=0$, therefore $\bar{\xi}-\xi$ is smooth by Lemma \ref{lem regularity of kernel and cokernel}. Then $\xi=(\xi-\bar{\xi})+\bar{\xi}\in L^{2,p}(B)$ or $C^{2,\alpha}(B)$, when $h\in L^{p}(B)$ or $C^{\alpha}(B)$, respectively. \end{proof}
\begin{lem}\label{lem regularity of kernel and cokernel}(see Gilkey \cite{Gilkey}) Suppose $f\in L^{2}_{p,b}$ belongs to the cokernel of $L_{A}$ in distribution sense (see (\ref{equ Def of Coker})). Then $f$ is smooth away from the singular points.  $\xi\in kerL_{A}\subset W^{1,2}_{p,b}$ also implies  $\xi$ is  smooth away from the singularities. 
\end{lem}
\begin{proof} It's an easy exercise on pseudo-differential operators. We only need to show the conditions  imply  $L^{\star}_{A}f=0$ where we view $L^{\star}_{A}f$ as an element in $H^{-1}$ (see Lemma 1.2.1 in \cite{Gilkey}).  Thus  Lemma 1.3.1 of \cite{Gilkey} is directly applicable. 

 To achieve this, for any ball $B$ such that $100B$ is still away from the singularities, we choose $\eta$ as the standard cutoff function which vanishes outside $2B$ and is identically $1$ in $B$. We also choose
$\chi$ as the standard cutoff function which vanishes outside $3B$ and is identically $1$ in $2B$. 
By Lemma 1.2.1 and 1.1.6 in \cite{Gilkey}, using a limiting argument with respect the smoothing of  
$\chi f$, we obtain  $\eta L^{\star}_{A} (\chi f)=0$ as an element in $H^{-1}$. Since $L_{A}^{\star}$ is elliptic, we conclude by Lemma 1.3.1 of \cite{Gilkey} that $f$ is smooth in $B$. 
\end{proof}
 \begin{lem}\label{lem bounding local perturbation of deformation operator }Suppose $\tau_{0}\leq \delta$, and $\psi_{1},\psi_{2}$ are two $G_{2}-$structures over $B_{O}(2\tau_{0})$, $|\psi_{1}-\psi_{2}|<\delta.$ Suppose $A_{1},\ A_{2}$ are 2 connections over $B_{O}(\tau_{0})$,  $|A_{1}-A_{2}|<\frac{\delta}{r}$.  
\begin{equation}\label{equ Lem bounding local perturbation of deformation operator }
\textrm{Then}\ |L_{A_{1},\psi_{1}}-L_{A_{2},\psi_{2}}|\leq C\delta|\nabla_{A_{2}}\xi|+\frac{C\delta|\xi|}{r}\ \textrm{in}\ B_{O}(\tau_{0}). \ \ \ \ \ \ \ \ \ \ \ \ \ \ \ \ \ \ \ \ \ \ \
\end{equation}
\end{lem}
\begin{rmk} The $C$ depends on $C^{2}-$norms of $\psi_{1},\psi_{2}$ and $C^{0}-$norm of $rA_{1}$ in local coordinates. The $\nabla_{A_2}$ is the covariant derivative with respect to the Euclidean metric in the coordinate. The estimate (\ref{equ Lem bounding local perturbation of deformation operator }) still holds for $\nabla_{A_{1}}$ and with respect to any smooth metric.
\end{rmk}
\begin{proof} In  (\ref{equ introduction formula for deformation operator}), we only estimate the difference from  $d_{A}^{\star}a$, the  errors from  the other terms 
  are similar.  By $d_{A}^{\star}=-\star d_{A} \star$, we note 
$\star_{\psi_{1}}d_{A_{1}}\star_{\psi_{1}}-\star_{\psi_{2}}d_{A_{2}}\star_{\psi_{2}}
=(\star_{\psi_{1}}-\star_{\psi_{2}})d_{A_{1}}\star_{\psi_{1}}+\star_{\psi_{2}}(d_{A_{1}}-d_{A_{2}})\star_{\psi_{1}}+\star_{\psi_{2}}d_{A_{2}}(\star_{\psi_{1}}-\star_{\psi_{2}}).$
We only estimate the last term $\star_{\psi_{2}}d_{A_{2}}(\star_{\psi_{1}}-\star_{\psi_{2}})$, the estimate of the other terms are similar and easier. Using
$
d_{A_{2}}[(\star_{\psi_{1}}-\star_{\psi_{2}})a]= (\star_{\psi_{1}}-\star_{\psi_{2}})\otimes \nabla_{A_{2}}a+[\nabla (\star_{\psi_{1}}-\star_{\psi_{2}})]\otimes a$, we get $|L_{A_{1},\psi_{1}}-L_{A_{2},\psi_{2}}|\leq C\delta|\nabla_{A_{2}}\xi|+C|\xi|$.
Then we obtain  (\ref{equ Lem bounding local perturbation of deformation operator }) when $r< \tau_{0}$.\end{proof}
\subsection{Appendix D: Density and smooth convergence of   Fourier Series \label{section Appendix D: Density and smooth convergence of   Fourier Series}}
\begin{lem}\label{Lemma Appendix D} Let $S$ be a  closed Riemannian manifold (of any dimension), and $\Xi\rightarrow S$ be a smooth $SO(m)-$vector bundle with an inner product. Suppose $A_{S}$ is a smooth connection on $\Xi$, and  $\Delta_{A_{S}}=\nabla^{\star}_{A_{S}}\nabla_{A_{S}}+\mathfrak{F}$ is a \textbf{self-adjoint} Laplacian-type operator acting on sections of $\Xi$, where $\nabla^{\star}_{A_{S}}\nabla_{A_{S}}$ is the rough Laplacian of $A_{S}$ and the Riemannian metric, $\mathfrak{F}$ is a smooth algebraic operator (which does not concern any covariant derivative). 
Let $\beta$ be the real eigenvalues of $\Delta_{A_{S}}$ repeated according to their multiplicities, and  $\Psi_{\beta}$ be the corresponding orthonormal basis in $L^{2}_{\Xi}(S)$. Then for any smooth section $\underline{f}$ to $\Xi$, the Fourier-series $\Sigma_{\beta}\underline{f}_{\beta}\Psi_{\beta}$ (of $\underline{f}$) converges to $\underline{f}$ in the $C^{\infty}-$topology.  

   Moreover, the speed of convergence only depends on the $C^{\infty}-$norm of $\underline{f}$ i.e.  
there exists a integer $\tau>0$ depending only on $\Delta_{A_{S}}$, such that for any $\epsilon>0$ and integer $s\geq 0$, there exists a $k$ depending only on  $|\underline{f}|_{W^{2\tau+2s,2}(S)}$, $\epsilon$, and $\Delta_{A_{S}}$, such that  $|\underline{f}-\Sigma_{\beta< k}\underline{f}_{\beta}\Psi_{\beta}|_{W^{2s,2}(S)}<\epsilon$.  
\end{lem}
\begin{cor}\label{cor Appendix D:} In the setting of Lemma \ref{lem density of smooth functions in weighted L2 space} and Theorem \ref{thm W22 estimate on 1-forms}, let $\underline{f}\in C^{\infty}_{c}[B_{O}(\rho)\setminus O]$, then the Fourier-series in (\ref{equ raw ODE for 1 forms}) converges in $C^{0}[B_{O}(\rho)]$ to $\underline{f}$. 
\end{cor}
\begin{proof}[Proof of Lemma \ref{Lemma Appendix D}:]  Since $\Delta_{A_{S}}$ is bounded from below, by considering $\Delta_{A_{S}}+a_{0}I$ for some big enough $a_{0}$, we can assume all the $\beta$'s are larger than $1000$. 

 In Claim \ref{clm by Zeta function}, we note that $\Delta_{A_{S}}^{l}(\Sigma_{\beta<k_{l}}\underline{f}_{\beta}\Psi_{\beta})$ is the  Fourier partial sum of $\Delta_{A_{S}}^{l}\underline{f}$. Let $F_{k}$ denote $\Sigma_{\beta\geq k}\underline{f}_{\beta}\Psi_{\beta}$, and  $k=100+\sup_{0\leq l\leq s}k_{l}$,  the standard $W^{2,2}-$estimate for $\Delta_{A_{S}}$ and Claim \ref{clm by Zeta function} imply
\begin{equation}\label{equ Cor appendix D 1}
|\Delta^{s-1}_{A_{S}}F_{k,m}|_{W^{2,2}(S)}\leq \bar{C}|\Delta^{s}_{A_{S}} F_{k,m}|_{L^{2}(S)}+\bar{C}|\Delta^{s-1}_{A_{S}}F_{k,m}|_{L^{2}(S)}\leq \bar{C}\epsilon
\end{equation}
uniformly in $m$. Let $m\rightarrow \infty$, we find $\Delta^{s}F_{k}\in W^{2,2}(S)$ and 
\begin{equation}\label{equ Cor appendix D 2}
|\Delta^{s-1}_{A_{S}}F_{k}|_{W^{2,2}(S)}\leq \bar{C}|\Delta^{s}_{A_{S}} F_{k}|_{L^{2}(S)}+\bar{C}|\Delta^{s-1}_{A_{S}} F_{k}|_{L^{2}(S)}\leq \bar{C}\epsilon.
\end{equation}
By induction, using Theorem 5.2 in \cite{Lawson}, by similar estimates as (\ref{equ Cor appendix D 1}) and  (\ref{equ Cor appendix D 2}), we obtain $|F_{k}|_{W^{2s,2}(S)}\leq \bar{C}_{s}\epsilon.$ Replacing $\bar{C}_{s}\epsilon$ by $\epsilon$,  the proof of Lemma \ref{Lemma Appendix D} is complete  by assuming the following. 
 \begin{clm}\label{clm by Zeta function}  For any $\epsilon>0$,  integer $l\geq 0$,  there exists a $k_{l}$ depending only on $|\underline{f}|_{W^{2l+2\tau,2}(S)}$, $\epsilon$, $l$, $\Delta_{A_{S}}$,  such that $|\Delta_{A_{S}}^{l}\underline{f}-\Delta_{A_{S}}^{l}(\Sigma_{\beta<k_{l}}\underline{f}_{\beta}\Psi_{\beta})|_{L^{2}(S^{6})}<\epsilon$.
\end{clm} The proof of the Claim is by the asymptotic property of zeta-functions. For any positive integer $t$, using 
\begin{equation}\underline{f}_{\beta}=\int_{S}<\underline{f},\Psi_{\beta}>=\frac{\int_{S}<\underline{f},\Delta^{t}_{A_{S}}\Psi_{\beta}>}{\beta^{t}}=\frac{\int_{S}<\Delta^{t}_{A_{S}}\underline{f},\Psi_{\beta}>}{\beta^{t}},
\end{equation}
we get $|\underline{f}_{\beta}|<\bar{C}\frac{|\underline{f}|_{W^{2t,2}(S)}}{\beta^{t}}.$ Then $
|\Delta_{A_{S}}^{l}F_{k_{l}}|_{L^{2}(S^{6})}\leq \bar{C}|\underline{f}|_{W^{2t,2}(S)}\Sigma_{\beta\geq k_{l}}\frac{1}{\beta^{t-l}}.$
The sum $\Sigma_{\beta\geq k_{l}}\frac{1}{\beta^{t-l}}$ is part of the zeta-function of $\Delta_{A_{S}}$. There exists a large enough $\tau$ with respect to the data in Lemma \ref{Lemma Appendix D}, such that $\Sigma_{\beta}\frac{1}{\beta^{t-l}}$  converges to an analytic function of   $t-l\geq \tau$. By Corollary 2.43 in \cite{Getzler}, or Lemma 1.10.1 in \cite{Gilkey}, we can take $\tau=\frac{dim S}{2}+2$. Nevertheless,  we don't need $\tau$ to be explicit. 

Let   $t=\tau+l$, and $k_{l}$ be large enough with respect to $\epsilon$ and the zeta function of $\Delta_{A_{S}}$, the proof  of Claim \ref{clm by Zeta function} is complete.  
  \end{proof}
    \begin{proof}[Proof of Corollary \ref{cor Appendix D:}:] The condition $\underline{f}\in C^{\infty}_{c}[B_{O}(\rho)\setminus O]$ implies that,  by viewing $\underline{f}$ as an $r-$dependent smooth section, 
   the Sobolev norms of $\underline{f}$  are uniformly bounded in $r$ i.e $|\underline{f}(r,\cdot)|_{W^{2t,2}(S^{n-1})}\leq C_{f,t}$. Moreover, $\underline{f}$ (and its Fourier-coefficients) vanishes when $r$ is small enough.  Then for any $\epsilon$, let $s=\frac{n-1}{4}+10$, by Sobolev imbedding,  there exists a $k$ as in Lemma \ref{Lemma Appendix D} such that  the estimate
   $$|\underline{f}(r,\cdot)-\Sigma_{\beta< k}\underline{f}_{\beta}(r)\Psi_{\beta}|_{C^{0}(S^{n-1})}\leq \bar{C}_{\underline{f}}|\underline{f}(r,\cdot)-\Sigma_{\beta< k}\underline{f}_{\beta}(r)\Psi_{\beta}|_{W^{2s,2}(S^{n-1})}<\epsilon$$  
   holds uniformly in $r$. The proof of Corollary \ref{cor Appendix D:} is complete.
  \end{proof}
   \begin{proof}[\textbf{Proof of Lemma \ref{lem density of smooth functions in weighted L2 space}}:] Without loss of generality, we assume $\rho=1$. Dropping the last condition in Definition \ref{Def of L22 norm model case}, we first show that   $C^{\infty}_{c}[B_{O}(1)\setminus O]$ is dense in $L^{2}_{p,b} [B_{O}(1)]$.
We assume $\underline{f}$ satisfies the condition after the "which" in  Lemma \ref{lem density of smooth functions in weighted L2 space}. Under Local coordinate, $\underline{f}$ is a matrix-valued function.  For any $\epsilon>0$, by absolute continuity of Lebesgue integration  (Theorem 4.12 in \cite{Zhouminqiang}), for any small enough $h>0$, we can decompose $\underline{f}=\underline{f}_{0}+\underline{f}_{1}$. $\underline{f}_{0}$ is supported in $V_{+,O}\setminus V_{+,h}$ and  $
\int_{ V_{+,O}\setminus V_{+,h}}|\underline{f}_{0}|^{2}wdV<(\frac{\epsilon}{2})^{2},\  \underline{f}_{1}\ \textrm{is supported in}\ V_{+,h},$ where  $V_{+,h}$ is the set of points with distance to $\partial V_{+,O}$ greater than $h$.

Since $\underline{f}_{1}$ is supported away from the singular point, using   Lemma 7.2 in \cite{GilbargTrudinger},  we can find $\bar{\underline{f}}_{+}$ such that $|\bar{\underline{f}}_{+}-\underline{f}_{1}|_{L^{2}_{p,b}(V_{+,O})}<\frac{\epsilon}{2}, supp \bar{\underline{f}}_{+} \subset\  V_{+,\frac{h}{2}}.$
Then $|\bar{\underline{f}}_{+}-\underline{f}|_{L^{2}_{p,b}(V_{+,O})}\leq |\bar{\underline{f}}_{+}-\underline{f}_{1}|_{L^{2}_{p,b}(V_{+,O})}+|\underline{f}_{0}|_{L^{2}_{p,b}(V_{+,O})}<\epsilon$. $\bar{\underline{f}}_{+}$ is exactly the desired approximation. Let $\bar{\underline{f}}_{-}$ be the same  approximation in $V_{-,O}$. We denote the partition of unity over $S^{n-1}$ subordinate to $U_{+}, U_{-}$ as $\eta_{+},\eta_{-}$, and pull them back to $\R^{7}\setminus O$. Let $\underline{f}_{\epsilon}=\eta_{+}\bar{\underline{f}}_{+}+\eta_{-}\bar{\underline{f}}_{-}$, we obtain 
 \begin{equation}\label{equ Density in L2 1}
 |\underline{f}_{\epsilon}-\underline{f}|_{L^{2}_{p,b}[B_{O}(1)]}\leq |\eta_{+}\underline{f}-\eta_{+}\bar{\underline{f}}_{+}|_{L^{2}_{p,b}[V_{+,O}]}+|\eta_{-}\underline{f}-\eta_{-}\bar{\underline{f}}_{-}|_{L^{2}_{p,b}[V_{-,O}]}<2\epsilon.
 \end{equation}
 
   Viewing $\underline{f}_{\epsilon}$ as a $r-$dependent smooth section of  $\Xi\rightarrow S^{6}(1)$,  Corollary \ref{cor Appendix D:} implies that the series  $\Sigma_{v}\underline{f}_{\epsilon,v}(r)\Psi_{v}$ (see (\ref{equ raw ODE for 1 forms})) converges to $\underline{f}_{\epsilon}$ in $C^{0}[B_{O}(1)]$, 
and the following holds for some  large enough $v_{0}>0$. 
 \begin{equation}\label{equ Density in L2 2}
  \int_{B_{O}(1)}|\underline{f}_{\epsilon}-\underline{f}_{[v_{0}],\epsilon}|^{2}wdV<\epsilon^{2},\ \textrm{where}\ \underline{f}_{[v_{0}],\epsilon}\triangleq \Sigma_{v< v_{0}}\underline{f}_{\epsilon,v}(r)\Psi_{v}.
 \end{equation}
 Inequalities (\ref{equ Density in L2 1}) and (\ref{equ Density in L2 2}) imply $|\underline{f}_{[v_{0}],\epsilon}-\underline{f}|_{L^{2}_{p,b}[B_{O}(1)]}<3\epsilon$.\end{proof} 
The following proof does not depend on Corollary \ref{cor Appendix D:}. 
\begin{proof}[\textbf{Proof of Lemma \ref{lem H=W}}:]  We only consider the case  $k=1$, and assume $\rho=1$.  The assertion $W^{1,2}_{p,b}[B_{O}(1)]\subset \mathfrak{W}^{1,2}_{p,b}[B_{O}(1)]$ is an easy exercise  using 
monotone convergence theorem and Theorem 7.4 in \cite{GilbargTrudinger} away from the singularity. 

     The assertion  $ \mathfrak{W}^{1,2}_{p,b}[B_{O}(1)]\subset W^{1,2}_{p,b}[B_{O}(1)]$ means every section $\xi\in \mathfrak{W}^{1,2}_{p,b}[B_{O}(1)]$  can be approximated by smooth sections defined in $B_{O}(1)$ (away from $O$) in $W^{1,2}_{p,b}[B_{O}(1)]$-topology. It suffices to show every $\xi\in \mathfrak{W}^{1,2}_{p,b}(V_{+,O})$ (in local coordinate) can be approximated by smooth multi-matrix-valued functions defined in $V_{+,O}$. This job is done by using the proof of Theorem 7.9 in \cite{GilbargTrudinger} with the "$\frac{\epsilon}{2^{j}}$" (in (7.25) there) replaced by $\frac{\epsilon}{2^{j}\tau_{j}}$, where $\tau_{j}=1+\sup_{\Omega_{j}}\frac{\omega}{r^{2}}$, and $\Omega_{j}$ is the corresponding open set in  a natural cover of $V_{+,O}$.\end{proof}
  \subsection{Appendix E: Various integral identities and proof of Proposition \ref{prop bounding L2 norm of Hessian for the model cone laplace equation} \label{section Appendix E: Various integral identities and proof}}
  \begin{lem}\label{lem rough L2 identity integration by parts }Under the conditions as in Proposition \ref{prop existence of solution with lowest order estimate} and  \ref{prop final ODE W22 estimate when v=0}, let $u$ be as in (\ref{equ solution when v=0}), let $\bar{u}$ and $\bar{f}$ be as in Claim \ref{clm weight change on the ODE}  and (\ref{weighted ODE}), we have
\begin{eqnarray}\label{equ rough L2 identity integration by parts}
& &\int^{\frac{1}{2}}_{0}\bar{f}^{2}rw_{0} dr
\\&=& \int^{\frac{1}{2}}_{0}|\frac{d^{2} \bar{u}}{d r^{2}}|^{2}rw_{0}dr+(k^{2}+2a^{2})\int^{\frac{1}{2}}_{0}|\frac{d \bar{u}}{d r}|^{2}\frac{w_{0}}{r} dr\nonumber
+ (a^{4}-2ka^{2}-2a^{2})\int^{\frac{1}{2}}_{0}\frac{ \bar{u}^{2} w_{0}}{ r^{3}} dr\\& &-k\int^{\frac{1}{2}}_{0}|\frac{d \bar{u}}{d r}|^{2}\frac{d w_{0}}{d r} dr
+(2+k)a^{2}\int^{\frac{1}{2}}_{0}\frac{ \bar{u}^{2} }{ r^{2}} \frac{d w_{0}}{d r}dr-a^{2}\int^{\frac{1}{2}}_{0}\frac{ \bar{u}^{2} }{ r}\frac{d^{2} w_{0}}{d r^{2}} dr\nonumber
\\& &+ k |\frac{d \bar{u}}{d r}|^{2}w_{0}|^{\frac{1}{2}}_{0}-  (k+1)a^{2} \frac{ \bar{u}^{2}}{ r^{2}}w_{0}|^{\frac{1}{2}}_{0}-2a^{2} \frac{d \bar{u}}{d r}\frac{ \bar{u}}{ r}w_{0}|^{\frac{1}{2}}_{0}+a^{2}\frac{ \bar{u}^{2}}{ r}\frac{d w_{0}}{d r}|^{\frac{1}{2}}_{0}\nonumber
\end{eqnarray}
\end{lem}
\begin{proof}[Proof of Lemma \ref{lem rough L2 identity integration by parts }:]
Integrating the square of both hand sides of (\ref{weighted ODE}) over $(0,\frac{1}{2})$ with respect to $ rw_{0}$, we  obtain  
\begin{eqnarray}\label{eqnarray simply integrate the square of both hand sides}
& &\int^{\frac{1}{2}}_{0}\bar{f}^{2}rw_{0} dr
\\&=& \int^{\frac{1}{2}}_{0}|\frac{d^{2} \bar{u}}{d r^{2}}|^{2}rw_{0}dr+k^{2}\int^{\frac{1}{2}}_{0}|\frac{d \bar{u}}{d r}|^{2}\frac{w_{0}}{r} dr\nonumber
+ a^{4}\int^{\frac{1}{2}}_{0}\frac{ \bar{u}^{2} w_{0}}{ r^{3}} dr\\& &+2k\int^{\frac{1}{2}}_{0}\frac{d^{2} \bar{u}}{d r^{2}}\frac{d \bar{u}}{d r}w_{0} dr
-2a^{2}\int^{\frac{1}{2}}_{0}\frac{d^{2} \bar{u}}{d r^{2}}\frac{\bar{u}}{ r}w_{0}dr-2ka^{2}\int^{\frac{1}{2}}_{0}\frac{d\bar{u}}{d r}\frac{\bar{u}}{ r^{2}}w_{0}dr\nonumber.
\end{eqnarray}
Integration by parts (of $\frac{d}{d r}$) gives 
\begin{eqnarray}
& &2k\int^{\frac{1}{2}}_{0}\frac{d^{2} \bar{u}}{d r^{2}}\frac{d \bar{u}}{d r}w_{0} dr=-k\int^{\frac{1}{2}}_{0}|\frac{d \bar{u}}{d r}|^{2}\frac{d w_{0}}{d r}dr
+k|\frac{d \bar{u}}{d r}|^{2}w_{0}\mid^{\frac{1}{2}}_{0}
\\& &-2ka^{2}\int^{\frac{1}{2}}_{0}\frac{d\bar{u}}{d r}\frac{\bar{u}}{ r^{2}}w_{0}dr
=-2ka^{2}\int^{\frac{1}{2}}_{0}\frac{\bar{u}^{2}w_{0}}{ r^{3}}dr+ka^{2}\int^{\frac{1}{2}}_{0}\frac{\bar{u}^{2}}{ r^{2}}\frac{d w_{0}}{d r}dr-ka^{2}\frac{\bar{u}^{2}w_{0}}{r^{2}}\mid^{\frac{1}{2}}_{0}.\nonumber
\end{eqnarray}
\begin{eqnarray}& &-2a^{2}\int^{\frac{1}{2}}_{0}\frac{d^{2}\bar{u}}{d r^{2}}\frac{\bar{u}}{ r}w_{0}dr
\\&=&-2a^{2}\frac{d u}{d r}\frac{u}{r}w_{0}\mid^{\frac{1}{2}}_{0} -2a^{2}\int^{\frac{1}{2}}_{0}\frac{\bar{u}^{2}w_{0}}{ r^{3}}dr+2a^{2}\int^{\frac{1}{2}}_{0}\frac{\bar{u}^{2}}{ r^{2}}\frac{d w_{0}}{d r}dr-a^{2}\frac{\bar{u}^{2}w_{0}}{r^{2}}\mid^{\frac{1}{2}}_{0}.\nonumber
\\& &+a^{2}\frac{\bar{u}^{2}}{r}\frac{d w_{0}}{d r}\mid^{\frac{1}{2}}_{0}+2a^{2}\int^{\frac{1}{2}}_{0}|\frac{d u}{d r}|^{2}\frac{w_{0}}{r}dr-a^{2}\int^{\frac{1}{2}}_{0}\frac{ u^{2}}{ r}\frac{d^{2}w_{0}}{d r^{2}}dr.\nonumber
\end{eqnarray}
Plugging the above in (\ref{eqnarray simply integrate the square of both hand sides}), the proof of Lemma \ref{lem rough L2 identity integration by parts } is complete. \end{proof}
\begin{lem}\label{lem log estimate} Let $b\geq 0$. For any $k$, there exists a constant $C_{k,b}$ which depends only on the positive lower bound of $|k-1|$ (not on the upper bound), with the following properties. Suppose $k<1$, then
\begin{equation}\label{equ the log integral 1}
\int^{r}_{0}x^{-k}(-\log x)^{2b}dx\leq \frac{C_{k,b}}{1-k}r^{1-k}(-\log r)^{2b},\ \textrm{for all}\ r\in\ [0,\frac{1}{2}].
\end{equation}
Suppose $k>1$, then 
\begin{equation}\label{equ the log integral 2}
\int^{\frac{1}{2}}_{r}x^{-k}(-\log x)^{2b}dx\leq \frac{C_{k,b}}{k-1}r^{1-k}(-\log r)^{2b},\ \textrm{for all}\ r\in\ [0,\frac{1}{2}].
\end{equation}
\end{lem}
\begin{proof}This is an absolutely easy practice in calculus. We only prove (\ref{equ the log integral 1}) assuming $b<\frac{1}{2}$. The general case and proof of (\ref{equ the log integral 2}) are similar except that we have more terms of the same nature.  We compute 
\begin{equation}\label{eqnarray log integral with induction}\int^{r}_{0}x^{-k}(-\log x)^{2b}dx
= \frac{x^{1-k}}{1-k}(-\log x)^{2b}|^{r}_{0}+\frac{2b}{1-k}\int^{r}_{0}x^{-k}(-\log x)^{2b-1}dx.
\end{equation}
Since $2b-1<0$, we have $\int^{r}_{0}x^{-k}(-\log x)^{2b-1}dx\leq C\int^{r}_{0}x^{-k}dx= \frac{Cr^{1-k}}{1-k}$. Thus the right hand side of (\ref{eqnarray log integral with induction}) is bounded by $\frac{C_{k}}{1-k}r^{1-k}(-\log r)^{2b}$.\end{proof}

  \begin{proof}[\textbf{Proof of Proposition \ref{prop bounding L2 norm of Hessian for the model cone laplace equation}}]  Let $\eta_{\epsilon}$ be as in (\ref{equ cut-off function bound near the singular point}), let $d_{j}$ denote $\nabla_{A_{O}, \frac{\partial}{\partial x_{j}}}$ (under Euclidean metric and coordinate), we compute for any $\varrho>0$ that 
\begin{eqnarray}\label{eqnarray L22 identity from  integration by parts}& &\int_{B_{O}(\frac{\varrho}{4})}|\nabla^{\star}_{A_{O}}\nabla_{A_{O}}\xi|^{2}\eta_{\epsilon}\chi^{2}wdV=\Sigma_{j,i}\int_{B_{O}(\frac{\varrho}{4})} <d_{j}d_{j}\xi,d_{i}d_{i}\xi>\eta_{\epsilon}\chi^{2}w dV
\\&=& -\Sigma_{j,i}\int_{B_{O}(\frac{\varrho}{4})} <d_{j}\xi,d_{j}d_{i}d_{i}\xi>\eta_{\epsilon}\chi^{2}w dV-\Sigma_{j,i}\int_{B_{O}(\frac{\varrho}{4})} <d_{j}\xi,d_{i}d_{i}\xi>[\nabla_{j}(\eta_{\epsilon}\chi^{2}w )]dV\nonumber 
\\&=& -\Sigma_{j,i}\int_{B_{O}(\frac{\varrho}{4})} <d_{j}\xi,d_{i}d_{j}d_{i}\xi>\eta_{\epsilon}\chi^{2}w dV+\Sigma_{j,i}\int_{B_{O}(\frac{\varrho}{4})} <d_{j}\xi,[F_{ij},d_{i}\xi]>\eta_{\epsilon}\chi^{2}w dV \nonumber 
\\& &-\Sigma_{j,i}\int_{B_{O}(\frac{\varrho}{4})} <d_{j}\xi,d_{i}d_{i}\xi>[\nabla_{j}(\eta_{\epsilon}\chi^{2}w )]dV \nonumber\end{eqnarray} \begin{eqnarray}
&=& \Sigma_{j,i}\int_{B_{O}(\frac{\varrho}{4})} <d_{i}d_{j}\xi,d_{j}d_{i}\xi>\eta_{\epsilon}\chi^{2}w dV
+\Sigma_{j,i}\int_{B_{O}(\frac{\varrho}{4})} <d_{j}\xi,d_{j}d_{i}\xi>[\nabla_{i}(\eta_{\epsilon}\chi^{2}w )]dV\nonumber
\\& &+\Sigma_{j,i}\int_{B_{O}(\frac{\varrho}{4})} <d_{j}\xi,[F_{ij},d_{i}\xi]>\eta_{\epsilon}\chi^{2}w dV \nonumber
-\Sigma_{j,i}\int_{B_{O}(\frac{\varrho}{4})} <d_{j}\xi,d_{i}d_{i}\xi>[\nabla_{j}(\eta_{\epsilon}\chi^{2}w )]dV
\\&=& \Sigma_{j,i}\int_{B_{O}(\frac{\varrho}{4})} <d_{j}d_{i}\xi,d_{j}d_{i}\xi>\eta_{\epsilon}\chi^{2}w dV
+\Sigma_{j,i}\int_{B_{O}(\frac{\varrho}{4})} <[F_{ij},\xi],d_{j}d_{i}\xi>\eta_{\epsilon}\chi^{2}w dV \nonumber
\\& &+\Sigma_{j,i}\int_{B_{O}(\frac{\varrho}{4})} <d_{j}\xi,d_{j}d_{i}\xi>[\nabla_{i}(\eta_{\epsilon}\chi^{2}w )]dV+\Sigma_{j,i}\int_{B_{O}(\frac{\varrho}{4})} <d_{j}\xi,[F_{ij},d_{i}\xi]>\eta_{\epsilon}\chi^{2}w dV\nonumber
\\& &- \Sigma_{j,i}\int_{B_{O}(\frac{\varrho}{4})} <d_{j}\xi,d_{i}d_{i}\xi>[\nabla_{j}(\eta_{\epsilon}\chi^{2}w )]dV. \nonumber
\end{eqnarray}
Then we distribute all derivatives like $\nabla (\eta_{\epsilon}\chi^{2}w)$. By the method in (\ref{equ integration by parts holds true in the case of L12 model estimate wrt to cone}), Lemma \ref{lem bound on C3 norm of solution to laplace equation when f is smooth and vainishes near O}, and the proof of (\ref{equ bounding L2 norm of gradient for the model cone laplace equation}),  all the integrals containing $\nabla \eta_{\epsilon}$ tend to $0$ as $\epsilon\rightarrow 0$, thus the equality between  top and bottom of  (\ref{eqnarray L22 identity from  integration by parts}) gives. 
  \begin{eqnarray}\label{eqnarray prop hessian L2 bound model case}& & \int_{B_{O}(\frac{\varrho}{4})}|\nabla^{2}_{A_{O}}\xi|^{2}\chi^{2}wdV
 \\&=& \int_{B_{O}(\frac{\varrho}{4})}|\nabla^{\star}_{A_{O}}\nabla_{A_{O}}\xi|^{2}\chi^{2}wdV-\Sigma_{j,i}\int_{B_{O}(\frac{\varrho}{4})} <[F_{ij},\xi],d_{j}d_{i}\xi>\chi^{2}w dV \nonumber
\\& &-\Sigma_{j,i}\int_{B_{O}(\frac{\varrho}{4})} <d_{j}\xi,d_{j}d_{i}\xi>[\nabla_{i}(\chi^{2}w )]dV-\Sigma_{j,i}\int_{B_{O}(\frac{\varrho}{4})} <d_{j}\xi,[F_{ij},d_{i}\xi]>\chi^{2}w dV\nonumber
\\& &+\Sigma_{j,i}\int_{B_{O}(\frac{\varrho}{4})} <d_{j}\xi,d_{i}d_{i}\xi>[\nabla_{j}(\chi^{2}w )]dV. \nonumber
  \end{eqnarray}
  
 Using (\ref{equ gradient of the cutoff function and weight}), (\ref{equ in L2 bound on the gradient consequence of Bochner formula}), and (\ref{eqnarray prop hessian L2 bound model case}),  by the proof of (\ref{eqnarray in model L12 bound 1}) and  (\ref{eqnarray the L12 estimate model case with a small error term on the right to be absorbed}), we deduce 
  \begin{eqnarray}\label{eqnarray prop hessian L2 bound 2}& &  \int_{B_{O}(\frac{\varrho}{4})}|\nabla^{2}_{A_{O}}\xi|^{2}\chi^{2}wdV
   \\& \leq & \bar{C}_{\vartheta}\int_{B_{O}(\frac{\varrho}{4})}\frac{|\nabla_{A_{O}}\xi|^{2}}{r^{2}}\chi^{2}wdV+\bar{C}_{\vartheta}\int_{B_{O}(\frac{\varrho}{4})}|\nabla\chi|^{2}|\nabla_{A_{O}}\xi|^{2}wdV\nonumber
   \\& &+\bar{C}_{\vartheta} \int_{B_{O}(\frac{\varrho}{4})}\frac{|\xi|^{2}}{r^{4}}\chi^{2}wdV+\bar{C}\int_{B_{O}(\frac{\varrho}{4})}|\underline{f}|^{2} \chi^{2}wdV.
   +\vartheta  \int_{B_{O}(\frac{\varrho}{4})}|\nabla_{A_{O}}^{2}\xi|^{2}\chi^{2}wdV.\nonumber
 \end{eqnarray}
 
 Let  $\chi$ be the standard cutoff function which is identically $1$ over $B_{O}(\frac{\varrho}{5})$ and vanishes outside $B_{O}(\frac{\varrho}{4.5})$, the  proof of (\ref{equ in model L12 bound 1}) implies 
 \begin{equation}\label{eqnarray prop hessian L2 bound 3}\bar{C}_{\vartheta}\int_{B_{O}(\frac{\varrho}{4})}|\nabla\chi|^{2}|\nabla_{A_{O}}\xi|^{2}wdV\leq  \bar{C}_{\vartheta}\int_{B_{O}(\frac{\varrho}{4.5})}\frac{|\nabla_{A_{O}}\xi|^{2}}{r^{2}}wdV. \end{equation} 
 
  Let $\vartheta=\frac{1}{10}$,  combining (\ref{eqnarray prop hessian L2 bound 2}) and (\ref{eqnarray prop hessian L2 bound 3}),  the proof  is complete. \end{proof}
\newpage
\small
\section{Notation and Subject Index} 
 The locations in the column of "definition" includes the nearby material.
\begin{center}\begin{tabular}{|p{6cm}|p{6cm}|}
  \hline
 Subject or Notation & definition   \\   \hline
 $A-$generic  & Def \ref{Def A generic} \\ \hline
 admissible connections  & Def \ref{Def Admissable connection with polynomial or exponential convergence}\\ \hline
 $H_{p,b}$, $H_{p}$, $N_{p,b}$, $N_{p}$   & Def \ref{Def Hybrid spaces}, Def \ref{Def abbreviation of notations for spaces} \\ \hline
$C^{k,\alpha}_{\gamma,b}$, $C^{k,\alpha}_{\gamma}$, $|\cdot|^{(\gamma,b)}_{k,\alpha}$,$|\cdot|^{(\gamma)}_{k,\alpha}$  & Def \ref{Def Schauder spaces}, Def \ref{Def Global Schauder norms},  Def \ref{Def abbreviation of notations for spaces}  \\ \hline
$|\cdot|^{[y]}_{k,\alpha,B}$, $|\cdot|^{\star}_{k,\alpha,B}$ & Proof of Theorem \ref{thm C0 est},(\ref{equ norm III}) \\ \hline
$\mathbb{U}_{\tau_{0}}$, $\tau_{0}-$admissible cover, $\mathbb{U}_{\rho_{0}}$  & Def \ref{Def admissable open cover}, Def \ref{Def Global Schauder norms}  \\ \hline
condition ${\circledS_{A,p}}$ & Def \ref{Def condition SAp}  \\ \hline
admissible $\delta_{0}-$deformation of the $G_{2}-$structure, $\phi$, $\psi$, $\phi_{0}$, $\psi_{0}$ & Def \ref{Def deformation of the G2 structure}, (\ref{eqnarray Euc G2 forms}) \\ \hline
$W^{2,2}_{p,b}$, $W^{1,2}_{p,b}$, $W^{1,2}_{p}$,$L^{2}_{p,b}$, $L^{2}_{p}$& Def \ref{Def global weight and Sobolev spaces}, Def \ref{Def of L22 norm model case}, Def \ref{Def abbreviation of notations for spaces} \\ \hline
$\underline{\otimes}$, $\llcorner$, $\lrcorner$, $\otimes$  & Def \ref{Def Two tensor products}, Def \ref{Def tensor product}  \\ \hline
$\bar{C}$  & Def \ref{Def special constants}, Def \ref{Def tensor product}  \\ \hline
$L_{A}$, $L_{A_{O}}$, $L_{A,\underline{\psi}}$, $L_{A}^{\star}$  & (\ref{equ Def of equation for square of model LAO}), Lemma \ref{lem Launderline psi is an isomorphism from Jp to Np}, (\ref{equ introduction formula for deformation operator}), (\ref{equ LA star formula}) \\ \hline
$Q_{A,p,b}$, $Q_{A,p}$ & Corollary \ref{Cor solving model laplacian equation over the ball without the compact support RHS condition}, Def \ref{Def abbreviation of notations for spaces} \\ \hline
$J_{p,b}$, $J_{p}$  & Remark \ref{rmk Jpb}, Def \ref{Def abbreviation of notations for spaces} \\ \hline
$G(\cdot,\cdot)$ & (\ref{equ Def G..})  \\ \hline
$v$, $v-$spectrum, $\beta$, $\Psi_{\beta}$, $\Psi_{v}$  & Def \ref{Def v spectrum}  \\ \hline
$K_{p,b}$  &  Theorem \ref{thm global parametrix}  \\ \hline
$\Xi$  & Def \ref{Def the bundle xi}  \\ \hline
$\perp$, $\parallel$  & Def \ref{Def Fredholm operators and isomorphisms}  \\ \hline
$O$, $O_{j}$, $V_{O,+}$, $V_{O,-}$, $U_{+}$, $U_{-}$  & Def \ref{Def admissable open cover}  \\ \hline 
$\Upsilon_{A_{O}}$, $\Upsilon_{A_{O_{j}}}$  & Proposition \ref{prop seperation of variable for general cone} \\ \hline
$\Delta_{s}$  & (\ref{equ cone formula for the 0form with proper basis})   \\ \hline
$w$, $w_{p,b}$  & Def \ref{Def of L22 norm model case}, Def \ref{Def global weight and Sobolev spaces} \\ \hline
$dV$ & Def \ref{Def volume forms}  \\ \hline
$\vartheta_{-p}$, $\vartheta_{1-p}$ & Def \ref{Def spectrum gap} \\ \hline
$coker L_{A}$ & (\ref{equ Def of Coker})  \\ \hline
$d_{j}$, $d_{i}$  & (\ref{eqnarray L22 identity from  integration by parts})  \\ \hline
$r,\ r_{x},\ r_{x,y},\ \underline{r_{x,y}}$ & Remark \ref{rmk homotopy property and def of r}, Def \ref{Def local Schauder norms}  \\ \hline 
(p,b)-Fredholm  & Def \ref{Def Fredholm operators and isomorphisms}  \\ \hline
 \end{tabular}
  \vspace{10pt}
\renewcommand\arraystretch{1.5}
  \end{center}
 
\end{document}